%% file: main.tex
\documentclass[a4paper, 10pt]{article}
 
\usepackage{packageList}
\usepackage{macros}

\usepackage[normalem]{ulem}			%

\graphicspath{{./Figures/}}
\pgfplotsset{compat=1.11}
\newlength\fwidth

\title{Adaptive meshfree approximation for linear elliptic partial differential equations with PDE-greedy kernel methods}
\author[1,4]{Tizian Wenzel \thanks{wenzel@math.lmu.de}}

\author[2]{Daniel Winkle \thanks{daniel.winkle@mathematik.uni-stuttgart.de}}
\author[3]{Gabriele Santin \thanks{gabriele.santin@unive.it}}
\author[4]{Bernard Haasdonk \thanks{haasdonk@mathematik.uni-stuttgart.de}}

\affil[1]{Department of Mathematics, Ludwig-Maximilians-Universität München, Germany}
\affil[2]{Institute for Stochastics and Applications, University of Stuttgart, Germany}
\affil[3]{Department of Environmental Sciences, Informatics and Statistics, Ca' Foscari University of Venice, Italy}
\affil[4]{Institute of Applied Analysis and Numerical Simulation, University of Stuttgart, Germany}

\newcommand{\Id}{\mathrm{Id}}
\newcommand{\F}{\mathcal{F}}

\newtheorem{assumption}{Assumption}

\newcommand{\hide}[1]{ }

\begin{document}

\maketitle %

\begin{abstract}

  We consider meshless approximation for solutions of
  boundary value problems (BVPs)
  of elliptic Partial Differential Equations (PDEs)
  via symmetric kernel collocation. We discuss the importance of the choice of
  the collocation points, in particular by using greedy kernel methods.
We introduce a scale of PDE-greedy selection criteria that generalizes existing techniques, such as
the PDE-$P$-greedy and the PDE-$f$-greedy rules for collocation point selection.

For these greedy selection criteria 
we provide bounds on the approximation error in terms of the number of greedily selected points and
analyze the corresponding convergence rates.
This is achieved by a novel analysis of Kolmogorov widths of special sets of BVP point-evaluation functionals.
Especially, %
we prove that target-data dependent algorithms that make use of the right hand side functions of
the BVP exhibit faster convergence rates than the target-data independent PDE-$P$-greedy.
The convergence rate of the PDE-$f$-greedy possesses a dimension independent rate, 
which makes it amenable to mitigate the curse of dimensionality.

The advantages of these %
greedy algorithms are highlighted
by numerical examples.

\end{abstract}

\newpage

\input{chapters/01_introduction}

\input{chapters/02_background}

\input{chapters/03_analysis_kolmogorov_rates}

\input{chapters/04_analysis_kernel_greedy}
\input{chapters/05_conv_rates_pde_greedy}

\input{chapters/07_NEW_numerical_experiments}

\input{chapters/08_outlook}
~ \\
\textbf{Acknowledgements:} The authors acknowledge the funding of the project by the Deutsche Forschungsgemeinschaft (DFG, German Research Foundation) under
Germany's Excellence Strategy - EXC 2075 - 390740016 and funding by the BMBF project ML-MORE. 
Tizian Wenzel acknowledges support from the Studienstiftung des deutschen Volkes (German national Academic Foundation) and
Daniel Winkle acknowledges funding by the International Max Planck Research School for Intelligent Systems (IMPRS-IS).
Part of the work was inspired by and done during a research stay of the first author with Prof.\ Suykens at ESAT, KU Leuven.

\noindent \textbf{Conflict of interest statement:} Not applicable.

\bibliography{references}			
\bibliographystyle{abbrv}

\appendix

\include{chapters/app_01_proofs}

\end{document}

%% file: chapters/01_introduction.tex
\section{Introduction} \label{sec:introduction}

Kernel methods \cite{wendland2005scattered,fasshauer2007meshfree,steinwart2008support} are a popular class of techniques in approximation theory, numerical mathematics and machine learning.
They revolve around the notion of a kernel $k$. Given a non-empty set $\Omega$, a real valued kernel is defined as a symmetric function $k: \Omega \times \Omega \rightarrow \R$, and here we focus on domains $\Omega \subset \R^d$.

For given points $X_n := \{x_1, \dots, x_n \} \subset \Omega$, the kernel matrix $K_{X_n} \in \R^{n \times n}$ is defined as $(K_{X_n})_{ij} := (k(x_i, x_j))_{ij}$, $i,j=1,\ldots,n$. 
A kernel is called strictly positive definite (s.p.d.) if the associated kernel matrix is
positive definite for any choice of pairwise distinct points $X_n \subset \Omega, n \in \N$.
Those s.p.d.\ kernels give rise to a unique Hilbert space, the so called \textit{Reproducing Kernel Hilbert Space (RKHS)} denoted as $\ns$.
The RKHS is a space of real-valued functions from $\Omega$ to $\R$ with inner product $\langle \cdot, \cdot \rangle_{\ns}$, 
and it is sometimes also called \textit{native space}. 
Within the RKHS, the kernel $k$ acts as a reproducing kernel, i.e.\ %
(i) $k(\cdot,x) \in \ns$ $\forall x\in\Omega$, and (ii)
$f(x) = \langle f, k(\cdot,x) \rangle_{\ns} \, \forall x \in \Omega, \forall f \in \ns$, which
is called the reproducing property.
Kernel interpolation can be characterized and analyzed in this framework:
For a given closed subspace $V \subset \ns$ we denote the $\ns$-orthogonal projection onto $V$ with $\Pi_V: \ns \to V$.
Given $n$ interpolation points $X_n$, the unique minimum-norm interpolant $s_n$ to a function $f \in \ns$ is equal to the $\ns$-orthogonal projection of $f$ onto the subspace $V_{X_n} := \Sp \{k(\cdot, x_i), x_i \in X_n \} \subset \ns$,
i.e.\ it holds
\begin{align}\label{eq:kernel_interpolant_standard}
s_n(\cdot) := \Pi_{V_{X_n}} (f)(\cdot) = \sum_{j=1}^n \alpha_j k(\cdot , x_j).
\end{align}
The coefficients $\alpha_j$ can be determined by the interpolation conditions 
\begin{align*}
s_n(x_i) &= f(x_i) \quad \forall i=1, \dots, n  \quad
\Leftrightarrow \quad K_{X_n}\alpha = (f(x_j) )_{j=1}^n,
\end{align*}
where $\alpha = (\alpha_1, \dots, \alpha_n)^\top \in \R^n$ is the vector of the coefficients.
A standard way to quantify the error between the function $f$ and the interpolant $s_n$ in the $\Vert \cdot \Vert_{L^\infty(\Omega)}$-norm is using the so called power function $P_{X_n}$, which is given as 
\begin{align}
\begin{aligned}
\label{eq:power_func_standard}
P_{X_n}(x) &:= \Vert k(\cdot,x) - \Pi_{V_{X_n}} (k(\cdot , x )) \Vert_{\ns} \\
&= \sup_{0 \neq f \in \ns} \frac{| f(x) - \Pi_{V_{X_n}} (f) (x) |}{\Vert f \Vert_{\ns}},
\end{aligned}
\end{align} 
and thus allows to estimate the residual $f - s_n$ as
\begin{align}
\label{eq:standard_power_func_bound}
| (f-s_n)(x) | \leq P_{X_n}(x) \cdot \Vert f \Vert_{\ns}. 
\end{align}
In order to obtain a good approximation, a suitable choice of the points $X_n \subset \Omega$ is required, even if it is usually unclear a priori.
A viable option is the use of greedy algorithms \cite{Temlyakov2008}, which start with an empty set $X_0 := \emptyset$ and then add further points step by step as $X_{n+1} := X_n \cup \{ x_{n+1} \}$.
The next interpolation point $x_{n+1}$ is usually chosen according to some optimality criterion, and for
greedy kernel interpolation common choices are the so called $P$-greedy \cite{marchi2005optimal}, $f \cdot P$-greedy (or \emph{psr}-greedy) \cite{Dutta2020},
$f$-greedy \cite{SchWen2000}, and $f/P$-greedy \cite{mueller2009komplexitaet} criteria.

A unifying analysis of those standard greedy kernel interpolation algorithms was recently provided in \cite{wenzel2022analysis} under the notion of \textit{$\beta$-greedy algorithms}, that we will generalize in the current article from function interpolation to PDE collocation. 
Among the greedy kernel interpolation algorithms, the $P$-greedy scheme does
not use the target values of the function approximation problem for
selecting the points. Therefore, this variant is called {\em target-data independent}. In contrast, the other mentioned
selection criteria make use of the target values, hence adapt to the target function, and
are denoted as {\em target-data dependent}.
This terminology will be adopted to the PDE approximation schemes under consideration, where the ``target-data'' will be
the point evaluations of the right hand side of the PDE problem.

The theory of kernel interpolation is not limited to the interpolation of function values. Indeed it is possible to consider arbitrary functionals $\Lambda_n := \{ \lambda_1, \dots, \lambda_n \} \subset \ns'$,
where $\ns'$ denotes the dual space of $\ns$,
and obtain a function $s_n$ of minimum norm that satisfies
\begin{align}\label{eq:gen_interp_cond}
\lambda_i(s_n) = \lambda_i(f) \quad \forall i=1, \dots, n.
\end{align}

This procedure is called \textit{generalized interpolation} \cite[Chapter 16]{wendland2005scattered}.
It is possible to show that, similarly to Eq.\ \eqref{eq:kernel_interpolant_standard}, it holds
\begin{align}
\label{eq:kernel_interpolant_generalized}
s_n(\cdot) = \Pi_{V_n}(f)(\cdot) = \sum_{j=1}^n \alpha_j v_{\lambda_j}(\cdot), %
\end{align}
where we now have $V_n := \Sp \{ v_{\lambda_i} ~ | ~ \lambda_i \in \Lambda_n \}$ and $v_{\lambda_1}, \dots, v_{\lambda_n} \in \ns$ are the unique Riesz
representers of the functionals $\lambda_1, \dots, \lambda_n$.
Observe that for the special case of standard kernel interpolation, due to the reproducing property, the Riesz representer of the point evaluation functionals 
$\delta_{x_j}$ are simply given by $k(\cdot, x_j)$, such that the generalized interpolant from Eq.\ \eqref{eq:kernel_interpolant_generalized} again reduces to 
Eq.\ \eqref{eq:kernel_interpolant_standard}, and Eq.\ \eqref{eq:kernel_interpolant_generalized} is indeed a generalized form of interpolation. 

This framework of generalized interpolation can especially be applied to approximating solutions of
linear PDEs \cite{renardy2006introduction,lions2012boundary2}. 
For this we consider linear elliptic boundary value problems (BVPs)
\begin{align}
\begin{aligned}
\label{eq:pde}
Lu &= f \quad \text{on } \Omega \subset \R^d \\
Bu &= g \quad \text{on } \partial \Omega,
\end{aligned}
\end{align}
with an elliptic differential operator $L$ of second order, for example the
Laplace operator $L=-\Delta$, and a suitable boundary operator $B$.
In order to keep the focus on symmetric kernel collocation and
the distribution of the collocation points, 
we do not aim for the most general case of BVPs,
but rather restrict to problems with Dirichlet boundary conditions,
hence we assume $B=\Id$ throughout the remaining presentation.
We remark that the results can also be extended to e.g.\ Neumann boundary conditions.

We propose and analyze methods that are based on collocation \cite{franke1998solving,franke1998convergence,fornberg2015solving,chen2016reduced}, hence we focus on cases
where we have well-posedness of the BVP in this strong formulation,
i.e.\ the solution is assumed to be classically differentiable of
sufficient order with corresponding continuous derivatives.
We comment on more general cases in the outlook.  

The constraints imposed by the PDE and the boundary values can be collected in
two sets of functionals, i.e.\ by introducing
\begin{equation}
    \left.
    \begin{aligned}
        \Lambda_L &:= \{ \delta_x \circ L : x \in \Omega \} \\
		\Lambda_B &:= \{ \delta_x : x \in \partial \Omega \} \quad \qquad
    \end{aligned}
    \right\}
    \quad \Lambda := \Lambda_L \cup \Lambda_B. \label{eq:lambda_sets}
\end{equation}

In this paper, we address the problem of designing and analyzing greedy collocation schemes for the approximation of the solution of these kind of problems.
First, we extend and unify the framework for the analysis of standard greedy kernel interpolation algorithms from \cite{wenzel2022analysis} to generalized 
greedy kernel interpolation. 
The analysis of these generalized schemes require decay estimates for certain Kolmogorov widths, which we derive here for the first time answering an open 
question posed in~\cite{schaback2019greedy}. 
Combining these new rates with the extension of the framework from \cite{wenzel2022analysis}, we are able to obtain error estimates for the whole family of 
generalized $\beta$-greedy schemes. These estimates are comparable with the rates obtained with optimally placed points, and additionally, when using full 
adaptivity ($\beta=1$), the  
resulting convergence rates break the curse of dimensionality, in the sense of having a
dimension-independent decay factor. This indicates that the approach may especially be beneficial
for high-dimensional PDEs, where standard mesh-based techniques cannot easily be applied.

The paper is structured as follows.
\Cref{sec:background_results} starts by giving required background information on PDEs, greedy kernel interpolation, generalized interpolation and pointing to the corresponding literature.
\Cref{sec:analysis_kolm_width} provides an analysis of Kolmogorov $n$-widths related to
the functionals of the considered BVP. %
\Cref{sec:analysis_fgreedy} extends the %
target-data dependent greedy analysis of \cite{wenzel2022analysis} to generalized kernel interpolation. 
\Cref{sec:conv_rates_pde_greedy} combines the results of the previous sections to derive convergence rates of the error $\Vert u - s_n \Vert_{L^\infty(\Omega)}$ between the approximant and the true solution of the PDE.
\Cref{sec:numerical_exp} presents numerical experiments on both low-dimensional and
high-dimensional BVPs.
\Cref{sec:outlook_ideas} gives a conclusion and formulates an outlook.

%% file: chapters/02_background.tex
\section{Background results} \label{sec:background_results}

In the following, we collect more required background information about partial differential equations (\Cref{subsec:background_pde}), 
interpolation with translational invariant kernels (\Cref{subsec:transl_inv_kernel}), greedy and generalized interpolation (\Cref{subsec:generalized_kernel_interpolation}),
as well as other kernel methods to solve PDEs (\Cref{subsec:solving_pde_collocation}).

\subsection{A priori bounds for elliptic PDEs}
\label{subsec:background_pde}

  We make use of the following assumptions, which especially give rise to a
  well-posed PDE boundary value problem.

\begin{assumption}
\label{ass:assumptions}
Let $\Omega \subset \R^d$ be a bounded Lipschitz domain with piecewise
smooth
boundary $\partial \Omega$ of dimension $d-1$. 
Let $L$ be a uniformly elliptic operator of second order
of the form 
\begin{equation}
  (Lu)(x) := \sum_{i,j=1}^d a_{ij}(x) \frac{\partial}{\partial {x_i}} \frac{\partial}{\partial x_j} u(x) + \sum_{i=1}^d b_i(x)
  \frac{\partial}{\partial x_i} u(x) + c(x) u(x),
  \label{eqn:second-order-operator}
\end{equation}
with bounded and smooth\footnote{
  Similar to results in literature that we will refer later, e.g.\ \cite{fuselier2012scattered},
  we assume infinite smoothness for the boundary pieces, which then are smooth
  manifolds.
  However, we anticipate that by deliberately tracking the exact derivative
  requirements during the analysis, the same results can be obtained for sufficiently high
  finite differentiability. The same applies to
  coefficient functions of the differential operator, which for simplicity we assume to be
  $C^\infty(\Omega)$ functions.} 
coefficient functions $a_{ij},b_i, c$ on $\Omega$
with $c\leq 0$. %
We assume to have Dirichlet boundary values $g \in \mathcal{C}(\partial \Omega)$,
PDE right hand side
$f \in \mathcal{C}(\Omega)$ %
and there exists a unique strong solution
$u\in \mathcal{C}^2(\Omega)\cap \mathcal{C} (\bar \Omega)$ of the BVP (Eq.~\eqref{eq:pde}).
\end{assumption}
The assumption on the existence of a unique solution can be removed, if one
leverages suitable existence and uniqueness theorems for special PDE problems.

Kernel methods are intrinsically related to the use of the $\Vert \cdot \Vert_{L^\infty(\Omega)}$ norm
because of the reproducing property.
Therefore we recall
an a-priori bound for
the $\Vert \cdot \Vert_{L^\infty(\Omega)}$ norm of the solution of Eq.~\eqref{eq:pde} in dependence of
the data, 
cf.\ \cite[Theorem 4.11]{renardy2006introduction}
\begin{equation}\label{eq:pre_max_prin}
\Vert u \Vert_{L^\infty(\Omega)} \leq  C \left( \Vert f \Vert_{L^\infty(\Omega)} + \Vert g \Vert_{L^\infty(\partial \Omega)} \right).
\end{equation}
Given a (kernel) approximant $s_n$, the difference $v:=u-s_n$ solves the same
BVP but with data $L v = f - L s_n$ and $g-s_n$, so that~\eqref{eq:pre_max_prin} implies that
\begin{align}
\label{eq:maximum_principle_elliptic_operators}
\Vert u - s_n \Vert_{L^\infty(\Omega)} &\leq C \left( \Vert f - Ls_n \Vert_{L^\infty(\Omega)} + \Vert g - s_n \Vert_{L^\infty(\partial \Omega)} \right),
\end{align}
which allows to conclude convergence rates on $\Vert u - s_n \Vert_{L^\infty(\Omega)}$ based on convergence rates on $\Vert f - Ls_n \Vert_{L^\infty(\Omega)}$ and $\Vert g - s_n \Vert_{L^\infty(\partial\Omega)}$, which will be derived in \Cref{sec:analysis_fgreedy}. The convergence rate analysis in the current article will be possible by assuming that the target function has some suitable Sobolev regularity, which is satisfied due to the assumed classical differentiability of the solution and the assumptions on the domain.

We emphasize that our approach can straightforwardly be extended to any
linear PDE problem of potential higher order, parabolic or hyperbolic or mixed type,
as long as a suitably regular solution exists and a corresponding a priori bound is available.
The restriction to second order elliptic problems is mainly for simplifying the
presentation and reflecting that our experiments will only cover such elliptic problems. 
\subsection{RKHS of translational invariant kernels and Sobolev spaces}
\label{subsec:transl_inv_kernel}

A translational invariant kernel can be written as
\begin{align}
\label{eq:transl_inv_kernel}
k(x, y) = \Phi(x - y), \qquad \text{for all\ } x, y \in \R^d,
\end{align}
for some function $\Phi: \R^d \rightarrow \R$.
Translational invariant kernels can be characterized by the decay of the Fourier
transform $\hat{\Phi}$ of $\Phi$, 
i.e.\ we assume there exist constants $c, C>0, \tau > d/2$, such that
\begin{align}
\label{eq:rbf_asymptotic_sobolev}
c (1+ \Vert \omega \Vert^2)^{-\tau} \leq \Phi(\omega) \leq C(1+ \Vert \omega \Vert^2)^{-\tau}
\end{align}
for all $\omega \in \mathbb{R}^d$. 

In this case, $\mathcal{H}_k(\R^d)$ is norm equivalent to a Sobolev space,
i.e.\ $\mathcal{H}_k(\R^d) \asymp H^\tau(\mathbb{R}^d) = W_2^\tau(\R^d)$, where the order $\tau>d/2$ is possibly fractional. %
This norm equivalence can be extended to the RKHS $\ns$ on bounded domains $\Omega \subset \R^d$ under suitable conditions on the boundary $\partial \Omega$, such as its Lipschitz continuity,
i.e.\ it holds $\ns \asymp H^\tau(\Omega) = W_2^{\tau}(\Omega)$. 

In order to estimate the error $\Vert f - s_n \Vert_{L^\infty(\Omega)}$ between a function $f \in \ns$ and its kernel interpolant $s_n$, 
one usually makes use of sampling inequalities and corresponding zero lemmata,
which exist for Sobolev spaces or spaces of analytic functions \cite{wendland2005approximate,narcowich2005sobolev}. 
They bound the error in terms of the fill distance $h = h_{X, \Omega}$, which describes the largest hole within $\Omega$ where no point from $X$ exists, i.e.\
\begin{align}
\label{eq:fill_distance}
  h := %
  h_{X, \Omega} := \sup_{x \in \Omega} \min_{x_j \in X} \Vert x - x_j \Vert.
\end{align}
The same definition of the fill distance can also be used for the boundary $\partial \Omega$ of $\Omega$.
Here we recall \cite[Theorem 2.2]{gia2006continuous},
where $| \cdot |_{W_q^m(\Omega)}$ denotes the Sobolev semi-norm, which uses only the highest derivatives, $\Vert f|_X \Vert_\infty := \max_{x \in X} |f(x)|$, and $(x)_+ := \max\{x,0\}$.
\begin{theorem}
\label{th:estimate_derivatives_wendland}
Suppose $\Omega \subset \mathbb{R}^d$ is a bounded domain satisfying an interior cone condition and having a Lipschitz boundary. Let $X \subset \Omega$ be a discrete set with sufficiently small fill distance $h \leq h_0$. Let $\tau = k + s$ with $k \in \mathbb{N}, 0 \leq s < 1, 1 \leq p < \infty, 1 \leq q \leq \infty, m \in \mathbb{N}_0$ with $k>m+d/p$ if $p>1$ or $k \geq m + d/p$ if $p=1$. Then for each $f \in W_p^\tau(\Omega)$ we have that
\begin{equation*}
|f|_{W_q^m(\Omega)} \leq C \left( h^{\tau-m-d(1/p-1/q)_+} |f|_{W_p^{\tau}(\Omega)} + h^{-m} \Vert f|_X \Vert_\infty \right),
\end{equation*}
where $C>0$ is a constant independent of $f$ and $h$.
\end{theorem}

Similar statements as in \Cref{th:estimate_derivatives_wendland} can also be derived for interpolation on manifolds under some assumptions on the shape of the manifold.
In the following we are particularly interested in statements for the boundary of the domain $\Omega$, i.e.\ the manifolds $\mathbb{M} \subseteq \partial \Omega$. 
For this, we recall the result \cite[Lemma 10]{fuselier2012scattered}.

\begin{theorem}
\label{th:sampling_inequ_manifolds}
Let $\mathbb{M}$ be a smooth manifold of dimension $\tilde{d}$ and let $1 \leq p, q \leq \infty, t \in \R$ with $t > \tilde{d}/p$ if $p>1$ or $t \geq \tilde{d}$ if $p=1$. 
Let $\mu \in \N$ satisfy $0 \leq \mu \leq \lceil t - \tilde{d}(1/p - 1/q)_+ \rceil - 1$.
Also, let $X \subset \mathbb{M}$ be a discrete set with fill distance $h_{X, \mathbb{M}} \leq C_\mathbb{M}$. If $f \in W_p^t(\mathbb{M})$ satisfies $f|_X = 0$, 
then 
\begin{align*}
|f|_{W_q^\mu(\mathbb{M})} \leq C h_{X, \mathbb{M}}^{t - \mu - \tilde{d}(1/p - 1/q)_+} |f|_{W_p^t(\mathbb{M})}.
\end{align*}
\end{theorem}
Here $h_{X, \mathbb{M}}$ is defined as the fill distance $h_{X,\Omega}$, but using the intristic distance on the manifold instead of the Euclidean one, and 
$W_p^t(\mathbb{M})$ is a Sobolev space on the manifold $\mathbb{M}$.

We remark that similar statements as in \Cref{th:estimate_derivatives_wendland} and \Cref{th:sampling_inequ_manifolds} are also available for interpolation with analytic kernels such as the Gaussian kernel. 
However, as we focus on Sobolev kernels
we do not recall those estimates here.

\subsection{Generalized kernel interpolation}\label{subsec:generalized_kernel_interpolation}

We consider the given set of functionals $\Lambda_n \equiv \{\lambda_1, \dots, \lambda_n\} \subset \ns'$, which we
assume to be linearly independent
and with associated Riesz representers $v_{\lambda_1}, \dots, v_{\lambda_n} \in \ns$, and
we recall that we use the ansatz space $V_n = \Sp \{ v_{\lambda_i} | \lambda_i \in \Lambda_n \}$. 
In order to bound approximation errors, 
a generalized power function similar to the standard power function in Eq.\ \eqref{eq:power_func_standard} can be defined
for generalized interpolation \cite[Chapter 16]{wendland2005scattered}
as
\begin{align}
\begin{aligned}
\label{eq:power_func_generalized}
P_{\Lambda_n}(\lambda) &:= \Vert v_\lambda - \Pi_{V_n}(v_\lambda) \Vert_{\ns} 
&= \sup_{0 \neq u \in \ns} \frac{|\lambda(u - \Pi_{V_n}(u))|}{\Vert u \Vert_{\ns}},
\end{aligned}
\end{align}
which immediately gives for all $u \in \ns$ that
\begin{align*}
| \lambda(u) - \lambda( \Pi_{V_n}(u) ) | \leq P_{\Lambda_n}(\lambda) \cdot \Vert u \Vert_{\ns}.
\end{align*}
We like to point out again that for $\lambda_1 := \delta_{x_1}, ..., \lambda_n := \delta_{x_n}$ the generalized approximant as well as the generalized power 
function boil down to their standard counterparts by identifying $x$ and $\delta_x$.

In order to avoid a frequent recomputation of the coefficients $\alpha_j$ within the standard
kernel expansion from Eq.\ \eqref{eq:kernel_interpolant_standard} and
Eq.\ \eqref{eq:kernel_interpolant_generalized}, typically the
Newton basis $\{ v_1, \dots, v_n \}$ \cite{Pazouki2011,schaback2019greedy}
that is obtained by orthonormalizing the Riesz representers $\{v_{\lambda_1}, \dots, v_{\lambda_n} \}$
is applied. 
Then the (standard or generalized) interpolant $s_n$ of $u \in \ns$ can be written as
\begin{align*}
s_n( \cdot ) = \sum_{j=1}^n \langle u, v_j \rangle_{\ns} v_j(\cdot),
\end{align*}
and it can be shown that the coefficients of the interpolant expressed in the Newton basis satisfy
\begin{align}
\label{eq:newtown_coefficient}
\langle u, v_n \rangle_{\ns} = \frac{\lambda_n(u - s_{n-1})}{P_{\Lambda_{n-1}}(\lambda_n)}.
\end{align}
Therefore, it also holds that
\begin{align}
\label{eq:ns_norm_via_sum}
\sum_{j=1}^\infty \left( \frac{|\lambda_j(u-s_{j-1})|}{P_{\Lambda_{j-1}}(\lambda_{j})} \right)^2 \leq \Vert u \Vert_{\ns}^2,
\end{align}
where the inequality is in fact an equality if and only if $s_n \stackrel{n \rightarrow \infty}{\longrightarrow} u$ in $\ns$.

\subsection{Solving PDEs by collocation} \label{subsec:solving_pde_collocation}

As motivated in the introduction in \Cref{sec:introduction}, the choice of $\Lambda$
according to Eq.~\eqref{eq:lambda_sets} naturally yields an approximation of the PDE solution in
the setting of generalized interpolation.
This approach is usually called symmetric kernel collocation since it is obtained by a symmetric
application of the differential operators to the two arguments of the kernel or, 
in other words, since the differential operators are applied both to construct the
ansatz in Eq.\ \eqref{eq:kernel_interpolant_generalized}, 
and to apply the generalized interpolation conditions of Eq.\ \eqref{eq:gen_interp_cond}. 
Since this method is based on the optimality principles described in \Cref{sec:introduction},
it can be analyzed by extending in a systematic manner the ideas used for plain function interpolation. 

If the approximant $s_n$ in Eq.\ \eqref{eq:kernel_interpolant_generalized} is instead assumed to be represented as $s_n:= \sum_{j=1}^n \alpha_j k(\cdot, x_j)$ 
independently of the specific problem to be solved, the resulting method is usually referred as unsymmetric kernel collocation, or method 
of \textit{Kansa} (see \cite{Kansa1990a,Kansa1990b}). This approach has the advantage of requiring less regularity on the kernel and of being usually 
computationally more flexible, but the corresponding theoretical analysis is more subtle \cite{Hon2001,Schaback2007}.

These ideas have been significantly extended by introducing a local approxi\-mation scheme known as Radial Basis Function Finite Difference
(RBF-FD) \cite{Tolstykh2003,Fornberg2015},
The convergence of this method has been extensively studied only
very recently \cite{tominec2021least}.

%% file: chapters/03_analysis_kolmogorov_rates.tex
\section{Analysis of the Kolmogorov widths for PDE collocation functionals} \label{sec:analysis_kolm_width}

For the analysis of the generalized greedy kernel approximation algorithms in \Cref{sec:analysis_fgreedy} we will extend an analysis introduced in \cite{santin2017convergence,wenzel2022analysis}. 
For this, the knowledge about some Kolmogorov $n$-widths \cite{pinkus2012n},
that are associated to our approximation problems, are required: %
The Kolmogorov $n$-widths $d_n(\Lambda)$ of $\Lambda \subset \calh'$ in the dual Hilbert space $\calh'$ of a Hilbert space $\calh$ are defined as  %
\begin{align}
\label{eq:definition_kolmogorov_nwidth}
d_n(\Lambda) := d_n(\Lambda, \calh') &= \inf_{\substack{G_n \subset \calh' \\ \dim(G_n) = n}} \sup_{\mu \in \Lambda} \dist(\mu, G_n)_{\calh'}.
\end{align}
The shorthand notation $d_n(\Lambda)$ will be used whenever the Hilbert space $\calh$ or its dual $\mathcal{H}'$ are clear.
We remark that we can rewrite this Kolmogorov width:
It is possible to move this definition from the dual space $\calh'$ to the primal space $\calh$, since the Riesz representer theorem gives an isometry. 
Furthermore we calculate the distance in the Hilbert space by using orthogonal projections and obtain that
\begin{align}
\label{eq:kolmogorov_dual}
d_n(\Lambda, \calh') &\equiv \inf_{\substack{G_n \subset \calh' \\ \dim(G_n) = n}} \sup_{\mu \in \Lambda} \dist(\mu, G_n)_{\calh'} 
= \inf_{\substack{H_n \subset \calh \\ \dim(H_n) = n}} \sup_{\mu \in \Lambda} \dist(v_\mu, H_n)_{\calh} \notag \\
&= \inf_{\substack{H_n \subset \calh \\ \dim(H_n) = n}} \sup_{\mu \in \Lambda} \Vert v_\mu - \Pi_{H_n}(v_\mu) \Vert_{\calh}.
\end{align}

For the problem of PDE approximation, the Hilbert space is chosen as the RKHS of the kernel $k$, i.e.\ $\calh := \ns$.
The set $\Lambda$ is defined via the PDE equations in \eqref{eq:lambda_sets}, i.e.\ $\Lambda := \Lambda_L \cup \Lambda_B$. 
With $x_\lambda \in \Omega\ \cup\ \partial\Omega$ we will denote in the following the point related to a functional $\lambda \in \Lambda \subset \ns'$, i.e.\ where the evaluation takes place within $\Omega$ (for evaluations of $\lambda \in \Lambda_L$) or on its boundary $\partial \Omega$ (for evaluations of the boundary functionals $\lambda \in \Lambda_B$).
Since we consider kernels that satisfy Eq.~\eqref{eq:rbf_asymptotic_sobolev}, 
we have the norm equivalence of the RKHS $\ns$ to a Sobolev space $H^\tau(\Omega)$ for some $\tau > d/2$. 
Therefore we immediately obtain
\begin{align*}
\dist (v_\mu, H_n)_{H^\tau(\Omega)} &\asymp \dist (v_\mu, H_n)_{\ns} \\
\Rightarrow \qquad d_n(\Lambda, H^\tau(\Omega)') &\asymp d_n(\Lambda, \ns'),
\end{align*}
i.e.\ instead of analyzing the Kolmogorov $n$-widths related to the RKHS $\ns$ we can analyze the $n$-widths related to the Sobolev space $H^\tau(\Omega)$, because they provide the same asymptotics. 

As remarked in \cite{schaback2019greedy}, 
a thorough analysis of those Kolmogorov $n$-widths for PDE collocation
functionals  does not seem to be available so far in the literature.
We refer to \Cref{rem:relation_kolmogorov_widths} below for a discrimination
from other types of PDE-related Kolmogorov $n$-widths.
Nevertheless \cite{schaback2019greedy} conjectured that there should be a $\tau-$ (smoothness of the Hilbert space) and $d-$ (dimensionality of the domain 
$\Omega \subset \R^d$) dependent decay rate $\kappa(\tau, d)$ in the sense that
\begin{align*}
d_n(\Lambda, \calh') \leq C n^{-\kappa(\tau, d)} \quad \text{for } n \rightarrow \infty.
\end{align*}
Indeed it will be proven in the following that the Kolmogorov $n$-widths
related to our PDE problems exactly behave according to such a relation.
We stress that the difficulty in obtaining bounds for $d_n(\Lambda)$ is due to the fact that $\Lambda$ is the union of two sets of different functionals.

\begin{rem}[Relation to Kolmogorov $n$-widths for parametric PDEs]
\label{rem:relation_kolmogorov_widths}
We want to briefly relate and discriminate the present type of
Kolmogorov $n$-widths for PDE collocation functionals to another
width in the literature.
Especially in the context of parametric PDEs and projection-based model
order reduction, 
solution manifolds of the parametric solutions appear and need to be
approximated, e.g.\ \cite{Haasdonk2017}.
A crucial property for successful approximation in these scenarios is a rapid decay of the Kolmogorov $n$-widths of those solution manifolds. 
Especially in transport-dominated cases, the decay can be very slow for the wave equation \cite{greif2019decay} or the convection of discontinuities \cite{ohlberger2016reduced}.
In contrast, here we are not aiming at approximating solution manifolds of parametric PDE solutions, 
but addressing the Kolmogorov $n$-width of the set $\Lambda$ of the PDE collocation functionals of a single non-parametric PDE.
\end{rem}

The next \Cref{subsec:lower_upper_bound_kolmogorov} provides results to obtain lower and upper bounds on $d_n(\bigcup_{j=1}^M \Lambda_j)$ with help of 
$d_n(\Lambda_j)$, $j=1, \dots, M$. 
Subsequently we derive in \Cref{subsec:application_sobolev} precise decay rates for our special case of Sobolev spaces.

\subsection{Abstract setting: Bounds on $d_n(\Lambda)$}
\label{subsec:lower_upper_bound_kolmogorov}

We call this subsection the \textit{abstract setting}, because the following \Cref{prop:lower_bound} and \Cref{prop:upper_bound} provide bounds on the Kolmogorov $n$-widths of general sets that are given as the union of two or more (not necessarily disjoint) sets. 
Moreover, although we use the same notation used before for orthogonality, it should be noted that \Cref{prop:lower_bound} and \Cref{prop:upper_bound} work also for Kolmogorov $n$-widths in Banach spaces because we only use distance-based arguments and no dot products.

The following \Cref{prop:lower_bound} shows that it is possible to lower bound the Kolmogorov $n$-width $d_n(\Lambda)$ of $\Lambda = \bigcup_{j=1}^M \Lambda_j$ with help of the Kolmogorov $n$-widths $d_n(\Lambda_j), j=1,\dots, M$.

\begin{prop}
\label{prop:lower_bound}
Let $\Lambda := \bigcup_{j=1}^M \Lambda_j \subset \mathcal{H}'$. Then it holds
\begin{align*}
d_n(\Lambda) \geq \max_{j=1, \dots, M} d_n(\Lambda_j).
\end{align*}
\end{prop}
\begin{proof}
For any $j=1, \dots, M$ we obtain
\begin{align*}
d_n(\Lambda) &\equiv \inf_{\substack{G_n \subset \calh' \\ \dim(G_n) = n}} \sup_{\mu \in \Lambda} \dist(\mu, G_n)_{\calh'} \\
&\geq \inf_{\substack{G_n \subset \calh' \\ \dim(G_n) = n}} \sup_{\mu \in \Lambda_j} \dist(\mu, G_n)_{\calh'} = d_n(\Lambda_j).
\end{align*}
Therefore the statement directly follows.
\end{proof}

In a similar way it is possible to derive an upper bound on $d_n(\Lambda)$ via the quantities $d_n(\Lambda_j), j=1, \dots, M$.
The proof can be found in \Cref{sec:proofs0}.
\begin{prop}
\label{prop:upper_bound}
Let $\Lambda := \bigcup_{j=1}^M \Lambda_j \subset \mathcal{H}'$. Then it holds
\begin{align*}
d_n(\Lambda) \leq \min_{\sum_{j=1}^M n_j \leq n} \max_{j=1, \dots, M} d_{n_j}(\Lambda_j).
\end{align*}
\end{prop}

\Cref{prop:lower_bound} and \Cref{prop:upper_bound} enable us to analyze the asymptotics of $d_n(\Lambda)$ with help of the knowledge of $d_n(\Lambda_j)$, $1 \leq j \leq M$. 
This will be applied in the following subsection using $M=2$ and $\Lambda_L, \Lambda_B$. %

\subsection{Kernel setting: Bounds on $d_n(\Lambda)$}
\label{subsec:application_sobolev}

We start by analyzing the Kolmogorov $n$-widths $d_n(\Lambda_L)$ and $d_n(\Lambda_B)$ as defined in Eq.\ \eqref{eq:lambda_sets}.
To the best of our knowledge there are no similar estimates on $d_n(\Lambda_L)$ and $d_n(\Lambda_B)$ available in the literature. 
For $d_n(\Lambda_L)$ we obtain the following statement, where the proof
can be found in the Appendix.
\begin{theorem}[Upper bound on $d_n(\Lambda_L)$]
  \label{th:kolmogorov_estimate_Lambda1}
Let $L$ be a linear differential operator according to \hyperref[ass:assumptions]{Assumption 1}, 
$\Omega \subset \R^d$ a bounded Lipschitz domain. 
Let $k$ be a kernel s.t.\ $\ns \asymp H^\tau(\Omega)$ with $\tau > 2 + d/2$. 
Then for $\Lambda_L$ as defined in Eq.\ \eqref{eq:lambda_sets}
and all $n\geq 1$ it holds
\begin{align*}
  d_n(\Lambda_L, \ns') \leq C_L n^{\frac{1}{2} - \frac{\tau-2}{d}},
\end{align*}
where $C_L>0$ is a constant depending on $L$ but not on $n$.
\end{theorem}

Similarly we can state an upper bound of the widths
of the set of the boundary functionals. The proof again is postponed
to the Appendix.

\begin{theorem}[Upper bound on $d_n(\Lambda_B)$]
\label{th:kolmogorov_estimate_Lambda2}
Let $\Omega \subset \R^d$ be a bounded Lipschitz domain with piecewise smooth boundary $\partial \Omega$ of dimension $d-1$.
Let $k$ be a kernel s.t.\ $\ns \asymp H^\tau(\Omega)$ with $\tau > d/2$.
Then for $\Lambda_B$ as defined in Eq.\ \eqref{eq:lambda_sets}
and all $n\geq 1$ it holds
\begin{align*}
  d_n(\Lambda_B) \leq C_B n^{\frac{1}{2} - \frac{\tau - 1/2}{d-1}},
\end{align*}
where $C_B>0$ is a constant independent of $n$.
\end{theorem}

We remark that $\frac{1}{2} - \frac{\tau-1/2}{d-1} < \frac{1}{2} - \frac{\tau-2}{d} < 0$, i.e.\ the decay in \Cref{th:kolmogorov_estimate_Lambda2} is faster than the decay in \Cref{th:kolmogorov_estimate_Lambda1}.

We can now
leverage the abstract results from \Cref{subsec:lower_upper_bound_kolmogorov} and combine them with the concrete estimates for the kernel setting
from \Cref{th:kolmogorov_estimate_Lambda1}
and \Cref{th:kolmogorov_estimate_Lambda2}.
The basic idea is to carefully balance the number of points which are spent on the interior and on the boundary of the domain.
Like this it is possible to derive an upper bound on the Kolmogorov $n$-width $d_n(\Lambda)$.

\begin{theorem}[Upper bound on $d_n(\Lambda)$]
\label{th:kolmogorov_estimate_Lambda}
Let $\Omega \subset \R^d$ be a bounded Lipschitz domain with piecewise
smooth boundary $\partial \Omega$ of dimension $d-1$.
Let $L$ be a linear differential operator according to \hyperref[ass:assumptions]{Assumption 1}. %
Let $k$ be a kernel s.t.\ $\ns \asymp H^\tau(\Omega)$ with $\tau > 2 + d/2$.
Then the Kolmogorov $n$-widths of the collocation functionals from
\eqref{eq:pde} %
satisfy for all $n\geq 1$
\begin{align*}
  d_n(\Lambda) \leq C n^{\frac{1}{2} - \frac{\tau - 2}{d}},
\end{align*}
where $C>0$ depends on $\tau, d, L$ but not on $n$.
\end{theorem}

Applied to the PDE setting of \cite{schaback2019greedy} this means that
the Kolmogorov $n$-widths related to the collocation functionals
of the PDE $Lu = f$ (here given by $d_n(\Lambda_L) \asymp n^{\frac{1}{2} - \frac{\tau - 2}{d}}$) decide about the
Kolmogorov $n$-widths for the whole boundary value problem
(here $d_n(\Lambda)$), as it was already conjectured
in \cite{schaback2019greedy}.
The additional boundary conditions do not influence the decay of the
Kolmogorov $n$-widths, because the Kolmogorov $n$-widths related to the
boundary conditions 
decay faster. \\

The proposed analysis in this section is in principle not restricted to only two sets of functionals $\Lambda_L$ and $\Lambda_B$. This is explained in the 
following remark.

\begin{rem}
We remark that the results of this section can be extended to more sets of functionals. %
This was already done for the abstract setting in \Cref{subsec:lower_upper_bound_kolmogorov} by considering $\Lambda = \bigcup_{j=1}^M \Lambda_j$.
In the PDE collocation setting this refers e.g.\ to more types of
boundary conditions, i.e.\ we can incorporate both Dirichlet
boundary conditions and Neumann boundary conditions by using
$\Lambda_{B_1}$ (for the Dirichlet-related functionals) and
$\Lambda_{B_2}$ (for the Neumann-related functionals) instead of
only using $\Lambda_B$.
\end{rem}

%% file: chapters/04_analysis_kernel_greedy.tex
\section{%
  Analysis of PDE-greedy schemes} \label{sec:analysis_fgreedy}
 
We start by defining the notion of a PDE-greedy kernel algorithm.
\begin{definition} 
Consider a BVP \eqref{eq:pde}, 
a kernel $k$ and selection criteria
$\{\eta_j\}_{j=1}^\infty$ with $\eta_j:\Lambda \to [0,\infty)$ for all $j\in\N$, and with $\Lambda$ defined in
Eq.\ \eqref{eq:lambda_sets}. 
An algorithm is called PDE-greedy kernel algorithm
if it selects $\lambda_{n+1} \in \Lambda$ according to 
\begin{align*}
\lambda_{n+1} = \argmax_{\lambda \in \Lambda} \eta_n(\lambda).
\end{align*}
\end{definition}
A whole class of PDE-greedy algorithms with specific selection criteria $\eta_n$ is given in \Cref{def:beta_greedy_algorithms}.
For the moment we continue with this general notion in order to stress the fact that the results of the following \Cref{sec:41} hold for general selection criteria.

For this general class of algorithms
some basic analytical bounds on geometric means of power
function evaluations can be provided, as presented in \Cref{sec:41}. 
Then in \Cref{sec:42} we define a scale of $\beta$-greedy
procedures, which are favourable
in the sense that these allow more refined analysis that is presented in
\Cref{sec:43}. Those results will serve for obtaining
the main convergence rate statements in \Cref{sec:conv_rates_pde_greedy}.

\subsection{Decay of power function quantities}
\label{sec:41}

In \cite[Section 3]{wenzel2022analysis} the abstract analysis of greedy algorithms in Hilbert spaces of \cite{devore2013greedy} was extended by considering a broader class of greedy algorithms. 
Here we will make use of those results and apply them to obtain bounds on 
\begin{align}
\label{eq:geom_mean_powerfunc}
\left[ \prod_{i=n+1}^{2n} P_{\Lambda_i}(\lambda_{i+1}) \right]^{1/n},
\end{align}
see \Cref{cor:decay_powerfunc_quantity},
with $\Lambda_i \equiv \{\lambda_1, \dots, \lambda_i\}$ as introduced in \Cref{subsec:generalized_kernel_interpolation}.
This will later enable us to derive convergence rates for a range of target-data dependent PDE-greedy kernel algorithms.

For this we briefly recall the notation of \cite[Section 3]{wenzel2022analysis}:
Let $\mathcal{H}$ be a Hilbert space with norm $\Vert \cdot \Vert := \Vert \cdot \Vert_\mathcal{H}$ and $\mathcal{F} \subset \mathcal{H}$ a compact subset. 
Without loss of generality we assume that $\Vert f \Vert \leq 1$ for all $f \in \mathcal{F}$.
We consider algorithms that select elements $f_0, f_1, \ldots \in \mathcal{F}$, without yet specifying any particular selection criterion. We define $V_n :=
\text{span}\{f_0, \dots, f_{n-1}\}$ and the following quantities, where $H_n$ is any $n$-dimensional subspace of $\mathcal{H}$.
\begin{align} 
\begin{aligned} \label{eq:quantities}
d_n :=& d_n(\mathcal{F})_\mathcal{H} := \inf_{H_n \subset \mathcal{H}} \sup_{f \in \mathcal{F}} \mathrm{dist}(f, H_n)_\mathcal{H}, \\
\sigma_n :=& \sigma_n(\mathcal{F})_\mathcal{H} := \sup_{f \in \mathcal{F}} \mathrm{dist}(f, V_n)_\mathcal{H},\\
\nu_n :=& \text{dist}(f_n, V_n)_\mathcal{H}.
\end{aligned}
\end{align}

Now we recall \cite[Corollary 2]{wenzel2022analysis} including its assumptions which were stated before in \cite[Theorem 1]{wenzel2022analysis}):

\begin{cor}[Corollary 2 from \cite{wenzel2022analysis}] \label{cor:decay_abstract_setting_prod} 
Consider a compact set $\mathcal{F}$ in a Hilbert space $\mathcal{H}$, and a greedy algorithm that selects elements from $\mathcal F$ according to any arbitrary 
selection rule.
\begin{enumerate}[label=\roman*)]
\item If $d_n(\mathcal{F}) \leq C_0 n^{-\alpha}, n\geq1$, 
then it holds
\begin{align*} %
\left(\prod_{i=n+1}^{2 n} \nu_{i} \right)^{1/n} 
&\leq2^{\alpha+1/2} \tilde{C}_0 e^\alpha  \log(n)^\alpha n^{-\alpha}, \;\;n\geq 3,
\end{align*}
with $\tilde C_0 := \max\{1, C_0\}$.

\item If $d_n(\mathcal{F}) \leq C_0 e^{-c_0 n^\alpha}, n\geq 1$, then it holds
\begin{align*} %
\left( \prod_{i=n+1}^{2n} \nu_i \right)^{1/n} \leq \sqrt{2 \tilde{C}_0} \cdot e^{-c_1 n^\alpha}, \;\;n\geq 2,
\end{align*}
with $\tilde{C}_0 := \max\{1, C_0\}$ and $c_1 = 2^{-(2+\alpha)}c_0 < c_0$.
\end{enumerate}
\end{cor}

We can leverage this abstract result to obtain convergence statements on the quantity of Eq.\ \eqref{eq:geom_mean_powerfunc} by generalizing a convenient link between the abstract setting and the kernel setting, which was first of all established in \cite{santin2017convergence} for function interpolation.
We generalize this %
by %
choosing
$\mathcal{H} := \ns$, $\mathcal{F} := \{ v_{\lambda} ~ | ~ \lambda \in \Lambda_L \cup \Lambda_B \}$ and thus $f_0 :=
v_{\lambda_1}, f_1 := v_{\lambda_2}$, etc. %
Hereby, we have $V_n = \Sp \{ v_{\lambda_i} ~ | ~ i = 1, ..., n \}$.

These choices allow us to formulate the abstract definitions of Eq.\ \eqref{eq:quantities} in the kernel setting for PDE approximation. 
For this we point again to the definition of the generalized power function from Eq.\ \eqref{eq:power_func_generalized}: %
The quantity $\sigma_n$ is then simply the maximal value $\sup_{\lambda \in \Lambda} P_{\Lambda_n}(\lambda)$ of the generalized power function $P_{\Lambda_n}$, while $\nu_n$ is the value of the generalized power function $P_{\Lambda_n}$ at the next selected functional $\lambda_{n+1}$ corresponding to the point $x_{\lambda_{n+1}}$:
\begin{align}\label{eq:connection_1}
\sigma_n 
& \equiv \sup_{f \in \mathcal{F}} \mathrm{dist}(f, V_n)_\calh
= \sup_{f \in \mathcal{F}} \Vert f - \Pi_{V_n}(f) \Vert_{\mathcal{H}} \nonumber\\
&~= \sup_{\lambda \in \Lambda_L \cup \Lambda_B} \Vert v_\lambda - \Pi_{V_n}(v_\lambda) \Vert_{\calh} 
= \sup_{\lambda \in \Lambda_L \cup \Lambda_B} P_{\Lambda_n}(\lambda),\\
\nu_n &\equiv \mathrm{dist}(f_n, V_n)_\mathcal{H} 
= \Vert f_n - \Pi_{V_n}(f_n) \Vert_{\mathcal{H}} \nonumber\\
&= \Vert v_{\lambda_{n+1}} - \Pi_{V_n}(v_{\lambda_{n+1}}) \Vert_{\mathcal{H}} = P_{\Lambda_n}(\lambda_{n+1})\nonumber.
\end{align}
Furthermore, $d_n$ can be similarly bounded as
\begin{align}\label{eq:connection_2}
d_n &\equiv \inf_{H_n \subset \mathcal{H}} \sup_{f \in \mathcal{F}} \mathrm{dist}(f, H_n)_\mathcal{H}  = \inf_{H_n 
\subset \mathcal{H}} \sup_{f \in \mathcal{F}} \Vert f - \Pi_{H_n}(f) \Vert_\mathcal{H}\\
&\leq \inf_{H_n \subset \Sp \{ \mathcal{F} \}} \sup_{f \in \mathcal{F}} \Vert f - \Pi_{H_n}(f) \Vert_{\mathcal{H}} \nonumber \\
&= \inf_{H_n \subset \Sp \{ \mathcal{F} \}} \sup_{\lambda \in \Lambda_L \cup \Lambda_B} \Vert v_\lambda - \Pi_{H_n}(v_\lambda) \Vert_{\mathcal{H}} 
= \inf_{\Lambda_n \subset \Lambda_L \cup \Lambda_B} \sup_{\lambda \in \Lambda_L \cup \Lambda_B} P_{\Lambda_n}(\lambda) 
\nonumber.
\end{align}
Here $H_n$ always indicates an $n$-dimensional subspace.
Due to Eq.\ \eqref{eq:connection_2}, any bound on $\sup_{\lambda \in \Lambda_L \cup \Lambda_B} P_{\Lambda_n}(\lambda)$ for a given set of functionals $\Lambda_n \subset \F$ (which corresponds to a distribution of points in $\Omega$ and on the boundary $\partial \Omega$) gives a bound on $d_n$. 

Additionally, observe that the assumption $\|f\|_\calh \leq 1$ for $f\in\mathcal F$ from the abstract setting is satisfied using the kernel setting as soon as
\begin{align}
\label{eq:pf_normalization}
  \sup_{\lambda \in \Lambda_L \cup \Lambda_B} P_0(\lambda) = & \sup_{\lambda \in \Lambda_L \cup \Lambda_B} \Vert v_\lambda \Vert_{\mathcal{H}} \notag
  \\
= &\max \left(\sup_{\lambda \in \Lambda_L} \Vert v_\lambda \Vert_{\mathcal{H}}, \sup_{\lambda \in \Lambda_B} \Vert v_\lambda \Vert_{\mathcal{H}} \right) \notag \\
= &\max \left( \sup_{x \in \Omega} \sqrt{L^{(1)}L^{(2)} k(x, x)}, \sup_{x \in \partial \Omega} \sqrt{k(x, x)} \right)
\leq 1,
\end{align}
where we use the explicit representations of the Riesz representers
  $v_\lambda = k(x,\cdot)$ if $\lambda=\delta_{x} \in \Lambda_B$, and $v_{\lambda} = L^{(1)} k(x,\cdot)$ if
   $\lambda = \delta_{x} \circ L \in \Lambda_L$. 
The validity of Eq.\ \eqref{eq:pf_normalization}
can always be ensured by rescaling the kernel $k$ to $\epsilon k$, $0 < \epsilon < 1$ (which does not change the RKHS, but only rescales its norm).
Applying the %
link to \Cref{cor:decay_abstract_setting_prod} immediately gives the following result:

\begin{cor} \label{cor:decay_powerfunc_quantity} 
Consider the set $\mathcal{F} = \{ v_{\lambda} ~ | ~ \lambda \in \Lambda_L \cup \Lambda_B \} \subset \ns$,
and a greedy algorithm that selects elements from $\mathcal F$ according to any arbitrary selection rule.
\begin{enumerate}[label=\roman*)]
\item If $d_n(\mathcal{F}) \leq C_0 n^{-\alpha}, n \geq 1$, then it holds
\begin{align*} %
\left(\prod_{i=n+1}^{2 n} P_{\Lambda_i}(\lambda_{i+1}) \right)^{1/n} 
&\leq2^{\alpha+1/2} \tilde{C}_0 e^\alpha  \log(n)^\alpha n^{-\alpha}, \;\;n\geq 3,
\end{align*}
with $\tilde C_0 := \max\{1, C_0\}$.
\item If $d_n(\mathcal{F}) \leq C_0 e^{-c_0 n^\alpha}, n \geq 1$, then it holds
\begin{align*} %
\left( \prod_{i=n+1}^{2n} P_{\Lambda_i}(\lambda_{i+1}) \right)^{1/n} \leq \sqrt{2 \tilde{C}_0} \cdot e^{-c_1 n^\alpha}, \;\;n\geq 2,
\end{align*}
with $\tilde{C}_0 := \max\{1, C_0\}$ and $c_1 = 2^{-(2+\alpha)}c_0 < c_0$.
\end{enumerate}
\end{cor}

\begin{proof}
This is a direct consequence of \Cref{cor:decay_abstract_setting_prod} using the choices of $\mathcal{H}, \mathcal{F}$ and $f_0, f_1, \ldots, f_i$ as introduced before and summarized in
Eq.~\eqref{eq:connection_1}.
\end{proof}

The assumptions of \Cref{cor:decay_powerfunc_quantity} about the decay of the Kolmogorov widths $d_n(\mathcal{F})$ can be ensured via \Cref{th:kolmogorov_estimate_Lambda} because of the choice $\mathcal{F} \equiv \{ v_\lambda, \lambda \in \Lambda \}$, and the Riesz representer isometry, i.e.
\begin{align*}
d_n(\Lambda) \equiv d_n(\Lambda, \ns') = d_n(\mathcal{F}, \ns) \equiv d_n(\mathcal{F}).
\end{align*}

We note that the results of this subsection simplify to the results of \cite[Section 5]{wenzel2022analysis} if we choose $L = \Id$, i.e.\ $\{ v_\lambda ~ | ~ \lambda \in \Lambda_L \cup \Lambda_B \} = \{ k(\cdot, x),\ x \in \overline{\Omega} \}$.

\subsection{Definition of PDE-$\beta$-greedy algorithms}
\label{sec:42}

We generalize the notion of $\beta$-greedy algorithms as introduced in \cite{wenzel2022analysis} for standard interpolation to PDE-$\beta$-greedy algorithms via the following definition.

\begin{definition}[Definition 4] \label{def:beta_greedy_algorithms}
A PDE-greedy kernel algorithm is called PDE-$\beta$-greedy algorithm with $\beta \in [0, \infty]$, if the next interpolation functional is chosen %
\begin{enumerate}
\item for $\beta \in [0, \infty)$ according to
\begin{align}
\label{eq:beta_greedy_selection_criterion}
\lambda_{n+1} = \argmax_{0 \neq \lambda \in \Lambda \setminus \Lambda_n} |\lambda(u-\Pi_{V_n}(u))|^\beta \cdot P_{\Lambda_n}(\lambda)^{1-\beta}, 
\end{align}
\item for $\beta = \infty$ according to 
\begin{align}
\lambda_{n+1} = \argmax_{0 \neq \lambda \in \Lambda \setminus \Lambda_n} \frac{|\lambda(u-\Pi_{V_n}(u))|}{P_{\Lambda_n}(\lambda)}.
\end{align}
\end{enumerate}
\end{definition}

This %
scale
of PDE-greedy algorithms includes several special cases, i.e.\ $\beta = 0$ (PDE-$P$-greedy), 
$\beta = 1$ (PDE-$f$-greedy) and $\beta = \infty$ (PDE-$f/P$-greedy) as visualized in \Cref{fig:number_line}.
Here we introduce the names in brackets in analogy to the existing
corresponding greedy kernel schemes for function interpolation. 

\begin{figure}[t]
\setlength\fwidth{.4\textwidth}
\input{Figures/number_line_tikz.tex}
\vspace{-.7cm}
\caption{Visualization of the scale of the PDE-$\beta$-greedy algorithms on the real line. 
  The important cases for $\beta \in \{0, %
  1\}$ and $\beta 
\rightarrow \infty$ are marked.}
\label{fig:number_line}
\end{figure}
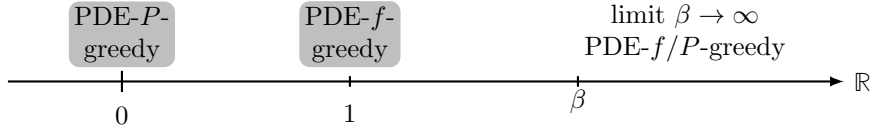

\subsection{Analysis of PDE-greedy algorithms}
\label{sec:43}

We remark that the following lemma and theorems are generalizations
of the corresponding statements within \cite{wenzel2022analysis},
when we pick $\Lambda = \{ \delta_x ~ | ~ x \in \Omega \}$ or equivalently
$\mathcal{F} = \{ k(\cdot, x) ~ | ~ x \in \Omega \}$.
Therefore the proofs
are not included in the main text, but can be found in \Cref{sec:proofs}.

\begin{lemma}[Generalization
 of Lemma 6 from \cite{wenzel2022analysis}]
\label{lem:estimate_product}
For any sequence $\{ \lambda_i \}_{i \in \N} \subset \Lambda$ and any $u \in \ns$, 
denoting the error es $e_i = u - \Pi_{V_i}(u)$, it holds for all $n \geq 1$ that
\begin{align}
\label{eq:estimate_product}
\left[ \prod_{i=n+1}^{2n} \lambda_{i+1}(e_i) \right]^{1/n} \leq n^{-1/2} \cdot \Vert e_{n+1} \Vert_{\ns} \cdot \left[ \prod_{i=n+1}^{2n} P_{\Lambda_i}(\lambda_{i+1}) \right]^{1/n}.
\end{align}
\end{lemma}

\noindent Now we can adapt \cite[Lemma 7]{wenzel2022analysis}. %

\begin{lemma}[Generalization of Lemma 7 from \cite{wenzel2022analysis}] \label{lem:beta_greedy}
Any PDE-$\beta$-greedy algorithm with $\beta \in [0,\infty]$ applied to a function $u \in \ns$ yields the following bounds for the error $e_i = u - \Pi_{V_i}(u)$ for $i\geq 0$
\begin{enumerate}[label={\alph*)}]
\item in the case of $\beta \in [0, 1]$
\begin{align}\label{eq:bound_ri_first}
\sup_{\lambda \in \Lambda} |\lambda(e_i)| &\leq |\lambda_{i+1}(e_i)|^{\beta} \cdot P_{\Lambda_i}(\lambda_{i+1})^{1-\beta} \cdot \Vert e_i \Vert_{\ns}^{1-\beta},
\end{align}
\item in the case of $\beta \in (1, \infty]$ with $1/\infty := 0$
\begin{align}\label{eq:bound_ri_second}
\sup_{\lambda \in \Lambda} |\lambda(e_i)| &\leq \frac{|\lambda_{i+1}(e_i)|}{P_{\Lambda_i}(\lambda_{i+1})^{1-1/\beta}} \cdot \sup_{\lambda \in \Lambda} P_{\Lambda_i}(\lambda)^{1 - 1/\beta}.
\end{align}
\end{enumerate}
\end{lemma}

\noindent Now also the main result of \cite[Section 4]{wenzel2022analysis} can be adapted.

\begin{theorem}[Generalization of Theorem 8 from \cite{wenzel2022analysis}] \label{th:final_result}
  Any PDE-$\beta$-greedy algorithm with $\beta \in [0,\infty]$ applied to a
  function $u \in \ns$ satisfies the following error bound for $n\geq 1$:
\begin{enumerate}[label={\alph*)}]
\item In the case of $\beta \in [0, 1]$
\begin{align} \label{eq:final_result_1}
\left[ \prod_{i=n+1}^{2n} \sup_{\lambda \in \Lambda} |\lambda(e_i)| \right]^{1/n} \leq n^{-\beta/2} \cdot \Vert e_{n+1} \Vert_{\ns} \cdot \left[ 
\prod_{i=n+1}^{2n} P_{\Lambda_i}(\lambda_{i+1}) \right]^{1/n},
\end{align}
\item In the case of $\beta \in (1, \infty]$: %
\begin{equation}
\begin{aligned} \label{eq:final_result_2}
\left[ \prod_{i=n+1}^{2n} \sup_{\lambda \in \Lambda} |\lambda(e_i)| \right]^{1/n} \leq n^{-1/2} &\cdot \Vert e_{n+1} \Vert_{\ns} \cdot \left[ \prod_{i=n+1}^{2n} 
P_{\Lambda_i}(\lambda_{i+1})^{1/\beta} \right]^{1/n}. 
\end{aligned}
\end{equation}
\end{enumerate}
\end{theorem}

We remark that also further results from \cite{wenzel2022analysis} can be extended to this generalized setting, e.g.\ \cite[Corollary 9]{wenzel2022analysis} which proves an improved standard power function estimate.  

%% file: Figures/number_line_tikz.tex
\begin{center}
\begin{tikzpicture}[>=latex, thick]
\draw[->] (0,0) -- (11cm,0) node [right] {$\R$};
\draw (1.5,-4pt) -- (1.5,4pt) node[below=10pt]{$0$};
\draw (1.5,5pt) node[above, align=center, fill=lightgray, rounded corners, inner sep=2pt]{PDE-$P$- \\ greedy};
\draw (4.5,-2pt) -- (4.5,2pt);
\draw (4.5, 5pt) node[above, align=center, fill=lightgray, rounded corners, inner sep=2pt]{PDE-$f$- \\greedy} node[below=10pt]{$1$};
\draw (7.5,-2pt) -- (7.5,2pt);
\draw (7.5,5pt) node[above right, align=center, rounded corners, inner sep=2pt]{$\text{limit} ~ \beta \rightarrow \infty$  \\ PDE-$f/P$-greedy} node[below=10pt]{$\beta$};
\end{tikzpicture}
\end{center}

%% file: chapters/05_conv_rates_pde_greedy.tex
\section{Convergence rates for PDE-greedy algorithms}
\label{sec:conv_rates_pde_greedy}

This section serves as an amalgamation of the results from the previous sections: 
We apply the analysis of the PDE-$\beta$-greedy algorithms from \Cref{sec:analysis_fgreedy} to concrete kernels and PDEs. 
By using estimates on the corresponding Kolmogorov $n$-widths from \Cref{sec:analysis_kolm_width} we obtain convergence rates for PDE-greedy algorithms
in terms of the error in the PDE functionals. 
Using the maximum principle for elliptic operators (see Eq.~ \eqref{eq:maximum_principle_elliptic_operators}) we can transfer the obtained convergence rates directly to convergence rates of the approximation of the true solution.

\begin{theorem} 
\label{th:main_theorem_conv_to_solution}
We consider the BVP as in Eq.~\eqref{eq:pde}
under \hyperref[ass:assumptions]{Assumption 1} and there exists a solution in $H^\tau(\Omega)$, and we assume that $k$ is a Sobolev kernel of smoothness $\tau > d/2 + 2$, i.e.\ $\ns \asymp H^\tau(\Omega)$.

Then any PDE-$\beta$-greedy algorithm with $\beta \in [0,\infty]$
satisfies the following error bound for
$n\geq 3$
(using $1/\infty = 0$):
\begin{equation}
\label{eq:decay_result_1}
\min_{n+1\leq i\leq 2n}\Vert e_i \Vert_{L^\infty(\Omega)} \leq C \cdot  n^{-\frac{\min\{1, \beta\}}{2}} (\log(n)\cdot n^{-1})^{\frac{\alpha}{\max\{1, \beta\}}} 
\Vert e_{n+1} \Vert_{\ns},
\end{equation}
with $C:=\left(2^{\alpha+1/2} \max\{1, C_0\} e^\alpha\right)^{\frac{1}{\max\{1, \beta\}}}$ and $\alpha := \frac{\tau - 2}{d} - \frac{1}{2} > 0$.
In particular
\begin{align*}
\min_{n+1\leq i\leq 2n} \Vert e_i \Vert_{L^\infty(\Omega)} &\leq C \cdot \log(n)^{\alpha} \cdot \Vert e_{n+1} \Vert_{\ns} \cdot \left\{
\begin{array}{ll}
n^{-\alpha-1/2} & \text{PDE}-f-\text{greedy} \\
n^{-\alpha-1/4} & \text{PDE}-f \cdot P-\text{greedy} \\
n^{-\alpha} &     \text{PDE}-P-\text{greedy}
\end{array}
\right..
\end{align*}
\end{theorem}

Note that the theorem also could be formulated for $n\geq1$ by enlarging
the constant $C$.

\begin{proof}
The proof is structured into several steps:
\begin{itemize}
\item Since $k$ is a kernel of finite smoothness such that $\ns \asymp H^\tau(\Omega)$, \Cref{th:kolmogorov_estimate_Lambda} gives a decay of the corresponding Kolmogorov $n$-width $d_n(\Lambda)$ of
\begin{align*}
d_n(\Lambda) \leq C \cdot n^{\frac{1}{2} - \frac{\tau - 2}{d}} \equiv C \cdot n^{-\alpha}.
\end{align*}
This bound satisfies the assumptions of \Cref{cor:decay_powerfunc_quantity}, such that it can be applied to derive a decay of the geometric mean of subsequent power function values as
\begin{align} %
\left(\prod_{i=n+1}^{2 n} P_{\Lambda_i}(\lambda_{i+1}) \right)^{1/n} 
&\leq2^{\alpha+1/2} \tilde{C}_0 e^\alpha  \log(n)^\alpha n^{-\alpha}, \;\;n\geq 3,
\end{align}
for any choice of functionals, thus especially for the $\beta$-greedy choices.
\item Now we apply \Cref{th:final_result}, especially Eq.~\eqref{eq:final_result_1} for $\beta \in [0, 1]$ and Eq.~\eqref{eq:final_result_2} for $\beta \in (1, \infty]$ and insert the bound on the geometric mean of the generalized power function values from the previous step. This gives for $\beta \in [0,1],n\geq 3$
\begin{align*}
\left[ \prod_{i=n+1}^{2n} \sup_{\lambda \in \Lambda} |\lambda(e_i)| \right]^{1/n} 
&\leq n^{-\beta/2} \cdot \Vert e_{n+1} \Vert_{\ns} \cdot \left[ \prod_{i=n+1}^{2n}
P_{\Lambda_i}(\lambda_{i+1}) \right]^{1/n} \\
&\leq n^{-\beta/2} \cdot \Vert e_{n+1} \Vert_{\ns} \cdot 2^{\alpha+1/2} \tilde{C}_0 e^\alpha  \log(n)^\alpha n^{-\alpha},
\end{align*}
and for $\beta \in (1, \infty]$ (using $1/\infty = 0$) and again $n\geq 3$
\begin{align*}
\left[ \prod_{i=n+1}^{2n} \sup_{\lambda \in \Lambda} |\lambda(e_i)| \right]^{1/n} 
&\leq n^{-1/2} \cdot \Vert e_{n+1} \Vert_{\ns} \cdot \left[ \prod_{i=n+1}^{2n} 
P_{\Lambda_i}(\lambda_{i+1})^{1/\beta} \right]^{1/n} \\
&\leq n^{-1/2} \cdot \Vert e_{n+1} \Vert_{\ns} \cdot \left( 2^{\alpha+1/2} \tilde{C}_0 e^{\alpha} \right)^{1/\beta} \log(n)^{\alpha/\beta} n^{-\alpha/\beta}.
\end{align*}
In both cases, the left hand side can be estimated with help of
\begin{equation*}
\min_{n+1 \leq j \leq 2n}\sup_{\lambda \in \Lambda} |\lambda(e_j)| \leq \sup_{\lambda \in \Lambda} |\lambda(e_i)|
\end{equation*}
for all $i=n+1, \dots, 2n$, which finally yields
\begin{align*}
\min_{n+1 \leq j \leq 2n} \sup_{\lambda \in \Lambda} |\lambda(e_j)|
\leq C' \cdot  n^{-\frac{\min\{1, \beta\}}{2}} (\log(n)\cdot n^{-1})^{\frac{\alpha}{\max\{1, \beta\}}} 
\Vert e_{n+1} \Vert_{\ns}
\end{align*}
with $C'= (2^{\alpha+\frac12} \tilde C_0 e^\alpha)^{\frac{1}{\max(1,\beta)}}$.
\item Now we make use of the maximum principle from Eq.~\eqref{eq:maximum_principle_elliptic_operators} which is available due to our assumptions on the BVP. Therefore it holds
\begin{align*}
\Vert u - s_n \Vert_{L^\infty(\Omega)} 
&\leq C ( \Vert u - s_n \Vert_{L^\infty(\partial \Omega)} + \Vert Lu - Ls_n \Vert_{L^\infty(\Omega)}) \\
&= C \cdot \sup_{\lambda \in \Lambda_L \cup \Lambda_B} |\lambda(u - s_n)|,\\
\Rightarrow \min_{j=n+1, \dots, 2n} \Vert u - s_j \Vert_{L^\infty(\Omega)} &\leq C \cdot \min_{j=n+1, \dots, 2n} \sup_{\lambda \in \Lambda_L \cup \Lambda_B} |\lambda(u - s_j)|.
\end{align*}
Since $e_j = u - s_j$ we can directly combine this last equation with the result of the previous step to obtain
\begin{align*}
  \min_{j=n+1, \dots, 2n} \Vert u - s_j \Vert_{L^\infty(\Omega)} \leq&
                                                                    CC'\cdot n^{-\frac{\min\{1, \beta\}}{2}} \cdot \\
&\quad (\log(n)\cdot n^{-1})^{\frac{\alpha}{\max\{1, \beta\}}} 
\Vert e_{n+1} \Vert_{\ns},
\end{align*}
and this concludes the proof.
\end{itemize}

\end{proof}

For $\beta \in (0, 1]$ it can be observed that we obtain a faster rate of convergence in the number of selected functionals $\lambda_i \in \Lambda_L \cup \Lambda_B$ than for $\beta = 0$ 
and therefore in the number of selected collocation points $x_i \in \Omega \cup \partial \Omega$.
This additional decay of $n^{-\beta/2}$ increases with increasing $\beta \in (0, 1]$, 
which corresponds to a more ``target-data dependent'' choice of the next
interpolation point according to the definition of PDE-$\beta$-greedy algorithms,
see \Cref{def:beta_greedy_algorithms}.
In particular the proven decay for PDE-$f$-greedy ($\beta = 1$) is better than the proven decay for the PDE-$P$-greedy.
This makes intuitively sense, since the PDE-$f$-greedy choice is adapted to the right hand side of the BVP, while the PDE-$P$-greedy is independent of the right hand side and only
worst-case optimal \cite{schaback2019greedy}.
It is furthermore optimal in the sense of realizing the same decay rate as the
Kolmogorov $n$-width (which is based on the abstract result of \cite{devore2013greedy}).
A visualization of the different point distributions can be found in
\Cref{subsec:smooth_case}.
Moreover, for $\beta\in(0,1]$ the new estimate~\eqref{eq:decay_result_1} breaks the curse of dimensionality.
Indeed, since $\tau>d/2 +2$ by assumption, we can only guarantee $\alpha = \frac{1}{2} - \frac{\tau - 2}{d}<0$ and thus the term 
$(\log(n) \cdot n^{-1})^{\alpha}$ may yield arbitrarily slow convergence, unless additional smoothness is taken into consideration. On the contrary, the term 
$n^{-\beta/2}$ is converging to zero at a speed that is independent on the dimension $d$ of the space. 
This breaking of the curse of dimensionality is theoretically and practically appealing for high dimensional PDE problems.
Note that the constant $C$ within Eq.~\eqref{eq:decay_result_1} usually depends on the dimension $d$, and in particular might grow in $d$.
Also the factor $n^{-\alpha}$ depends on the dimension $d$, especially $\alpha$ decreases in $d$.
We point to \cite{rieger2024approximability} for a discussion on the curse of dimensionality within kernel-based approximation.

As already remarked in \cite{wenzel2022analysis}, we conjecture that the analysis so far is suboptimal. 
In particular in view of \cite{santin2024optimality} we do not expect that the
additional $\log(n)^\alpha$ factor is really required, which occurs for $\beta > 0$.

Finally we want to remark that the whole analysis also works for PDEs on manifolds, e.g.\ \Cref{th:final_result} holds for both $\Omega$ a domain or a manifold.
The analysis of target-data dependent greedy kernel algorithms works in very general terms, 
because it simply leveraged the abstract analysis of \cite{devore2013greedy} and \cite[Section 3]{wenzel2022analysis} for Hilbert spaces and connects it to kernel approximation in RKHS.

In order to derive practically usable convergence rates, the Kolmogorov $n$-widths need to be analyzed (\Cref{sec:analysis_kolm_width}), and to derive those for manifolds, sampling inequalities and corresponding zero lemmas for manifolds can be leveraged, as e.g.\ in \Cref{th:sampling_inequ_manifolds}.
In order to transfer the resulting convergence rates on $\sup_{\lambda \in \Lambda} |\lambda(e_i)|$ to convergence rates on $e_i \rightarrow 0$ in some norm, PDE solution theory is required, e.g.\ some well-posedness estimate or a maximum principle.

\begin{rem}
Similar statements as in \Cref{th:main_theorem_conv_to_solution} can be derived for the case of analytic kernels such as the Gaussian kernel, but we do 
not consider them here since their RKHSs are difficult to describe~\cite{steinwart2006explicit}, and especially they are usually subsets of 
Sobolev spaces of arbitrary orders. 
This means one usually expects an exponential rate of convergence, 
such that the additional target-data dependent decay rate as $(\log(n) \cdot n^{-1})^\alpha$ is not that appealing.

However, it should be noted that already in the case $\beta=0$ no converge rates of the various PDE-greedy methods with these kernels are known in the 
literature. 
Moreover, even if interpolation with these kernels usually gives spectral convergence that scales as $\exp(-n^{1/d})$, for $\beta>0$ the additional 
$\beta$-dependent term is still mitigating the curse of dimensionality, and in particular the dimension-independent rate of $-1/2$ is 
in line with the rates of interpolation obtained in \cite{Fasshauer2012c} for the Gaussian kernel.

\end{rem}

%% file: chapters/07_NEW_numerical_experiments.tex
\section{Numerical experiments} \label{sec:numerical_exp}

The implementation of the PDE greedy algorithms is coined PDE-VKOGA, in analogy to the Vectorial Kernel Orthogonal Greedy Algorithm (VKOGA) for inteporlation, see \cite{Wirtz2013}, and is provided via a corresponding git repository\footnote{\url{https://gitlab.mathematik.uni-stuttgart.de/pub/ians-anm/pde-vkoga}}.
The code for reproducing the subsequent numerical experiments is also freely available\footnote{\url{https://gitlab.mathematik.uni-stuttgart.de/pub/ians-anm/paper-2024-pde-greedy}}.

In order to better balance between the importance of the functionals in the interior $\Omega$ and on the
boundary $\partial \Omega$,
we introduce a corresponding weighting into the PDE-greedy selection criteria and thus use
\begin{align*}
\lambda_{n+1} = \argmax_{0 \neq \lambda \in \Lambda \setminus \Lambda_n}
\begin{cases}
|\lambda(u-\Pi_{V_n}(u))|^\beta \cdot P_{\Lambda_n}(\lambda)^{1-\beta} & \lambda \in \Lambda_L \\
w \cdot |\lambda(u-\Pi_{V_n}(u))|^\beta \cdot P_{\Lambda_n}(\lambda)^{1-\beta} & \, \lambda \in \Lambda_B.
\end{cases}
\end{align*}

As the greedy approach can be interpreted
as a pivoted partial Cholesky decomposition of the kernel collocation matrix (see~\cite{Pazouki2011}),
the choice of the pivoting elements is influenced by this weighting.

As a first experiment we choose a low-dimensional case of $d=2$ with a smooth solution which allows
comparisons of the different greedy variants with the finite element method. 
The second experiment adresses the case of solutions in $d=2$ with singularities, which are
more challenging for the present approach. 
As final setting we demonstrate the potential of the greedy procedures for high-dimensional problems where
$d=12$.

\subsection{Smooth case for $d=2$}
\label{subsec:smooth_case}

As domain we consider the sector of the 2D unit ball with opening
angle $\alpha \pi$ for $\alpha= \frac23$ as visualized in \Cref{fig:smooth_2D_vis_collocation_points}. %
We consider the boundary value problem
\begin{alignat*}{2}
-\Delta u &= f, &&\text{on } \Omega, \\
u &= g, \qquad &&\text{on } \partial \Omega,
\end{alignat*}
which has a
smooth solution $u(x) = \Vert x \Vert^2$ for source function $f(x) = -4$
and Dirichlet-boundary values $g(x) = u(x)$.

As a kernel, we use the cubic Matérn kernel, 
which is a translational invariant kernel with $\Phi(x - y) = \varphi(\Vert x - y \Vert)$, 
$\varphi(r) = (15 + 15r + 6r^2 + r^3) \cdot \exp(-r)$.
The sets 
$\Omega$ and $\partial \Omega$ were uniformly randomly
sampled using $50000$ and $1000$ points, respectively.
The $L^2(\Omega)$ and $L^\infty(\Omega)$ errors were numerically
approximated using a uniform grid of mesh width $2\cdot 10^{-3}$
(approximately 250000 points).

We consider PDE-$P$-, PDE-$f \cdot P$- and PDE-$f$-greedy, using weight factors of $w \in \{10^0, 10^1, ..., 10^8\}$.
The results are displayed in \Cref{tab:smooth_2D_errors} (for $w=10^6$), 
\Cref{fig:smooth_2D_weight_dependence} and \Cref{fig:smooth_2D_vis_collocation_points}:
First we compare the approximation of the three PDE-greedy methods to a classical P1 finite element method (FEM),
which uses a triangular grid with polygonal approximation of the sector and different global refinement levels.
\Cref{tab:smooth_2D_errors} lists the $L^2(\Omega)$ and the $L^\infty(\Omega)$ approximation error for $u - s_n$ for several expansion sizes, which match the number of degrees of freedom for the FEM solution.
Note that we do not compute the kernel approximants for more than 20000 degrees of freedom, 
as our greedy method and its current implementation are most efficient for smaller expansion sizes.
For PDE-$f \cdot P$-greedy and PDE-$f$-greedy, a stability related stopping criterion limited the expansion sizes to 2061 (PDE-$f \cdot P$-greedy) respectively 1229 (PDE-$f$-greedy).
The corresponding numbers were marked with a star.

Note that the occuring kernel matrices are full, while the corresponding FEM matrices are sparse.
While this can be seen as a weakness of the kernel-based approach, this limited number of points
suffices to reach a comparably high accuracy. 
Moreover, they work well in high dimensions where FEM is infeasible, see also \Cref{subsec:high_dim_example}.

In any case, the final FEM accuracy using 20193 DOFs is already satisfied by the kernel models using 341 collocation points.
This clearly demonstates that both  PDE-$P$-greedy and PDE-$f$-greedy outperform the FEM approximation. 
\Cref{fig:smooth_2D_weight_dependence} gives an overview on the influence of the weighting parameter $w$ for PDE-$f$-, PDE-$f \cdot P$- and PDE-$P$-greedy. 
For all the cases, the plots show the decay of the errors $\sup_{\lambda \in \Lambda} |\lambda(u-s_n)|$ and $\sup_{x \in \Omega} |(u-s_n)(x)|$ in the number of collocation points.
Given a suitable weighting (of e.g.\ $w = 10^6$), all three PDE-greedy methods provide approximately the same rate of decay for the $\sup_{x \in \Omega} |(u-s_n)(x)|$ error (right column).
However in terms of the $\sup_{\lambda \in \Lambda} |\lambda(u-s_n)|$ decay,
the PDE-$f$-greedy method provides the faster rate of decay (left column).
\Cref{fig:smooth_2D_vis_collocation_points} visualizes the selected collocation points for both PDE-$f$-greedy and PDE-$P$-greedy for using a weighting of $w=10^6$. 
One can observe that PDE-$f$-greedy chooses more collocation points close to the boundary,
while the PDE-$P$-greedy collocation points are rather uniformly distributed, which can be even proved in the case of interpolation (instead of collocation), see \cite{wenzel2021novel}.

\begin{table}[h]
    \centering
    \caption{Numerical results regarding \Cref{subsec:smooth_case}: 
       Errors for \textbf{smooth} solution on sector example.
       The results marked with $^*$ were obtained using smaller expansion sizes as elaborated in the main text.
	}
    \begin{tabular}{|l||c|c||c|c|c|}
    \hline
    & \textbf{\# DOFs} & 341 & 1305 & 5105 & 20193 \\ \hline
    \multirow{ 2}{*}{FEM} & $L^2(\Omega)$ & 1.33e-3 & 3.35e-4 & 8.38e-5 & 2.10e-5 \\ 
    & $L^\infty(\Omega)$ & 2.18e-3 & 6.06e-4 & 1.64e-4 & 4.40e-5 \\ \hline
    \multirow{ 2}{*}{PDE-$P$-greedy}    & $L^2(\Omega)$ & 3.38e-7 & 2.05e-8 & 4.44e-9 & - \\ 
    & $L^\infty(\Omega)$ & 1.89e-6 & 1.58e-7 & 1.02e-8 & - \\ \hline
    \multirow{ 2}{*}{PDE-$f \cdot P$-greedy}        & $L^2(\Omega)$ & 3.44e-7 & 2.08e-8 & 5.92e-9$^*$ & - \\ 
    & $L^\infty(\Omega)$ & 7.90e-7 & 4.21e-8 & 1.29e-8$^*$ & - \\ \hline
    \multirow{ 2}{*}{PDE-$f$-greedy}        & $L^2(\Omega)$ & 1.04e-6 & 7.24e-8$^{*}$ & - & - \\ 
    & $L^\infty(\Omega)$ & 2.44e-6 & 1.50e-7$^{*}$ & - & - \\ \hline
    \end{tabular}
    \label{tab:smooth_2D_errors}
\end{table}

\begin{figure}[h]
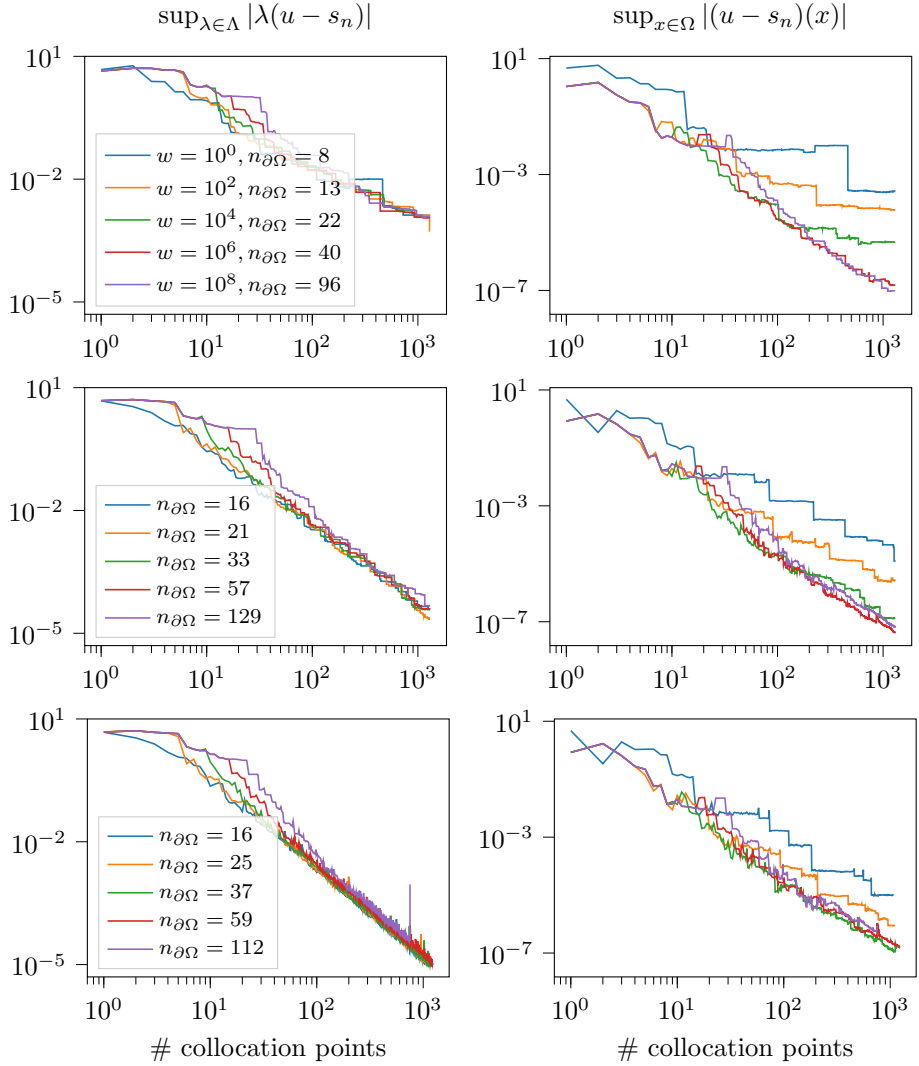

\centering
\setlength\fwidth{.55\textwidth}
\input{Figures/smooth_2D_beta_0_lambda.tex}
\input{Figures/smooth_2D_beta_0_nolambda.tex}

\input{Figures/smooth_2D_beta_0.5_lambda.tex}
\input{Figures/smooth_2D_beta_0.5_nolambda.tex}

\input{Figures/smooth_2D_beta_1_lambda.tex}
\input{Figures/smooth_2D_beta_1_nolambda.tex}
\caption{Numerical results regarding \Cref{subsec:smooth_case}: 
  Visualization of the decay of the errors 
$\sup_{\lambda \in \Lambda} |\lambda(e_i)|$
  (left) and $\sup_{x \in \Omega} |e_i(x)|$
  (right) for PDE-$P$-greedy (top), PDE-$f \cdot P$-greedy (middle), PDE-$f$-greedy (bottom) for different values of the weighting parameter $w$.
For each of the three PDE-greedy methods, the number of selected collocation points $n$ which are selected on the boundary is denoted by $n_{\partial \Omega}$.
}
\label{fig:smooth_2D_weight_dependence}
\end{figure}

\begin{figure}[h]
    \centering
    \includegraphics[width=0.48\textwidth]{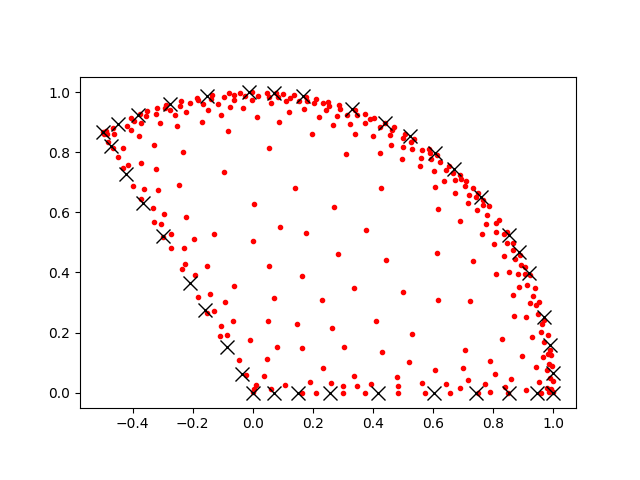}
    \includegraphics[width=0.48\textwidth]{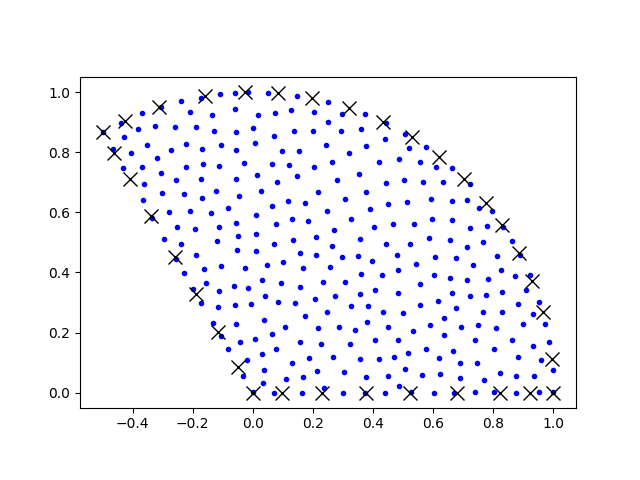}
    \caption{Numerical results regarding \Cref{subsec:smooth_case}: 
    Visualization of the distribution of 341 collocation points selected by PDE-$f$-greedy (left) and PDE-$P$-greedy (right) for the smooth solution.
    Collocation points on the boundary are visualized as crosses.
    On can observe that the PDE-$f$-greedy selected centers cluster adaptively next to the boundary,
    while the PDE-$P$-greedy centers are rather uniformly distributed.}
    \label{fig:smooth_2D_vis_collocation_points}
\end{figure}

\subsection{Singular case for $d=2$}
\label{subsec:singular_case} 

\begin{table}[h]
    \centering
    \caption{Numerical results regarding \Cref{subsec:singular_case}: 
       Errors for \textbf{singular} solution on sector example.
       The results marked with $^*$ were obtained using smaller expansion sizes as elaborated in the main text.
	}
    \begin{tabular}{|l||c|c||c|c|c|}
    \hline
    & \textbf{\# DOFs} & 341 & 1305 & 5105 & 20193 \\ \hline
    \multirow{ 2}{*}{FEM} & $L^2(\Omega)$ & 2.50e-4 & 6.31e-5 & 1.58e-5 & 3.97e-6 \\ 
    & $L^\infty(\Omega)$ & 2.39e-3 & 8.45e-4 & 2.95e-4 & 1.01e-4 \\ \hline
    \multirow{ 2}{*}{PDE-$P$-greedy} & $L^2(\Omega)$ & 9.04e-5 & 5.86e-5 & 4.71e-5 & - \\ 
    & $L^\infty(\Omega)$ & 1.23e-3 & 1.04e-3 & 1.38e-3 & - \\ \hline
    \multirow{ 2}{*}{PDE-$f$-greedy}        & $L^2(\Omega)$ & 2.39e-3 & 8.81e-5 & 1.46e-4$^{*}$ & - \\ 
    & $L^\infty(\Omega)$ & 5.11e-3 & 2.32e-4 & 3.58e-4$^{*}$ & - \\ \hline
    \end{tabular}
    \label{tab:singular_2D_errors}
\end{table}

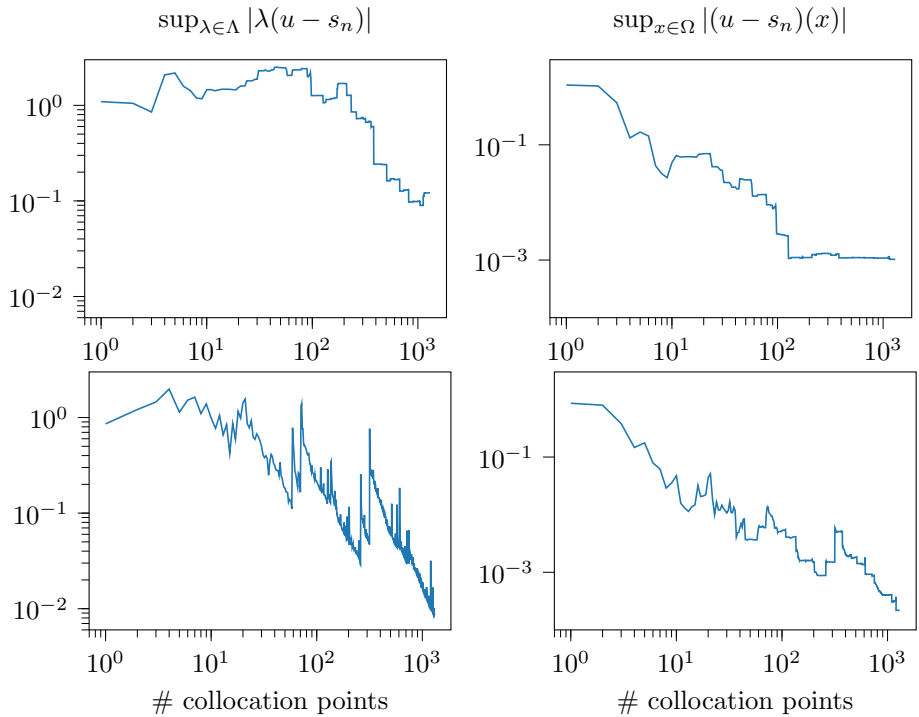
\begin{figure}[h]
\centering
\setlength\fwidth{.55\textwidth}
\input{Figures/vis_decay_pacman_singular_p_lambda.tex}
\input{Figures/vis_decay_pacman_singular_p_sol.tex}

\input{Figures/vis_decay_pacman_singular_f_lambda.tex}
\input{Figures/vis_decay_pacman_singular_f_sol.tex}
\caption{
Numerical results regarding \Cref{subsec:singular_case}:
Visualization of the decay of the errors $\sup_{\lambda \in \Lambda} |\lambda(u - s_n)|$ (left) and $\sup_{x \in \Omega} |(u-s_n)(x)|$ (right) over the expansion size $n$ ($x$-axis) for PDE-$P$-greedy (top) using $w = 10^6$ %
and PDE-$f$-greedy (bottom) using $w=10^3$.
\label{fig:singular_2D_weight_dependence}
}
\end{figure}

\begin{figure}[h]
    \centering
    \includegraphics[width=0.48\textwidth]{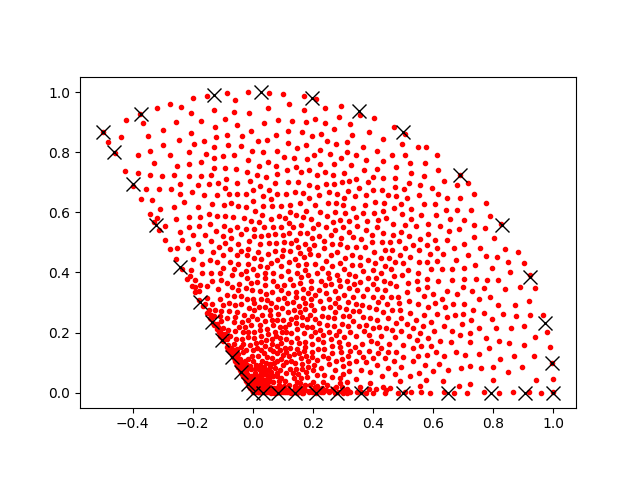}
    \includegraphics[width=0.48\textwidth]{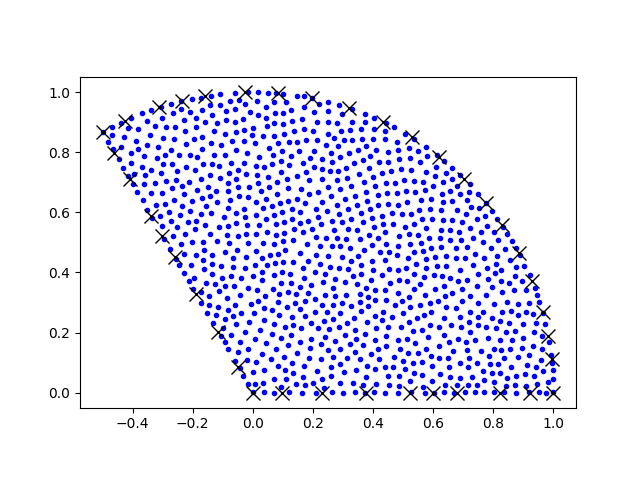}
    \caption{
    Numerical results regarding \Cref{subsec:singular_case}:    
    Visualization of the distribution of 1305 collocation points selected by PDE-$f$-greedy (left) and PDE-$P$-greedy (right) for the singular solution.
    Collocation points on the boundary are visualized as crosses.
    On can observe that the PDE-$f$-greedy selected centers cluster adaptively next to the singularity,
    while the PDE-$P$-greedy centers are rather uniformly distributed.}
    \label{fig:singular_2D_vis_collocation_points}
\end{figure}

As next example we consider the identical geometry and PDE as in the previous section, 
but now with data-functions that result in a solution with singularity in its derivatives,
i.e.\ the solution is 
\begin{align*}
u(x) = \Vert x \Vert^{1/\alpha}\sin(\phi(x)/\alpha)
\end{align*}

with $\phi(x) = \arctan(x_2/x_1)$, which is the solution of the PDE for $f(x) = 0$ and $g(x)=u(x)$.  
We apply the same kernel, sampling and greedy procedures as in the previous section. 

Instead of showing the results for all the weightings, we focus on suitable weightings which are given as $w = 10^3$ for the PDE-$f$-greedy and $w=10^6$ for the PDE-$P$-greedy.
The results are displayed in \Cref{tab:singular_2D_errors}, \Cref{fig:singular_2D_weight_dependence}
and \Cref{fig:singular_2D_vis_collocation_points}: As before, in \Cref{tab:singular_2D_errors} we show errors in both the $L^2(\Omega)$ and $L^\infty(\Omega)$ norms, 
including a comparison to a classical P1 FEM.
In this case of a singular solution, the PDE-greedy methods do no longer outperform the FEM, though still match the accuracy of the FEM.
For the PDE-$f$-greedy, the maximal expansion size was restricted to 1621 due to stability reasons, 
and the corresponding numbers are marked with a star.
We remark that the performance of the PDE-greedy method deteriorate further, when the singularity is intensified, e.g.\ by increasing the opening angle $\alpha$ towards $2$.
In \Cref{fig:singular_2D_weight_dependence} we observe that we still obtain convergence in the residual error and a bit
slower in the collocation functional values.
This convergence actually is a nice result, as it demonstrates that we can even approximate
some target functions which are lying outside of the RKHS of the chosen kernel. 
But we observe a clear quantitative difference compared to the smooth example, which is that the
errors only decay about half the orders of magnitude. Finally, 
\Cref{fig:singular_2D_vis_collocation_points} visualizes the selected collocation points for PDE-$f$-greedy 
and PDE-$P$-greedy. %
One can clearly observe a clustering of the PDE-$f$-greedy collocation points towards the singularity,
while the PDE-$P$-greedy collocation points are again uniformly distributed.

\subsection{High-dimensional PDE for $d=12$}
\label{subsec:high_dim_example}

In order to show the applicability of the proposed PDE-greedy methods in high dimensional problems,
we consider the linear Poisson problem
$-\Delta u(x) = -2d$ on the scaled unit hypercube $\Omega = d^{-1/2} \cdot [0,1]^d$ in dimension $d = 12$ with Dirichlet data on $\partial \Omega$ given by the function $g(x) = 1 + \sum_{i=1}^d x_i^2$.
Like this, the exact solution to the corresponding BVP is given by $u(x) = 1 + \sum_{i=1}^d x_i^2$.

We approximate the solution by approximants $s_n$ using PDE-$\beta$-greedy
with $\beta=1$, i.e. PDE-$f$-greedy and the cubic Mat\'ern kernel.

For the numerical experiment, the boundary $\partial \Omega$ was discretized using $9600$ uniformly
randomly sampled training points, 
while $10^5$ randomly sampled training points were used for the interior $\Omega$.
To evaluate the error in $\Omega$ independently of these training points,
we used further $10^5$ randomly sampled test points.

\Cref{fig:fgreedy_weights}
visualizes the resulting decay of the error $\Vert u - s_n \Vert_{L^\infty(\Omega)}$ as well as $\sup_{\lambda \in \Lambda} |\lambda(u-s_n)|$ over the expansion size $n$.
Overall the weighting of interior vs.\ boundary points seems to play an essential role
in the quality of the approximate solution.
In principle this weighting could as well be interpreted as an additional
hyperparameter of the greedy schemes, which could be chosen based on
suitable model selection procedures in order to further optimize the results. 
Independent of this dependence of the weighting, the error decay curves
clearly indicate that the scheme is applicably to higher-dimensional geometrical input
domains, where classical methods such as FEM are not applicable. 

\begin{figure}[h]
\centering
\setlength\fwidth{.54\textwidth}	%
\input{Figures/vis_highdim_decay_Lerror.tex}
\input{Figures/vis_highdim_decay_error.tex}
\caption{Numerical results regarding \Cref{subsec:high_dim_example}:
  Visualization of the approximation error
  $\sup_{\lambda \in \Lambda} |\lambda(u-s_n)|$
  and
  $\Vert u - s_n \Vert_{L^\infty(\Omega)}$
  (evaluated on $10^5$ randomly sampled test points) 
over the expansion size $n$ ($x$-axis) for PDE-$f$-greedy on a Poisson problem in $\Omega \subset \R^{12} $ for different weightings $w$ of boundary vs.\ interior values.
\label{fig:fgreedy_weights}}
\end{figure}
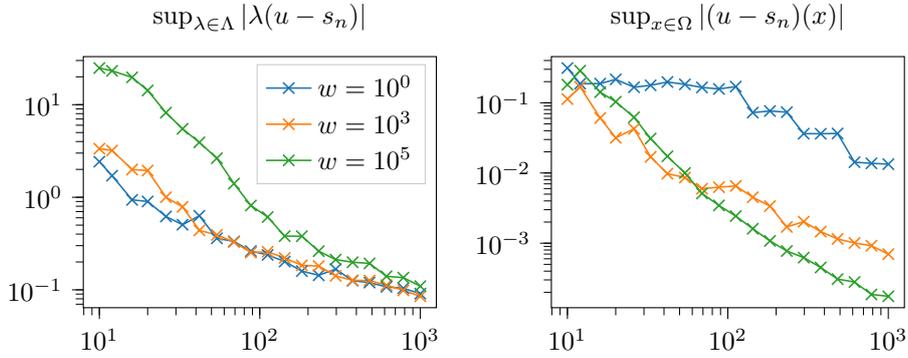

%% file: Figures/vis_decay_pacman_singular_p_lambda.tex
\begin{tikzpicture}

\definecolor{darkgray176}{RGB}{176,176,176}
\definecolor{steelblue31119180}{RGB}{31,119,180}

\begin{axis}[
title={$\sup_{\lambda \in \Lambda} |\lambda(u - s_n)|$},
width=0.951\fwidth,
height=0.75\fwidth,
log basis x={10},
log basis y={10},
tick align=outside,
tick pos=left,
x grid style={darkgray176},
xmin=0.698612127441536, xmax=1866.55792073854,
xmode=log,
xtick style={color=black},
y grid style={darkgray176},
ymin=6e-3, ymax=3,
ymode=log,
ytick style={color=black}
]
\addplot [semithick, steelblue31119180]
table {%
0 1.99999986283073
1 1.09256390794118
2 1.0498932483441
3 0.851338525818038
4 2.08156980775781
5 2.17925351973239
6 1.58780238793352
7 1.42651107582425
8 1.19044888704191
9 1.17199742804758
10 1.45671871549448
11 1.46035399711466
12 1.42076748934652
13 1.45337487023707
14 1.47430038184387
15 1.4729043658262
16 1.47278829454359
17 1.46707404846786
18 1.45662049346055
19 1.45768490644084
20 1.54128641612991
21 1.59175595720147
22 1.59369802582221
23 1.59738111419857
24 1.8014373075796
25 1.80458064905594
26 1.80844347848296
27 1.80703643665517
28 1.84816022254789
29 1.87598237899194
30 1.87517947621926
31 2.30200150523467
32 2.29849646628409
33 2.29983462400144
34 2.30219940104412
35 2.30182909613723
36 2.34797095617157
37 2.29652147191227
38 2.28333460646061
39 2.28323481900718
40 2.29556438275391
41 2.35332865919819
42 2.35053661886657
43 2.34941652959752
44 2.49662565786483
45 2.50781307000065
46 2.50778798829129
47 2.5074307117415
48 2.48321134987664
49 2.4832035827026
50 2.48258600680576
51 2.47243471743287
52 2.47231937578175
53 2.47214128566821
54 2.47171326066893
55 2.46949117301501
56 2.46948099619983
57 2.35270926682694
58 2.05558676608015
59 2.05558532247057
60 2.05557452145904
61 2.05490254492582
62 2.05782368953002
63 2.05818081505786
64 2.05883772094338
65 2.34368101879892
66 2.34368163255124
67 2.34366330087544
68 2.35008810908154
69 2.35050419632848
70 2.35328737222927
71 2.35341546592807
72 2.34760191397605
73 2.34834459108048
74 2.34834752600683
75 2.34792005587445
76 2.34800460514091
77 2.34800454869182
78 2.3480105584785
79 2.42075130102166
80 2.42056648928533
81 2.42055933075886
82 2.42065885666337
83 2.42074136010892
84 2.4178306102218
85 2.41788981802098
86 2.41859073272718
87 2.4186522425813
88 2.41868486500557
89 2.41872306189465
90 1.99304580156117
91 1.99304448858468
92 1.99345945778619
93 2.06556866410818
94 2.06557641376088
95 2.0657688320783
96 2.18000368579826
97 1.87872691457267
98 1.26527774156603
99 1.26532392739702
100 1.26532233740676
101 1.26532415994208
102 1.26532414139921
103 1.26349563658231
104 1.264613270176
105 1.26491986297527
106 1.26492845677949
107 1.26492689750659
108 1.26492945078205
109 1.26500037016804
110 1.26500496615726
111 1.26500683953521
112 1.26774455711558
113 1.26742802578592
114 1.26743148882152
115 1.26743134771324
116 1.26804953353505
117 1.26810694194727
118 1.26585861680554
119 1.26586153967823
120 1.26591812736402
121 1.26592549090634
122 1.26595531112811
123 1.26614461942874
124 1.26614983943084
125 1.26615038165003
126 1.26615396671286
127 1.06339015314964
128 1.06339174836434
129 1.06344853562855
130 1.06335644184973
131 1.06336426659452
132 1.07780078433009
133 1.07785041454376
134 1.07785109842393
135 1.15222172382018
136 1.15221699333835
137 1.15221714174126
138 1.15224408916483
139 1.15225870361493
140 1.15269642512278
141 1.15274676738664
142 1.15255766193191
143 1.15257447612891
144 1.15258196504343
145 1.15270090561225
146 1.15270863573747
147 1.1527201537443
148 1.15252684996593
149 1.15349743205732
150 1.1770116605682
151 1.17702080774572
152 1.17716274823496
153 1.17716273750137
154 1.17754072279873
155 1.17762030325223
156 1.17763739610567
157 1.17796184588889
158 1.17811061954763
159 1.17806557010251
160 1.17806553898783
161 1.17807179673992
162 1.17817379620475
163 1.20171395843877
164 1.20171408956919
165 1.20171239110423
166 1.20171692682061
167 1.20171585388266
168 1.19934391060194
169 1.19934600138468
170 1.19934640721841
171 1.19937827246818
172 1.48502972714957
173 1.48505846656656
174 1.66232911797688
175 1.66454869224979
176 1.66454915411389
177 1.69359669846741
178 1.69359886898302
179 1.69364830539144
180 1.69365079585672
181 1.69365112672487
182 1.69365129874364
183 1.69423651301852
184 1.69423756006643
185 1.69423826555676
186 1.69435101055256
187 1.6943541470983
188 1.69435525472744
189 1.69435618027833
190 1.69435473799936
191 1.69442264127385
192 1.69444430187199
193 1.69444430257376
194 1.69540569400973
195 1.69540919227862
196 1.69540905698559
197 1.6954333803889
198 1.69543405566416
199 1.69548555593941
200 1.69548899780685
201 1.69548930528224
202 1.69549058217461
203 1.69551504801525
204 1.6842541274814
205 1.68427583275851
206 1.68427621107031
207 1.68427659219612
208 1.684276590772
209 1.68427658353502
210 1.68434299922722
211 1.68434665616634
212 1.27042447326385
213 1.27047178550077
214 1.27047226339048
215 1.27047500357302
216 1.27047507681756
217 1.27047710766311
218 1.27048539112473
219 1.2704866149275
220 1.27048745465576
221 1.27048802848009
222 1.27052939311884
223 1.27052994353046
224 1.27053424758313
225 1.27044701228358
226 1.2704416711072
227 1.2704326577446
228 1.27728866844457
229 1.2772905375998
230 1.27729601707221
231 1.27730487958457
232 1.27730478401273
233 1.27730823097375
234 0.85386998762201
235 0.854065242406349
236 0.854069093340615
237 0.854072847415747
238 0.854047723889597
239 0.854048718893218
240 0.854050933190248
241 0.854050938537332
242 0.854079166720415
243 0.854109105839563
244 0.852268342913407
245 0.852268762839623
246 0.852268373024575
247 0.852263688594099
248 0.852486687995402
249 0.85249351879241
250 0.852501258234236
251 0.852503378778059
252 0.852500676611298
253 0.852500665435836
254 0.852474691877574
255 0.852490736764132
256 0.852490986638591
257 0.851271989303276
258 0.851271378713908
259 0.851271471475459
260 0.851271597909548
261 0.851271917232526
262 0.723830616806246
263 0.723830730919547
264 0.723831625206716
265 0.7238316303583
266 0.723905877304747
267 0.723912149077457
268 0.723912104633379
269 0.723912358017615
270 0.723945689140559
271 0.72828428992449
272 0.728284248681001
273 0.728284529766171
274 0.728045702922254
275 0.728045826845707
276 0.728046256252742
277 0.728047595142859
278 0.728051610290666
279 0.728060186412882
280 0.728263257820554
281 0.728268822369606
282 0.728268994664555
283 0.728269034251297
284 0.728291927217637
285 0.728293216189412
286 0.728294152779239
287 0.728295440767794
288 0.728297933448301
289 0.728302847947566
290 0.732367261149177
291 0.732368604454979
292 0.732368729438492
293 0.732419409601953
294 0.745293266104792
295 0.745294040513353
296 0.745308075726272
297 0.745513479939583
298 0.745513273051381
299 0.745513937397171
300 0.745514017922626
301 0.745609573815703
302 0.745609686596008
303 0.745610184028815
304 0.721157640166074
305 0.721162348006462
306 0.721195784561791
307 0.72179222775605
308 0.721792283974723
309 0.72179239025575
310 0.721794092372557
311 0.721796754105813
312 0.72179662862455
313 0.721797013844577
314 0.721797057328541
315 0.721797318987119
316 0.721797387094551
317 0.721808235506216
318 0.721808306725287
319 0.723673760513624
320 0.663876234969286
321 0.663876811169245
322 0.663876996678156
323 0.663889554802615
324 0.663925906294885
325 0.663925365283306
326 0.663925923323832
327 0.66392600208215
328 0.663930942383426
329 0.663931147636851
330 0.663931415076615
331 0.66393142212999
332 0.663931404434283
333 0.663938791731669
334 0.664000376912957
335 0.664005209744098
336 0.664005246271343
337 0.664006873049608
338 0.663999243867891
339 0.664004157229337
340 0.66402436278202
341 0.679862842665888
342 0.679865598953199
343 0.679833276982976
344 0.679833286908596
345 0.679833325933354
346 0.679833326724498
347 0.679833629456991
348 0.679833647143042
349 0.679836987239847
350 0.679837212276864
351 0.679837710338
352 0.6801274493217
353 0.673314657778857
354 0.673314685739013
355 0.673341484457919
356 0.673343037622303
357 0.673343244661807
358 0.673343389394041
359 0.673343522894439
360 0.67334515870679
361 0.598099548708172
362 0.601844060178551
363 0.601844078786973
364 0.601844320820658
365 0.601845210959261
366 0.60184524324489
367 0.601846120268195
368 0.60184654347477
369 0.598862195862086
370 0.59886897622822
371 0.598868981125246
372 0.598869567360819
373 0.598870660994513
374 0.598870668019976
375 0.598870902992655
376 0.598898649614495
377 0.598900748029021
378 0.598900741044398
379 0.598900837054231
380 0.598900878245325
381 0.598981535182459
382 0.242722605055355
383 0.242722884088275
384 0.242723100841212
385 0.242723140832396
386 0.242723138132809
387 0.24274684515956
388 0.24274664373194
389 0.242746656114787
390 0.242746659577895
391 0.242749182055376
392 0.242749217832189
393 0.242752745482388
394 0.242752864851695
395 0.242753150573641
396 0.242753365868569
397 0.242753234327799
398 0.242753300503078
399 0.242765404878309
400 0.242765428891365
401 0.242765421683405
402 0.242766932316969
403 0.24276693752375
404 0.24276709644378
405 0.242767180187866
406 0.24276718225533
407 0.242767191361161
408 0.242767437015604
409 0.242778756674497
410 0.242778779986213
411 0.242779093036061
412 0.242779185172382
413 0.242948609577309
414 0.242948784205776
415 0.242948818146244
416 0.242948871040512
417 0.242948856799386
418 0.242949043264826
419 0.242965959726374
420 0.242967768358067
421 0.242965045611323
422 0.242965584816764
423 0.242990973361764
424 0.242991336579316
425 0.242993937925538
426 0.242993976778965
427 0.242989575644112
428 0.24299213153033
429 0.242992191765674
430 0.24299219282239
431 0.2430434000439
432 0.243054931349422
433 0.243055067738515
434 0.243055082911227
435 0.243055273829434
436 0.243055909247696
437 0.24305573698988
438 0.243055838221299
439 0.24305692745756
440 0.243057094322356
441 0.243057138711286
442 0.243057143126683
443 0.243057571981199
444 0.24305772591502
445 0.243058111988243
446 0.243055666462017
447 0.242259157599614
448 0.240468886245621
449 0.24046890312961
450 0.240469129943214
451 0.240470689626291
452 0.24076232459854
453 0.240800361280979
454 0.240819918440596
455 0.240820032564228
456 0.240820037159577
457 0.240820044478114
458 0.240820052403674
459 0.240820391255222
460 0.240820864863253
461 0.240821308590056
462 0.240821309407006
463 0.240825648848288
464 0.240825695710637
465 0.2408257717278
466 0.240826230486732
467 0.240826354074875
468 0.240826395087243
469 0.24082639495202
470 0.240824286172397
471 0.240824614777979
472 0.240824710295065
473 0.240824781440831
474 0.240824816858513
475 0.240824998270743
476 0.240825453744609
477 0.240825465922029
478 0.240825476382605
479 0.239973471598494
480 0.239973471933974
481 0.239864072026413
482 0.239974123221518
483 0.239974123946708
484 0.239974155236616
485 0.23997490789804
486 0.239974937917969
487 0.239974947380359
488 0.240000753931122
489 0.2400006970201
490 0.240000771014576
491 0.240000796545461
492 0.240000896386578
493 0.240000900516154
494 0.240002137962276
495 0.240062732889501
496 0.240062922940304
497 0.24001610674025
498 0.240016654226126
499 0.240016662224125
500 0.240016580997474
501 0.240016595071568
502 0.240017167677155
503 0.240017170087486
504 0.240017217898963
505 0.240017675795199
506 0.240017675155797
507 0.162075527859177
508 0.162079422802854
509 0.162079446933118
510 0.162080656474262
511 0.162080655270741
512 0.162080654816142
513 0.162080654095626
514 0.162080659110015
515 0.162090208196649
516 0.162090357322228
517 0.162090371817629
518 0.162090538936971
519 0.162091221470748
520 0.162091884564202
521 0.162092006275213
522 0.162092007496099
523 0.163049875197242
524 0.163049904155813
525 0.16304995541859
526 0.163048312826453
527 0.163048320444221
528 0.16304849546333
529 0.163048495034386
530 0.163048541561531
531 0.163048780391532
532 0.163060629785169
533 0.163055860149918
534 0.163055861210039
535 0.163060367611365
536 0.163060423034529
537 0.163062494093099
538 0.163062505209472
539 0.163160672608737
540 0.163160673248383
541 0.163160738716372
542 0.16316073857978
543 0.162530478317611
544 0.162530329674511
545 0.162530334702117
546 0.162530544252526
547 0.162530545957196
548 0.162530971816117
549 0.162531079487793
550 0.17039701364651
551 0.170397087606216
552 0.170397064575379
553 0.170397066292828
554 0.170397070751479
555 0.170397437219982
556 0.170397585605141
557 0.170404255399654
558 0.170404256178454
559 0.170405959989446
560 0.170406375195012
561 0.170406759223614
562 0.170406759250791
563 0.17040675989428
564 0.170406888602898
565 0.17040717934803
566 0.170405697993758
567 0.170747608766388
568 0.170747792141435
569 0.170747930205629
570 0.170747973415485
571 0.170747975997568
572 0.170748113979971
573 0.17074819394609
574 0.170748880093664
575 0.170712552050568
576 0.170712555584429
577 0.170712539681433
578 0.170721185310281
579 0.170721914574723
580 0.170721917789491
581 0.170456693443431
582 0.170464304131313
583 0.170464304240784
584 0.170464441044453
585 0.170464353983117
586 0.170464358189771
587 0.170461471960545
588 0.170461773055976
589 0.170461775686192
590 0.170465715767441
591 0.170471021550387
592 0.170471022090938
593 0.170471021888891
594 0.167858294673808
595 0.167858305114799
596 0.167858324989691
597 0.167858328497237
598 0.167868734985176
599 0.167869011801028
600 0.167869030604886
601 0.167869038017114
602 0.167869079868825
603 0.16786831042674
604 0.167868322292737
605 0.167868334974035
606 0.167868343052674
607 0.167869651152117
608 0.16786971486473
609 0.167869717119308
610 0.166832363202012
611 0.166832401326454
612 0.166832388204294
613 0.166832388465422
614 0.166832571484042
615 0.16684145644149
616 0.166842439738713
617 0.166847473682694
618 0.166847476067101
619 0.166847498755807
620 0.166847503827221
621 0.166846554207844
622 0.166846663016388
623 0.166846777409774
624 0.166846900336448
625 0.16692819426133
626 0.166928194045517
627 0.16697257970319
628 0.166972578813602
629 0.16697258115606
630 0.167072706358401
631 0.167074609177027
632 0.167075063270065
633 0.167226634363009
634 0.167226639762651
635 0.167226642804377
636 0.167225032985829
637 0.167225054013042
638 0.167225053579868
639 0.167225114015615
640 0.167225155087138
641 0.167225156138824
642 0.167225444201749
643 0.167225444130144
644 0.167225578371492
645 0.167225572545192
646 0.167235547455462
647 0.167235552601381
648 0.167235552923441
649 0.167247996641589
650 0.167248006246263
651 0.16724841345324
652 0.167248412500627
653 0.168737582144247
654 0.168737581115154
655 0.168737575091882
656 0.168737726440632
657 0.16873772967874
658 0.168737719558152
659 0.168676364754353
660 0.168676368982187
661 0.168676369972987
662 0.16867638197577
663 0.168676637392282
664 0.168676640135362
665 0.168676902840104
666 0.168676904414808
667 0.1687056506959
668 0.168695092812656
669 0.16869509981775
670 0.16895051801106
671 0.168951853265636
672 0.168951868938743
673 0.126599095835855
674 0.12659910750918
675 0.126599245638109
676 0.126599386984009
677 0.126599763981136
678 0.126599764026158
679 0.126599972237454
680 0.126603680830511
681 0.126603682819277
682 0.126603682855503
683 0.126603692390865
684 0.126603723712702
685 0.126603724919157
686 0.12660378070473
687 0.126603817151329
688 0.126603817590111
689 0.126603767595838
690 0.126603770632482
691 0.126603775779205
692 0.126603966867865
693 0.126603969280136
694 0.12661114413792
695 0.126611148147531
696 0.12661115139252
697 0.126607983683588
698 0.126608716866956
699 0.126608745239352
700 0.126608756782271
701 0.12660876362091
702 0.126608792016642
703 0.126608800762406
704 0.126608801955302
705 0.126608810311042
706 0.126608862062786
707 0.126612653836144
708 0.126612662149785
709 0.126612694262657
710 0.126612705607062
711 0.126612721383226
712 0.126612726794688
713 0.126612745434732
714 0.126613234016649
715 0.126613234464548
716 0.126613362171492
717 0.126615849481173
718 0.126615855496436
719 0.126618716001765
720 0.126618717671586
721 0.126618728474081
722 0.126618734671217
723 0.126629267876802
724 0.126629304354908
725 0.12663081529772
726 0.126630815660204
727 0.126630841833892
728 0.12663084558616
729 0.126630864107389
730 0.126630876183342
731 0.126630876291005
732 0.126630946694409
733 0.126641569529142
734 0.129565478423673
735 0.129565521275222
736 0.129565523800244
737 0.129565565036111
738 0.129565565448473
739 0.129850025725753
740 0.129850028831805
741 0.129850033863344
742 0.129850236407584
743 0.129850259483736
744 0.129850277254979
745 0.129850320462951
746 0.129850322149035
747 0.129850451399972
748 0.129850451890605
749 0.12985045568619
750 0.129850459162982
751 0.129850459795309
752 0.129850464292771
753 0.12985049057338
754 0.129850495773692
755 0.129850516625539
756 0.129850518405643
757 0.129850521258079
758 0.129850521976349
759 0.12985052778768
760 0.1298512393182
761 0.129851240524496
762 0.130324293604102
763 0.13032429525163
764 0.13032429502215
765 0.130324297337041
766 0.130324298792556
767 0.130324421797513
768 0.130324425357929
769 0.130324425083241
770 0.130324435122854
771 0.130324440535376
772 0.130324459157223
773 0.13032446347473
774 0.130324484472018
775 0.130324484345981
776 0.130324489216985
777 0.130324514722508
778 0.130324524230102
779 0.130324651930627
780 0.130324923146293
781 0.130324923149946
782 0.13032493354015
783 0.1303249334361
784 0.130324934044701
785 0.130324941229741
786 0.1303252728602
787 0.130325273553276
788 0.130325274700868
789 0.130325320450074
790 0.130325536392867
791 0.130325537204688
792 0.13032956463234
793 0.130329572506739
794 0.130329595812118
795 0.130331592037284
796 0.130331592038021
797 0.130331615403983
798 0.130331616064586
799 0.130331615768894
800 0.130331621529635
801 0.130422994976876
802 0.130422950194295
803 0.130422955886653
804 0.130426578219262
805 0.130428539530725
806 0.130428551840357
807 0.130428552611834
808 0.130428554889304
809 0.130428556504149
810 0.131224978464234
811 0.131224979608372
812 0.131224990097255
813 0.131224991913842
814 0.131224998612883
815 0.13122500564819
816 0.131225020019366
817 0.131225025335882
818 0.0970480308765582
819 0.0970480982285454
820 0.0970481072643787
821 0.0970481107293605
822 0.0970481113972554
823 0.0970482298738009
824 0.0970482345279468
825 0.0972436187138229
826 0.0972498648108322
827 0.0972498803660588
828 0.0972498813796581
829 0.0972498820762104
830 0.0972498856162897
831 0.0972498928441794
832 0.0972499022901842
833 0.0972499022917357
834 0.0972501125635262
835 0.0972501127521622
836 0.0972501560421363
837 0.0974387766911809
838 0.0974387791708861
839 0.0974388227239781
840 0.0974389000091477
841 0.0974389256088371
842 0.0974388575430909
843 0.097438857463456
844 0.0974388600222128
845 0.0974388600756199
846 0.0974388605285129
847 0.0974389595200189
848 0.0974641795278313
849 0.0974641793947365
850 0.0974641996222466
851 0.0974642045255573
852 0.0974642119213746
853 0.0974642119857354
854 0.0974642155414558
855 0.0974642387108296
856 0.0974642644167873
857 0.0974642656381398
858 0.0974642777649211
859 0.0974642797442166
860 0.0974642805886643
861 0.0974643156847577
862 0.0974643154733264
863 0.0974531169142338
864 0.0974531184327204
865 0.0974531209690842
866 0.0974531872402254
867 0.0974532008803969
868 0.0974532072598462
869 0.0974532076383919
870 0.0974532181565748
871 0.0974532184479357
872 0.0974532307542193
873 0.097453249604718
874 0.0974532495711061
875 0.0974532509286063
876 0.0974533185067672
877 0.0974533195046687
878 0.0974533436570632
879 0.0974533436531857
880 0.0974284448214641
881 0.0974457518456302
882 0.0974457536990831
883 0.0974457592200479
884 0.097445761603797
885 0.0974459639653843
886 0.0974459983820625
887 0.0974459981925701
888 0.097445999378331
889 0.097446000741742
890 0.0974460181607606
891 0.09744612537162
892 0.0974461306376138
893 0.0974480850409906
894 0.097448085727382
895 0.0974480935257321
896 0.0974480961270479
897 0.0974480994140318
898 0.0974481263585539
899 0.0989712576466292
900 0.0989716477205867
901 0.0989716546331327
902 0.0989719584277788
903 0.0989713787925017
904 0.0989714484869055
905 0.0989714501976558
906 0.0989714507964216
907 0.0989714523261054
908 0.098971461229925
909 0.0986556849090183
910 0.0986557211004056
911 0.0986737187737743
912 0.0986737196527588
913 0.0986737686309078
914 0.0986737670227345
915 0.0986737807189595
916 0.0986738052035643
917 0.0986738057195219
918 0.0986739540095238
919 0.0986739470588514
920 0.0986739494680076
921 0.098673967738559
922 0.0986744318386652
923 0.098674448335615
924 0.0986744495731064
925 0.0986743674745905
926 0.0986743852665458
927 0.0986743852702324
928 0.0986744318154018
929 0.0986744340061471
930 0.0986744339271246
931 0.0986744362062104
932 0.0986788355695671
933 0.0986788702731693
934 0.0986788703841772
935 0.0986788737384668
936 0.0986788739020229
937 0.0986652114510996
938 0.0986652217435478
939 0.0986652280754304
940 0.0986652281115469
941 0.0986652286021588
942 0.098561332066558
943 0.0985613344333416
944 0.0985613344391957
945 0.0985613380864104
946 0.0985627412430168
947 0.0985642764496876
948 0.0985642771094607
949 0.0985642776322603
950 0.0985642822908053
951 0.0985642877986772
952 0.0985642880916972
953 0.0985643179764765
954 0.0985643457969078
955 0.09856434651416
956 0.0985643498842116
957 0.0985643501556919
958 0.0985643841929256
959 0.0985643858033948
960 0.0985644906515017
961 0.0985649001705678
962 0.0985649003771177
963 0.0985649004968843
964 0.0985649006496895
965 0.0985649067682558
966 0.0985649152238315
967 0.0985649287607866
968 0.0985650462819798
969 0.0985650499303135
970 0.0985650533640773
971 0.0985650533932581
972 0.0985650535444892
973 0.0985651048052823
974 0.0985651790035469
975 0.0985651790027608
976 0.0985651827847328
977 0.0985651830522463
978 0.0985651957676276
979 0.098565202534473
980 0.0985652026617396
981 0.0985652029079113
982 0.098565202783582
983 0.09856520300837
984 0.0985653357648138
985 0.0985653396728568
986 0.0985653710894015
987 0.0985653774280351
988 0.0985653782955733
989 0.0985948942090155
990 0.0985948946385494
991 0.0985948992814847
992 0.0985948999590456
993 0.0985949001121955
994 0.098594899720085
995 0.0985962532283154
996 0.0985962535593655
997 0.0985980033015341
998 0.0985980040021345
999 0.0985980277363209
1000 0.099381761355516
1001 0.0993817612384965
1002 0.0993817616270345
1003 0.0993817619069036
1004 0.0993817627053842
1005 0.0993817630638759
1006 0.0993817680362088
1007 0.0993817681456994
1008 0.0993817685534593
1009 0.0993817709406155
1010 0.0993817756419719
1011 0.0993810976119642
1012 0.0993811080568533
1013 0.099394538157009
1014 0.0993945385821339
1015 0.0993945979538478
1016 0.0993945982060422
1017 0.0993945986537776
1018 0.0993946062967044
1019 0.0993946099380166
1020 0.0993947099831606
1021 0.0993946403912712
1022 0.0993946405237588
1023 0.0993946421626489
1024 0.0993946533563106
1025 0.0993946533729257
1026 0.0993942487780254
1027 0.0993942490351884
1028 0.0993969396942179
1029 0.0994120599687455
1030 0.0994120602348745
1031 0.0994120603524472
1032 0.0994120629179551
1033 0.0994120627208626
1034 0.098546594745351
1035 0.0985465983923859
1036 0.0985465986919112
1037 0.0985466107938614
1038 0.0985466109916971
1039 0.0985466124697378
1040 0.0985466107894014
1041 0.0985466110793156
1042 0.0985466111713816
1043 0.0985466142323494
1044 0.0985466352521849
1045 0.0985466397893391
1046 0.098832853311708
1047 0.0988328537314274
1048 0.0897922712223075
1049 0.0897922719700311
1050 0.0897922720218025
1051 0.0897922747972269
1052 0.0897928013761326
1053 0.089792822030875
1054 0.0897928245083473
1055 0.0897928255886024
1056 0.0897929054846496
1057 0.0897929058121646
1058 0.0897929059878343
1059 0.0897932053399789
1060 0.0897932056171954
1061 0.0897932061608794
1062 0.0897932078912706
1063 0.0897932080213759
1064 0.0897932048254347
1065 0.0897932915750536
1066 0.0897932916223307
1067 0.0897933100066041
1068 0.0897933112469769
1069 0.0897933156937179
1070 0.0897933159261535
1071 0.0897933235996571
1072 0.0897933660131908
1073 0.0897933675251631
1074 0.0897933677888368
1075 0.0897926992394399
1076 0.0897927004444122
1077 0.089793694975579
1078 0.0897937270502978
1079 0.0897937652101659
1080 0.089793765205024
1081 0.0897937656007121
1082 0.0897937672523101
1083 0.0897937672532258
1084 0.0897937674635702
1085 0.0897937879130306
1086 0.0897828417381957
1087 0.0897813787541125
1088 0.089781383970624
1089 0.0897866591318076
1090 0.0897866590985097
1091 0.0897866803660642
1092 0.089786680938514
1093 0.0897866812414465
1094 0.089788624165841
1095 0.0897855549023065
1096 0.0897855715021826
1097 0.0897855715023001
1098 0.0897855901570505
1099 0.0897855919666423
1100 0.0897855935514523
1101 0.089785621983392
1102 0.0897856265111752
1103 0.089785626497387
1104 0.0897856296175095
1105 0.0897856300111032
1106 0.089797888778786
1107 0.0897978895768952
1108 0.089797890938301
1109 0.0897978909255225
1110 0.089797891040185
1111 0.089797891875004
1112 0.0897978921481841
1113 0.0897978929906336
1114 0.089798002980193
1115 0.0897986379862743
1116 0.0897956849349272
1117 0.0897956872484534
1118 0.0897956881611731
1119 0.0897956882437946
1120 0.089795565574381
1121 0.0897955710758214
1122 0.0897955719580967
1123 0.0897955720883243
1124 0.111064088107652
1125 0.111064091711946
1126 0.111064091676262
1127 0.111064091694272
1128 0.111186753619695
1129 0.111186765822381
1130 0.11118676567961
1131 0.111186765935921
1132 0.111186769788824
1133 0.112977035331378
1134 0.112910014536887
1135 0.112910014701316
1136 0.112910062590441
1137 0.112910062595691
1138 0.112910062510145
1139 0.113115213658266
1140 0.113115199071962
1141 0.113115206121742
1142 0.113110644424378
1143 0.113114281311114
1144 0.113114309674849
1145 0.120400737571008
1146 0.120400740182243
1147 0.120400740605159
1148 0.120401247052656
1149 0.120401248851718
1150 0.120401248866852
1151 0.120401249861416
1152 0.12040125045008
1153 0.120401250723248
1154 0.120401256674696
1155 0.120401256590824
1156 0.120401300430383
1157 0.120401307994408
1158 0.120401307055545
1159 0.120473587320588
1160 0.120473587313262
1161 0.12047365665616
1162 0.120473658411237
1163 0.121330169226725
1164 0.121330169561902
1165 0.121330193075498
1166 0.12133367853221
1167 0.121333683868281
1168 0.121333684843688
1169 0.12133368556758
1170 0.121333698197103
1171 0.121333723522604
1172 0.12133354399336
1173 0.121333623317181
1174 0.121333676046301
1175 0.121333677765396
1176 0.121333685478641
1177 0.121333685492834
1178 0.121333686395728
1179 0.12133379596407
1180 0.121333795958683
1181 0.121333796077552
1182 0.121333799989573
1183 0.121333806150454
1184 0.121333836453826
1185 0.121333838982279
1186 0.121333839374395
1187 0.121333839900337
1188 0.121333840294557
1189 0.12133392022095
1190 0.121334075174773
1191 0.121334075298744
1192 0.121334083403771
1193 0.121334083646775
1194 0.121334111738908
1195 0.121334112091903
1196 0.121334112295192
1197 0.121334112633124
1198 0.121335571717239
1199 0.121335571763932
1200 0.121353135444511
1201 0.121353198979701
1202 0.12135319874695
1203 0.1213531987512
1204 0.121353198733064
1205 0.121353198976903
1206 0.121353198941584
1207 0.121353199158599
1208 0.121353206215826
1209 0.121353245001272
1210 0.12135324552726
1211 0.121353245691993
1212 0.12135324969695
1213 0.121353249619815
1214 0.121353250536927
1215 0.121353250512414
1216 0.121362278332101
1217 0.121362278484765
1218 0.121361712300007
1219 0.121361736674947
1220 0.121383016767306
1221 0.121383020090253
1222 0.121383055292007
1223 0.121383055602665
1224 0.121383056546461
1225 0.121383479686077
1226 0.121383684927848
1227 0.121383687423423
1228 0.121383689026466
1229 0.121383687461519
1230 0.121383687411263
1231 0.121383687565972
1232 0.12138368974102
1233 0.121395487955034
1234 0.121395494616611
1235 0.121395494836537
1236 0.12139549449174
1237 0.121397567066732
1238 0.121397534614234
1239 0.121397541113313
1240 0.121397543277114
1241 0.121397543552115
1242 0.121397543578876
1243 0.121397544041746
1244 0.121402849827402
1245 0.121402854521945
1246 0.121402866641451
1247 0.121402866811389
1248 0.121402866887287
1249 0.121402877582017
1250 0.121402879603686
1251 0.121402880108559
1252 0.121402880126191
1253 0.121402886401354
1254 0.121402888059571
1255 0.121402889426898
1256 0.121403327950716
1257 0.121403328010924
1258 0.121403328333504
1259 0.121403329246677
1260 0.121403329447741
1261 0.121403330026824
1262 0.121403285732344
1263 0.121403285929281
1264 0.121403297488812
1265 0.121403298256791
1266 0.12140331193061
1267 0.121403280238213
1268 0.121403280154487
1269 0.121412329261757
1270 0.121411402042213
1271 0.121411402051501
1272 0.12141140485443
1273 0.121411405905584
1274 0.121411405949337
1275 0.121411407436872
1276 0.121408544697747
1277 0.121411725199667
1278 0.121411727226331
1279 0.121411728418466
1280 0.12141172849419
1281 0.121411727800382
1282 0.121411727851392
1283 0.121411728648208
1284 0.121455576698218
1285 0.121455576504011
1286 0.121455576591752
1287 0.121455576738691
1288 0.121455626398916
1289 0.121455656164743
1290 0.121455661904521
1291 0.121455663359613
1292 0.121455669836321
1293 0.121455677436802
1294 0.121455677322667
1295 0.121455683208893
1296 0.121455725205928
1297 0.121455728334877
1298 0.121455728320955
1299 0.121455728352662
1300 0.12145573001056
1301 0.121455729993427
1302 0.121455731542861
1303 0.121455731828123
1304 0.121455731809247
};
\end{axis}

\end{tikzpicture}

%% file: Figures/vis_decay_pacman_singular_p_sol.tex
\begin{tikzpicture}

\definecolor{darkgray176}{RGB}{176,176,176}
\definecolor{steelblue31119180}{RGB}{31,119,180}

\begin{axis}[
title={$\sup_{x \in \Omega} |(u-s_n)(x)|$},
width=0.951\fwidth,
height=0.75\fwidth,
log basis x={10},
log basis y={10},
tick align=outside,
tick pos=left,
x grid style={darkgray176},
xmin=0.698612127441536, xmax=1866.55792073854,
xmode=log,
xtick style={color=black},
y grid style={darkgray176},
ymin=1e-4, ymax=3,
ymode=log,
ytick style={color=black}
]
\addplot [semithick, steelblue31119180]
table {%
0 1.99926184372634
1 1.09174182568746
2 1.0497911925378
3 0.539198131473817
4 0.131798298583389
5 0.166315569735812
6 0.142454327268245
7 0.0432638055005943
8 0.031576744134449
9 0.026962258518149
10 0.0490849765104915
11 0.065294823802466
12 0.0614741397065988
13 0.0619042798373735
14 0.0622391014757238
15 0.0621426911860745
16 0.0618375985834885
17 0.0608276469219946
18 0.0688255362708305
19 0.0689753683542236
20 0.0699083834692091
21 0.0704921807161896
22 0.0700086739407733
23 0.0703168586023439
24 0.0416979911716922
25 0.0417044833487104
26 0.0416323556309441
27 0.0417479846744144
28 0.0385622783145356
29 0.0362374954180267
30 0.0362285969378784
31 0.0222528226942555
32 0.0221994293077388
33 0.0221711298513316
34 0.0221336533702594
35 0.0221436467743852
36 0.0206833475262613
37 0.0184193523181049
38 0.0183906273621597
39 0.0183899871006621
40 0.0184812981249047
41 0.0171896295362961
42 0.0171842051834465
43 0.0172371571368133
44 0.0258543960193425
45 0.0255071889514724
46 0.025503533559873
47 0.0254970362759595
48 0.0247007125538556
49 0.0246825885452153
50 0.0246321258126545
51 0.0247017483486276
52 0.0246994611593461
53 0.0246708314012756
54 0.0246931009801115
55 0.0247207310868265
56 0.0247210735039161
57 0.0207388981742782
58 0.0129001344301696
59 0.0128997464228813
60 0.0128997931498114
61 0.0128969172366185
62 0.0128345007942656
63 0.0128365192601054
64 0.0128079967743233
65 0.013613241707509
66 0.0136132063467918
67 0.0136131617142112
68 0.0136498232568845
69 0.0136425803369407
70 0.0137082812595499
71 0.0137062458688717
72 0.0136912927189996
73 0.0136956416295009
74 0.0136955927370095
75 0.0136993027275132
76 0.013697689950791
77 0.0136976916755018
78 0.013697595874832
79 0.00910425595341446
80 0.00910320253801422
81 0.009103426308817
82 0.00910144562398973
83 0.00910068597003399
84 0.00907867020150177
85 0.00907734720896802
86 0.00895014622876977
87 0.00894911846750346
88 0.00894852004635172
89 0.00895161261301802
90 0.00779516378971445
91 0.0077951524818467
92 0.00779488606878176
93 0.00829429554141248
94 0.00829409614210674
95 0.0082729591249775
96 0.00882807002032804
97 0.00664461037415065
98 0.0028510957175345
99 0.00285023600331069
100 0.00285023021561059
101 0.00285023715654287
102 0.00285023765648673
103 0.00284614792306193
104 0.00279018531238284
105 0.00278218886766557
106 0.00278166755936438
107 0.00278167935470797
108 0.00278162321071584
109 0.00278021462935052
110 0.00278007869436681
111 0.00278000832935343
112 0.00272214794926606
113 0.00272170984480069
114 0.00272168063484424
115 0.00272168397795491
116 0.00272298226395273
117 0.00272128605326438
118 0.00265676302292106
119 0.00265668811326858
120 0.00265403628855232
121 0.00265383654207918
122 0.00265305093939538
123 0.00264771013531973
124 0.00264756980741576
125 0.00264755333125999
126 0.00264761050051887
127 0.00106834081918517
128 0.00106834206220352
129 0.00106839987696206
130 0.00106852547484726
131 0.0010685404326356
132 0.00107474231242954
133 0.00107478352092483
134 0.00107478336915334
135 0.00110609098267656
136 0.00110608905765153
137 0.00110608908956666
138 0.00110611761658608
139 0.00110613204312293
140 0.00110683486602725
141 0.00110687248084407
142 0.00110659754808173
143 0.00110661376496446
144 0.00110661280907309
145 0.00110672779213417
146 0.00110673371702319
147 0.00110674400008359
148 0.00110663388536958
149 0.00110460188721118
150 0.00110986060816698
151 0.00110986939065949
152 0.00110878047667318
153 0.00110878046660146
154 0.00110953468938013
155 0.00110955369284449
156 0.00110957166983394
157 0.00110988550667801
158 0.00111016350849358
159 0.00111017321367446
160 0.00111017318423845
161 0.00111017871680619
162 0.00111026547611082
163 0.00108082912858465
164 0.0010808292202753
165 0.00108082875711513
166 0.00108083007443871
167 0.00108082999561776
168 0.0010799470474645
169 0.00107994707933901
170 0.00107994730856764
171 0.00107997747749233
172 0.00115668736468888
173 0.00115670889787012
174 0.00109480981964794
175 0.00109449354037072
176 0.00109449377827819
177 0.00110244318712249
178 0.00110244479703225
179 0.00110248230173626
180 0.00110248430153725
181 0.00110248452490902
182 0.00110248468363006
183 0.00110510178401424
184 0.00110510262621033
185 0.00110510314470447
186 0.0011051340621957
187 0.00110513497269316
188 0.00110513580673377
189 0.00110513653469102
190 0.00110513643321708
191 0.00110521546957587
192 0.00110519316331015
193 0.00110519316807434
194 0.00110581845288893
195 0.00110582127229075
196 0.00110582195653364
197 0.00110584475762687
198 0.00110584528588364
199 0.00110590916714348
200 0.00110591191579479
201 0.00110591215513689
202 0.00110591306262742
203 0.00110593232706435
204 0.00110707840858781
205 0.0011070953270369
206 0.00110709540000276
207 0.00110709531609876
208 0.00110709531672581
209 0.00110709531090469
210 0.00110711120581652
211 0.00110711406993524
212 0.00124755109939945
213 0.0012475906021312
214 0.00124759113478312
215 0.00124759287465448
216 0.00124759300758481
217 0.00124759538830421
218 0.00124760321028594
219 0.00124760461710305
220 0.0012476055767312
221 0.00124760624229237
222 0.00124764016959267
223 0.00124764077127693
224 0.00124764487664164
225 0.00124782155106473
226 0.00124782277220148
227 0.00124782537928758
228 0.00121171643965012
229 0.00121171857552982
230 0.00121172481142229
231 0.0012117342870015
232 0.00121173434451904
233 0.00121173810835584
234 0.00128086418152962
235 0.00128111055039937
236 0.00128111739930126
237 0.0012811239453776
238 0.00128096765055941
239 0.00128096935893596
240 0.00128097319177312
241 0.00128097320130305
242 0.00128102032485256
243 0.00128107174099967
244 0.00128095180138321
245 0.00128095255543736
246 0.00128095268779793
247 0.00128096162536329
248 0.00128094629405662
249 0.00128095370084114
250 0.00128096726307803
251 0.00128097170270913
252 0.00128097222921042
253 0.00128097222320078
254 0.00128111492481175
255 0.00128114339209695
256 0.001281143411338
257 0.00128105570729753
258 0.00128105551263125
259 0.00128105558395486
260 0.00128105580824478
261 0.00128105601966122
262 0.00130054355351872
263 0.00130054378158828
264 0.00130054540536362
265 0.0013005454149384
266 0.00130064422588871
267 0.00130065591793671
268 0.00130065590897188
269 0.00130065611749819
270 0.00130071132393406
271 0.0013069073104004
272 0.00130690707503733
273 0.00130690757313157
274 0.00130689230605285
275 0.001306892482418
276 0.00130689331134826
277 0.00130689576568943
278 0.0013069029370818
279 0.0013069187293997
280 0.00130691372327685
281 0.0013069247966575
282 0.00130692510772312
283 0.00130692518483455
284 0.00130696796390972
285 0.00130696840923128
286 0.00130697017110015
287 0.00130697248458667
288 0.00130697684809666
289 0.00130699887812691
290 0.00130639514319242
291 0.00130639751066464
292 0.00130639774612029
293 0.00130654287353971
294 0.00130305848794809
295 0.00130305936212349
296 0.00130308419948788
297 0.00130334266997134
298 0.00130334236657492
299 0.00130334352619799
300 0.00130334366637319
301 0.00130347783951046
302 0.00130347802477582
303 0.00130347828566624
304 0.00128997131643782
305 0.00128997984951984
306 0.00128996716871188
307 0.00129234835634251
308 0.00129234845921533
309 0.00129234859393645
310 0.00129235123948912
311 0.00129235602439093
312 0.00129235850691733
313 0.00129235919362203
314 0.00129235927402505
315 0.0012923597391894
316 0.00129235986154774
317 0.00129237954957651
318 0.00129237968100426
319 0.00126932657770307
320 0.00121270866640844
321 0.00121270815334928
322 0.00121270850676414
323 0.00121272982286524
324 0.00121279600150248
325 0.00121279195327184
326 0.0012127929756669
327 0.00121279312074507
328 0.00121280144058411
329 0.0012128018062445
330 0.00121280209497998
331 0.0012128020915112
332 0.00121280206804841
333 0.00121281552013341
334 0.0012126034878468
335 0.00121261458552269
336 0.00121261465193445
337 0.00121261705903719
338 0.00121265455871877
339 0.00121265984851915
340 0.00121268676706476
341 0.00121630544307338
342 0.00121631049072768
343 0.00121630180320853
344 0.00121630179609933
345 0.00121630180096921
346 0.00121630179894772
347 0.00121630234790993
348 0.00121630237108805
349 0.00121630055935462
350 0.00121630097576109
351 0.00121630188689004
352 0.00121614487278299
353 0.00121596815842451
354 0.00121596820448278
355 0.00121601499621216
356 0.00121601790447268
357 0.001216018291192
358 0.00121601856152198
359 0.00121601872358523
360 0.00121602169972301
361 0.00128059002267311
362 0.00127940077937905
363 0.00127940082391209
364 0.00127940142781524
365 0.00127940362387924
366 0.00127940367675428
367 0.00127940312044106
368 0.00127940418779793
369 0.00127787993690798
370 0.00127789239577281
371 0.0012778924116057
372 0.00127789386651478
373 0.00127789665452971
374 0.00127789666951705
375 0.00127789716322635
376 0.0012778807275422
377 0.00127788606376034
378 0.0012778860764151
379 0.00127788631696535
380 0.00127788640627591
381 0.00127798884340824
382 0.00108932325546918
383 0.00108932480738955
384 0.00108932596341416
385 0.00108932617679902
386 0.00108932616563573
387 0.00108944350288653
388 0.00108944334374206
389 0.0010894434202775
390 0.00108944343888662
391 0.00108945958926498
392 0.00108945956891793
393 0.00108945880648625
394 0.00108945894263646
395 0.00108946053926684
396 0.00108946146915567
397 0.00108946144285715
398 0.00108946177298241
399 0.00108946779966757
400 0.00108946792271114
401 0.00108946793146969
402 0.00108947733784182
403 0.00108947736569687
404 0.00108947819790894
405 0.00108947863674747
406 0.00108947864764453
407 0.00108947865525333
408 0.00108947975627927
409 0.00108954877638578
410 0.00108954890151414
411 0.0010895506901254
412 0.00108955118336418
413 0.00109002013077975
414 0.0010900210489766
415 0.00109002123641821
416 0.00109002149725446
417 0.00109002153367665
418 0.00109002245345602
419 0.00109010266817
420 0.00109010897253525
421 0.00109013424125037
422 0.00109013733904595
423 0.00109012964014088
424 0.00109013152582538
425 0.00109014549736153
426 0.0010901457068222
427 0.00109016234259318
428 0.0010901750361505
429 0.00109017535279365
430 0.00109017536023837
431 0.0010901151580256
432 0.00109011032101569
433 0.00109011104054724
434 0.00109011112211532
435 0.00109011214678323
436 0.0010901157744152
437 0.00109011610421539
438 0.00109011675047244
439 0.00109012223519467
440 0.00109012270929987
441 0.00109012295429678
442 0.00109012297718314
443 0.0010901231657281
444 0.00109012400637698
445 0.00109012584482526
446 0.00109012730780744
447 0.0010899043068302
448 0.00109809319127585
449 0.00109809327018451
450 0.0010980943370229
451 0.00109809534599736
452 0.00109597170667541
453 0.00109599222950907
454 0.0010959947362581
455 0.00109599521380255
456 0.00109599524305404
457 0.00109599528632853
458 0.00109599528886406
459 0.00109599691873186
460 0.0010959978532088
461 0.00109599949678718
462 0.0010959994957993
463 0.00109601403645287
464 0.00109601423348504
465 0.00109601429329476
466 0.00109601643734036
467 0.00109602085694549
468 0.0010960210524662
469 0.00109602105323714
470 0.00109604780273309
471 0.00109604942414299
472 0.00109605005570068
473 0.00109605039665306
474 0.00109605056505346
475 0.00109605134481394
476 0.00109605356575315
477 0.00109605358075493
478 0.00109605363947862
479 0.0010961506996745
480 0.00109615070172797
481 0.00109664969543566
482 0.00109590675170002
483 0.00109590676712812
484 0.00109590689849681
485 0.00109591056933578
486 0.00109591069874981
487 0.00109591073403448
488 0.00109571315130119
489 0.00109571317989587
490 0.00109571356217519
491 0.001095713672705
492 0.00109571414710841
493 0.00109571415864651
494 0.00109572291042848
495 0.00109576051232407
496 0.00109576058106065
497 0.00109591346547555
498 0.00109591591584457
499 0.00109591594777281
500 0.00109591592747837
501 0.00109591599601355
502 0.0010959193898894
503 0.00109591940992115
504 0.00109591968230505
505 0.00109592166428452
506 0.00109592166294292
507 0.00110896654800485
508 0.00110895256550436
509 0.00110895273578526
510 0.00110896109501035
511 0.001108961102116
512 0.00110896110201741
513 0.00110896110117409
514 0.00110896113527859
515 0.00110881519530781
516 0.00110881594890699
517 0.00110881605938484
518 0.00110881888084613
519 0.00110882570002091
520 0.0011088302590736
521 0.00110883110786153
522 0.0011088309915539
523 0.00110293111502369
524 0.00110293133103201
525 0.00110293172058351
526 0.00110293206123702
527 0.00110293211650569
528 0.00110293324812738
529 0.00110293324736022
530 0.00110293359130043
531 0.00110293476099388
532 0.00110280322629608
533 0.00110289805369757
534 0.00110289807649622
535 0.00110279742496444
536 0.00110279633956267
537 0.001102815722414
538 0.00110281579718641
539 0.00110186461931527
540 0.00110186461849904
541 0.00110186504910748
542 0.0011018650483241
543 0.0011021207714863
544 0.0011021418833228
545 0.00110214191867053
546 0.00110214345976001
547 0.00110214347024984
548 0.00110214746597315
549 0.00110214830321009
550 0.0011023383326767
551 0.00110233879994914
552 0.00110233868018494
553 0.0011023386893978
554 0.00110233871785748
555 0.00110234099313122
556 0.00110234192402592
557 0.00110226439545813
558 0.00110226439697869
559 0.00110227257579365
560 0.00110227519975536
561 0.0011022744074527
562 0.00110227440805022
563 0.00110227441215804
564 0.00110227525573103
565 0.00110228253116018
566 0.00110228003512991
567 0.00110413126250442
568 0.00110413167421863
569 0.0011041327830188
570 0.0011041330955659
571 0.00110413311231006
572 0.00110413410752042
573 0.00110413461328895
574 0.00110413393120101
575 0.00110391215095529
576 0.00110391216709149
577 0.00110391205529137
578 0.0011039139498823
579 0.00110391822344758
580 0.00110391824410794
581 0.0011041108224279
582 0.0011041507738434
583 0.00110415078751713
584 0.00110415148380683
585 0.00110415111818174
586 0.00110415114779627
587 0.00110414787553381
588 0.00110414978937801
589 0.00110414980214113
590 0.00110417422439024
591 0.00110420102663333
592 0.00110420102962072
593 0.00110420102825648
594 0.00110371619927685
595 0.00110371626641204
596 0.00110371640631102
597 0.00110371643072482
598 0.00110371648117735
599 0.00110371820205701
600 0.00110371833535816
601 0.00110371828298539
602 0.0011037185928755
603 0.00110372037866746
604 0.00110372045554752
605 0.00110372053841834
606 0.00110372059044317
607 0.00110372823002103
608 0.00110372863882602
609 0.00110372867865305
610 0.00109796173792254
611 0.00109796176110644
612 0.0010979616727822
613 0.00109796167342591
614 0.00109796286532937
615 0.00109795859204143
616 0.00109795869900853
617 0.00109798836569319
618 0.00109798838083863
619 0.00109798854635446
620 0.00109798858188404
621 0.00109798718951115
622 0.00109798769328262
623 0.00109798880701639
624 0.00109798960795926
625 0.00109838382109739
626 0.00109838381982064
627 0.00109827083887204
628 0.00109827082761194
629 0.00109827084181102
630 0.00109828239926379
631 0.00109829115801263
632 0.00109828782321175
633 0.00109829815594886
634 0.00109829819270724
635 0.00109829821221319
636 0.00109829425077934
637 0.00109829438722486
638 0.00109829438317899
639 0.00109829466337019
640 0.00109829492256885
641 0.00109829492632452
642 0.00109829663704475
643 0.00109829663655137
644 0.00109829764848812
645 0.00109829761954083
646 0.00109833838609052
647 0.00109833841930884
648 0.00109833840835494
649 0.00109837509094257
650 0.00109837513797806
651 0.00109837709195126
652 0.00109837709005789
653 0.00110431167556979
654 0.00110431166850544
655 0.00110431163848634
656 0.00110431265647382
657 0.00110431267739775
658 0.00110431263172628
659 0.00110429618752428
660 0.00110429621485664
661 0.00110429623044705
662 0.00110429630766085
663 0.00110429757658603
664 0.0011042975960458
665 0.00110429927464084
666 0.00110429928463707
667 0.00110424805079168
668 0.00110421330916943
669 0.00110421334790445
670 0.00110546748430829
671 0.00110547358259194
672 0.00110547369050429
673 0.0010923669467009
674 0.00109236704791882
675 0.00109236826164927
676 0.00109236906670152
677 0.00109236946443225
678 0.00109236946999869
679 0.00109237128054551
680 0.00109236253255851
681 0.00109236255008582
682 0.00109236255157708
683 0.00109236263607659
684 0.00109236291057169
685 0.00109236292106596
686 0.00109236338578134
687 0.00109236370882759
688 0.00109236371208565
689 0.00109236601315255
690 0.00109236603981877
691 0.00109236608561125
692 0.00109236771163301
693 0.00109236773310961
694 0.00109234803296054
695 0.00109234804896241
696 0.00109234807719116
697 0.0010923579837554
698 0.00109235568537236
699 0.00109235593998913
700 0.00109235604187519
701 0.00109235609955105
702 0.00109235634940585
703 0.00109235641311334
704 0.00109235642887051
705 0.00109235650306538
706 0.00109235694823728
707 0.00109238123607747
708 0.0010923813099355
709 0.00109238158549885
710 0.00109238168742176
711 0.00109238158366765
712 0.00109238162892322
713 0.00109238174376891
714 0.00109238488446128
715 0.00109238489054242
716 0.00109238596927508
717 0.00109237856124422
718 0.00109237815537555
719 0.00109239935798433
720 0.0010923993727372
721 0.00109239946873263
722 0.0010923995240022
723 0.00109246759193393
724 0.00109246747879177
725 0.00109246277127828
726 0.00109246277439357
727 0.00109246300745847
728 0.00109246304007371
729 0.00109246315743694
730 0.00109246326213652
731 0.00109246326260437
732 0.0010924638040084
733 0.0010924963955945
734 0.00109253882399085
735 0.00109253920086405
736 0.00109253922272345
737 0.00109253957408018
738 0.00109253957831879
739 0.00109432509117036
740 0.00109432511786545
741 0.00109432514975261
742 0.00109432700288004
743 0.00109432719903202
744 0.00109432736938841
745 0.00109432774202678
746 0.00109432775622254
747 0.00109432883367377
748 0.00109432883594018
749 0.00109432886874172
750 0.00109432889919003
751 0.00109432890302319
752 0.001094328941873
753 0.00109432917071151
754 0.00109432921587493
755 0.00109432914407059
756 0.00109432915929575
757 0.00109432918396513
758 0.00109432918999119
759 0.00109432924027564
760 0.00109433464986974
761 0.00109433466037645
762 0.00109443999028636
763 0.00109444000572334
764 0.00109444000360548
765 0.00109444002374115
766 0.00109444003603709
767 0.00109444077430676
768 0.0010944408049065
769 0.0010944408021627
770 0.00109444088812061
771 0.00109444093362332
772 0.00109444107752754
773 0.00109444111459478
774 0.00109444129452929
775 0.00109444129588665
776 0.00109444133927683
777 0.00109444156076322
778 0.00109444164308781
779 0.00109444278563653
780 0.00109444583397011
781 0.00109444583401808
782 0.00109444594389707
783 0.00109444595307795
784 0.00109444595901831
785 0.00109444602415421
786 0.0010944483972648
787 0.00109444840281769
788 0.00109444839577777
789 0.00109444878925125
790 0.00109445047202228
791 0.00109445047775569
792 0.00109447502306992
793 0.00109447509075333
794 0.0010944752330988
795 0.00109446857655615
796 0.00109446857656037
797 0.00109446877223451
798 0.00109446877298791
799 0.00109446877066688
800 0.00109446882039843
801 0.00109404920805423
802 0.00109404867829399
803 0.00109404872642282
804 0.00109403832489585
805 0.00109403451501255
806 0.00109403461942037
807 0.00109403462574353
808 0.00109403464909019
809 0.00109403465930402
810 0.00109589699676049
811 0.00109589700702406
812 0.00109589709632019
813 0.00109589711190261
814 0.00109589716952674
815 0.00109589726995485
816 0.00109589735948967
817 0.0010958974044144
818 0.00108714540000121
819 0.00108714626032458
820 0.00108714637217222
821 0.00108714641391794
822 0.00108714642193886
823 0.00108714762704643
824 0.001087147661474
825 0.00108018664695853
826 0.00108028173996111
827 0.00108028190909537
828 0.00108028191616261
829 0.00108028192496357
830 0.00108028185267828
831 0.00108028196700505
832 0.00108028208032018
833 0.00108028208033906
834 0.00108028275959438
835 0.00108028276386429
836 0.00108028326282006
837 0.00108069382270704
838 0.00108069384964238
839 0.00108069377741526
840 0.00108069584662251
841 0.00108069617821482
842 0.00108069639492148
843 0.0010806963922998
844 0.00108069641931041
845 0.00108069642646069
846 0.00108069643145137
847 0.00108069764537455
848 0.00108070117893777
849 0.0010807011766083
850 0.00108070118477999
851 0.00108070124176596
852 0.00108070132051918
853 0.00108070132125437
854 0.00108070134629235
855 0.00108070163257223
856 0.00108070223030965
857 0.00108070224373003
858 0.00108070202969102
859 0.00108070205352218
860 0.00108070205422472
861 0.00108070258549064
862 0.00108070258291426
863 0.0010807267882802
864 0.0010807268019053
865 0.00108072681890503
866 0.00108072700579975
867 0.00108072716321295
868 0.00108072723165908
869 0.00108072723530683
870 0.00108072734489983
871 0.00108072734829823
872 0.00108072749200661
873 0.00108072772507595
874 0.00108072772468137
875 0.00108072774060397
876 0.00108072841074547
877 0.00108072842265017
878 0.00108072868618425
879 0.00108072868625286
880 0.00107802540945223
881 0.00107775542474475
882 0.00107775544509447
883 0.00107775550406375
884 0.00107775553178979
885 0.00107775651110043
886 0.0010777568760918
887 0.00107775687401701
888 0.0010777568870266
889 0.00107775690247203
890 0.00107775710042191
891 0.00107775503316598
892 0.00107775505800078
893 0.00107775903365459
894 0.00107775904157847
895 0.00107775913072916
896 0.00107775916093478
897 0.00107775920013742
898 0.00107775951720113
899 0.00108244125865786
900 0.00108244398762847
901 0.0010824440705306
902 0.00108244505536037
903 0.00108244857687567
904 0.00108244974317673
905 0.00108244976156313
906 0.00108244976807526
907 0.00108244978441219
908 0.00108244962358794
909 0.00108268412652457
910 0.0010826844919718
911 0.00108227681277095
912 0.00108227681402284
913 0.00108227737673627
914 0.00108227738943323
915 0.00108227753081369
916 0.00108227787179116
917 0.00108227787708204
918 0.00108227971524832
919 0.00108227972382058
920 0.00108227975251141
921 0.00108227993865406
922 0.00108228030447632
923 0.00108228048985359
924 0.0010822805029409
925 0.00108228081000439
926 0.00108228098270846
927 0.00108228098273666
928 0.00108228146629585
929 0.00108228148945955
930 0.00108228149436229
931 0.00108228152101719
932 0.00108229205667776
933 0.0010822925365237
934 0.00108229253643954
935 0.00108229257198444
936 0.00108229257546433
937 0.00108223032933008
938 0.00108223044095879
939 0.00108223050764478
940 0.00108223050780176
941 0.00108223050820522
942 0.00108213953340086
943 0.00108213955949443
944 0.00108213956013214
945 0.00108213958208681
946 0.00108212018383091
947 0.00108209960605143
948 0.00108209961338246
949 0.00108209961903438
950 0.00108209968522832
951 0.00108209974877349
952 0.00108209975294438
953 0.00108209976424711
954 0.00108210009328502
955 0.00108210010148313
956 0.00108210014009558
957 0.00108210014208243
958 0.00108210048946478
959 0.00108210047332791
960 0.00108210159681543
961 0.00108210493508198
962 0.00108210493798633
963 0.00108210493913297
964 0.0010821049394858
965 0.00108210501026029
966 0.00108210510151507
967 0.00108210524916164
968 0.00108210643061191
969 0.00108210647301932
970 0.00108210651119678
971 0.00108210651156049
972 0.00108210651115015
973 0.00108210662917529
974 0.00108210731889602
975 0.00108210731886116
976 0.00108210736226044
977 0.0010821073648446
978 0.00108210750036619
979 0.00108210757809535
980 0.001082107579508
981 0.00108210757982463
982 0.00108210753524007
983 0.00108210753674665
984 0.00108210886480564
985 0.0010821089118187
986 0.00108210947837684
987 0.00108210955589438
988 0.00108210956533106
989 0.00108179556576027
990 0.00108179557048405
991 0.00108179561735899
992 0.00108179562375343
993 0.00108179562592636
994 0.00108179561851607
995 0.00108176862635312
996 0.00108176862834886
997 0.00108177104682028
998 0.00108177105459428
999 0.0010817713097746
1000 0.00108217449131498
1001 0.00108217449043679
1002 0.00108217449439829
1003 0.00108217449759596
1004 0.00108217450618886
1005 0.0010821745182068
1006 0.00108217457564597
1007 0.0010821745767875
1008 0.00108217458140536
1009 0.00108217461920534
1010 0.00108217429616086
1011 0.00108217524907461
1012 0.00108217535626332
1013 0.00108224085929165
1014 0.00108224086409692
1015 0.00108224153287217
1016 0.00108224153594616
1017 0.00108224154070946
1018 0.00108224162839754
1019 0.00108224167008175
1020 0.0010822418838945
1021 0.00108224166576165
1022 0.00108224167355964
1023 0.00108224169124971
1024 0.00108224182291039
1025 0.00108224181734951
1026 0.0010822466791216
1027 0.00108224668229484
1028 0.00108225207658519
1029 0.00108203628129933
1030 0.00108203628373893
1031 0.0010820362850541
1032 0.00108203631758452
1033 0.00108203632052262
1034 0.00108224973180149
1035 0.00108224965383563
1036 0.00108224965378856
1037 0.00108224978657812
1038 0.00108224978865645
1039 0.00108224980288285
1040 0.00108224982714411
1041 0.00108224983074612
1042 0.00108224983258154
1043 0.00108224986718408
1044 0.00108225010709906
1045 0.00108225016001851
1046 0.00108310422958779
1047 0.00108310422991797
1048 0.00108307007049779
1049 0.00108307008010566
1050 0.00108307007994513
1051 0.00108307011566366
1052 0.00108306599236974
1053 0.00108306622238796
1054 0.00108306625488575
1055 0.00108306626716459
1056 0.0010830668237316
1057 0.00108306682762827
1058 0.00108306681849624
1059 0.00108306980510675
1060 0.00108306980873518
1061 0.00108306981245132
1062 0.00108306983284012
1063 0.00108306983365769
1064 0.00108306983105111
1065 0.00108307072471825
1066 0.0010830707252838
1067 0.00108307093675686
1068 0.00108307095149329
1069 0.00108307100912386
1070 0.00108307101025829
1071 0.00108307110814199
1072 0.00108307109181682
1073 0.00108307110982286
1074 0.00108307111196004
1075 0.00108303388117781
1076 0.00108303389565001
1077 0.00108304535363257
1078 0.00108304572517315
1079 0.00108304646245028
1080 0.001083046462385
1081 0.00108304646285107
1082 0.00108304648419044
1083 0.00108304648368018
1084 0.0010830464862488
1085 0.00108304677365822
1086 0.00108304583138641
1087 0.00108303976492552
1088 0.00108303982687619
1089 0.00108308060624518
1090 0.00108308060579354
1091 0.00108308085544073
1092 0.00108308085078423
1093 0.00108308085466424
1094 0.00108272149760791
1095 0.00108272248565577
1096 0.0010827227089576
1097 0.00108272270895915
1098 0.00108272271410592
1099 0.0010827227374981
1100 0.00108272275633037
1101 0.00108272309915458
1102 0.00108272315833235
1103 0.00108272315587432
1104 0.00108272319317892
1105 0.0010827231978634
1106 0.00108269065194144
1107 0.00108269066064337
1108 0.00108269066733779
1109 0.00108269066641786
1110 0.00108269066755717
1111 0.00108269067828903
1112 0.00108269068153688
1113 0.00108269069219546
1114 0.00108269171875741
1115 0.00108269710226527
1116 0.00108266400198609
1117 0.00108266403205648
1118 0.00108266404251989
1119 0.00108266404349711
1120 0.0010826636813277
1121 0.00108266373954313
1122 0.00108266375024857
1123 0.00108266375181598
1124 0.00114435143708858
1125 0.00114435133348745
1126 0.00114435133239899
1127 0.00114435133255864
1128 0.00114216034975412
1129 0.00114216048345983
1130 0.00114216048179028
1131 0.00114216048447968
1132 0.00114216051902249
1133 0.0011421762063033
1134 0.00114218638064356
1135 0.00114218638233976
1136 0.00114218642181707
1137 0.0011421864218697
1138 0.00114218642104458
1139 0.00114195610830015
1140 0.00114195588200672
1141 0.00114195595733535
1142 0.00114195639412462
1143 0.00114196523631427
1144 0.00114196555503532
1145 0.00104141580621775
1146 0.001041415828547
1147 0.00104141583233908
1148 0.00104140963320187
1149 0.00104140964937605
1150 0.00104140964934873
1151 0.00104140965527755
1152 0.00104140966021449
1153 0.00104140966245803
1154 0.00104140971386601
1155 0.00104140971314348
1156 0.00104141012016012
1157 0.00104141018521564
1158 0.00104141019844817
1159 0.00104188650307901
1160 0.00104188650225168
1161 0.00104188711520337
1162 0.00104188709563147
1163 0.00103003352759345
1164 0.00103003353009345
1165 0.00103003369562327
1166 0.00103004141664331
1167 0.00103004142551666
1168 0.00103004144525842
1169 0.0010300414511919
1170 0.00103004156144015
1171 0.00103004180087463
1172 0.00103004169260634
1173 0.00103004254347416
1174 0.00103004299133169
1175 0.0010300430059933
1176 0.00103004307139209
1177 0.00103004307160437
1178 0.00103004307263799
1179 0.00103004388805594
1180 0.00103004388806149
1181 0.00103004388899497
1182 0.00103004391694261
1183 0.0010300439707196
1184 0.00103004417866659
1185 0.00103004420928254
1186 0.001030044209406
1187 0.0010300442139124
1188 0.0010300442174227
1189 0.00103004321233957
1190 0.00103004267455487
1191 0.00103004267650619
1192 0.0010300427422949
1193 0.00103004274531782
1194 0.00103004301946741
1195 0.00103004302241572
1196 0.00103004302422516
1197 0.00103004302702203
1198 0.00103005439417592
1199 0.00103005439464399
1200 0.00102981662447155
1201 0.00102981584149631
1202 0.0010298158412072
1203 0.0010298158412676
1204 0.0010298158432771
1205 0.00102981584049422
1206 0.00102981584020334
1207 0.00102981584206274
1208 0.00102981588761408
1209 0.00102981616930453
1210 0.00102981617388864
1211 0.00102981617406739
1212 0.00102981620944487
1213 0.00102981621299092
1214 0.00102981622105558
1215 0.00102981622056841
1216 0.00102969944676401
1217 0.00102969944809517
1218 0.00102969937108921
1219 0.00102969959040466
1220 0.00102947754204585
1221 0.00102947757045069
1222 0.00102947764952033
1223 0.0010294776520563
1224 0.00102947766036099
1225 0.00102948001846603
1226 0.00102948149951798
1227 0.00102948151884696
1228 0.00102948154785465
1229 0.00102948154721139
1230 0.00102948154842308
1231 0.00102948154957239
1232 0.00102948156823834
1233 0.00102947796218955
1234 0.00102947801291675
1235 0.00102947801402542
1236 0.00102947801330733
1237 0.00102946300376217
1238 0.00102946345739352
1239 0.00102946351577615
1240 0.0010294635348842
1241 0.0010294635373298
1242 0.00102946353762756
1243 0.00102946354165057
1244 0.00102962507847559
1245 0.00102962510908977
1246 0.00102962521136885
1247 0.0010296252127826
1248 0.00102962521289252
1249 0.00102962530262096
1250 0.00102962531986317
1251 0.00102962532429984
1252 0.00102962532456008
1253 0.00102962537957274
1254 0.0010296254047002
1255 0.00102962541664109
1256 0.00102962684298746
1257 0.00102962684355767
1258 0.00102962684634278
1259 0.00102962685438768
1260 0.0010296268561778
1261 0.0010296268610559
1262 0.00102962690441588
1263 0.00102962690609654
1264 0.00102962692272457
1265 0.0010296269201624
1266 0.00102962704308385
1267 0.00102962702957576
1268 0.00102962702951959
1269 0.00102964570670649
1270 0.00102963508503562
1271 0.00102963508510179
1272 0.00102963511000964
1273 0.00102963511843313
1274 0.00102963511823218
1275 0.00102963513126353
1276 0.00102963451227889
1277 0.00102964101794156
1278 0.00102964098301994
1279 0.00102964098653024
1280 0.0010296409871009
1281 0.00102964098639591
1282 0.00102964098687264
1283 0.00102964099374847
1284 0.00102906743308662
1285 0.00102906743419684
1286 0.00102906743497
1287 0.00102906743617859
1288 0.00102906761491339
1289 0.00102906908278877
1290 0.00102906913493017
1291 0.00102906914780498
1292 0.00102906920610257
1293 0.00102906927466351
1294 0.00102906927448876
1295 0.00102906932869251
1296 0.00102906940756253
1297 0.00102906943320913
1298 0.00102906943309322
1299 0.00102906943332548
1300 0.00102906944776926
1301 0.00102906944783587
1302 0.00102906945867054
1303 0.00102906946117387
1304 0.0010290694610422
};
\end{axis}

\end{tikzpicture}

%% file: Figures/vis_decay_pacman_singular_f_lambda.tex
\begin{tikzpicture}

\definecolor{darkgray176}{RGB}{176,176,176}
\definecolor{steelblue31119180}{RGB}{31,119,180}

\begin{axis}[
width=0.951\fwidth,
height=0.75\fwidth,
log basis x={10},
log basis y={10},
tick align=outside,
tick pos=left,
x grid style={darkgray176},
xmin=0.698612127441536, xmax=1866.55792073854,
xmode=log,
xtick style={color=black},
y grid style={darkgray176},
ymin=6e-3, ymax=3,
xlabel={\# collocation points},
ymode=log,
ytick style={color=black}
]
\addplot [semithick, steelblue31119180]
table {%
0 1.99999986283073
1 0.858627797537033
2 1.2126116694909
3 1.45399326827497
4 1.98532085283386
5 1.13888733618878
6 1.52190141273227
7 1.63352654014375
8 1.10173140772129
9 1.38901956847127
10 0.987028872605351
11 0.772419007316128
12 1.05028261714041
13 0.659058996734137
14 0.841496651100361
15 0.420731613294365
16 0.865533254326732
17 0.592556123255621
18 1.16078727706366
19 0.986079104863106
20 1.42420312051922
21 1.56290235789806
22 0.858709805152347
23 0.783015293301129
24 0.923728088862314
25 0.624781042139246
26 0.591550319282359
27 0.673382318020073
28 0.629869873309307
29 0.570647110538036
30 0.499699404993711
31 0.407725175954302
32 0.380779262589939
33 0.397766407775365
34 0.371068197562202
35 0.249300691382116
36 0.348304737627156
37 0.411329675773608
38 0.390279530276986
39 0.330843690547739
40 0.316135792432632
41 0.283923513016344
42 0.279774564294493
43 0.27892652328449
44 0.247314821290889
45 0.339202027782474
46 0.243846777111123
47 0.234353810032981
48 0.201352901239675
49 0.185948178167897
50 0.182621154444693
51 0.16759834107041
52 0.166082550463578
53 0.164004278805007
54 0.13080506339579
55 0.139379902116305
56 0.136373257291776
57 0.133582900600672
58 0.110775569170196
59 0.781852285342523
60 0.535806447611293
61 0.399881257099228
62 0.279550494408119
63 0.259640768891911
64 0.235441561607831
65 0.20719216871039
66 0.186939001553933
67 0.263854256180621
68 0.25664440964626
69 0.251455741807348
70 0.165258011659204
71 1.30718170590662
72 1.39725287955019
73 0.766377726323979
74 0.764377274492461
75 0.543405203074371
76 0.507306545336589
77 0.509286436527741
78 0.4913356089334
79 0.471161492846982
80 0.450582159371318
81 0.41905593480712
82 0.407911431046402
83 0.412075244272501
84 0.418987036328179
85 0.368772448894211
86 0.328049550972167
87 0.326922295589606
88 0.29227418111743
89 0.272305057407947
90 0.315288793046792
91 0.267599890281593
92 0.262983032725926
93 0.257441769357095
94 0.252005246690339
95 0.229506013275432
96 0.253231097220556
97 0.222779112702319
98 0.221954302637397
99 0.214588534259727
100 0.214952353750911
101 0.208287626476091
102 0.202294165784513
103 0.19598562133384
104 0.194070000407851
105 0.186484578802817
106 0.182212387784821
107 0.175267884906942
108 0.174184032719652
109 0.301282758206996
110 0.173372671251854
111 0.167743743286454
112 0.174438014894235
113 0.167751703403872
114 0.166844154610078
115 0.164485296783178
116 0.203884533748766
117 0.159211089852566
118 0.15397384755187
119 0.153872269350854
120 0.15342626092991
121 0.145582863001204
122 0.14188595253933
123 0.141500050577031
124 0.136556476186192
125 0.146455520594704
126 0.125781837966264
127 0.286983489260754
128 0.163919020228187
129 0.157297375748825
130 0.148882203639208
131 0.143376582902453
132 0.157876196026375
133 0.132573272791473
134 0.131768267432829
135 0.129008444676085
136 0.337294182575697
137 0.347013917618673
138 0.257160782094318
139 0.222531672883232
140 0.167422319728543
141 0.152901751671086
142 0.14863805892982
143 0.146633025992031
144 0.143569935661769
145 0.134900175609436
146 0.145326469453674
147 0.132860983129674
148 0.171062953036281
149 0.12327451448939
150 0.120050108063704
151 0.117434458972191
152 0.138369007377331
153 0.109212520237776
154 0.116285064261266
155 0.0978396368374947
156 0.098341338039676
157 0.0916853474525914
158 0.0902037351838093
159 0.085966827244955
160 0.0812962862095304
161 0.081102261409758
162 0.0807342584763873
163 0.0907324917004184
164 0.0773198395192063
165 0.0768769048710669
166 0.0760922587857491
167 0.0758262318072959
168 0.0969139339044807
169 0.0765977836992974
170 0.0753248441657781
171 0.0727932969867306
172 0.069812239899545
173 0.0694894297825605
174 0.0663363559867648
175 0.0661516619679286
176 0.0665578340101013
177 0.0675401626108427
178 0.064606992708837
179 0.0637231570129872
180 0.0626373673024831
181 0.0618445520813146
182 0.0800743122478326
183 0.0664532466302313
184 0.0575629883368666
185 0.073766623525468
186 0.060447727813716
187 0.0591040738271592
188 0.0587913067144231
189 0.0615803334367621
190 0.0573460077346299
191 0.0531313754885805
192 0.0539955147338645
193 0.0521785042248633
194 0.0920128975654965
195 0.0577857989585958
196 0.0580369073647376
197 0.0530826009246573
198 0.0526366459087946
199 0.0522073739171524
200 0.0507764564327655
201 0.0505315944466625
202 0.0468591057370773
203 0.116504666668959
204 0.0679330980396177
205 0.0568997386608529
206 0.0540622941995998
207 0.0543424566911243
208 0.0515185723458687
209 0.0478513148118451
210 0.0475378542044997
211 0.0477345791358763
212 0.051034835563677
213 0.0523313383063851
214 0.0455732424632997
215 0.0493219006049431
216 0.0451146543788124
217 0.0463004420505351
218 0.0464460604569953
219 0.0458771935191301
220 0.0451531394675741
221 0.0438502918396025
222 0.0460644278109876
223 0.043895155184074
224 0.0433262311338469
225 0.0432594379081908
226 0.042093725583091
227 0.0419997052385108
228 0.0419315080250414
229 0.0410178495440997
230 0.0409073307256553
231 0.0402332634175371
232 0.0441331212725563
233 0.0401823373902133
234 0.0398258945961895
235 0.040340527555328
236 0.0403321836081202
237 0.0383345361783311
238 0.0409700684707581
239 0.0379853339937193
240 0.0376418250143874
241 0.0371476255091882
242 0.0370920568474485
243 0.037219045253424
244 0.0361309653702226
245 0.035673394682291
246 0.0356676343971111
247 0.0390212925073766
248 0.0355335436572887
249 0.0351417028450036
250 0.0518132693935611
251 0.0330865842719721
252 0.0401987316815049
253 0.0331332036368998
254 0.0327470022772161
255 0.0324634395445056
256 0.0316377403211683
257 0.0316085375401861
258 0.0309939601285277
259 0.0290527817256737
260 0.028799424996669
261 0.255115859628703
262 0.163400170554115
263 0.16340512386731
264 0.141370983779644
265 0.109932847996998
266 0.102047906729681
267 0.088882278208018
268 0.0875558641077077
269 0.0875115380157585
270 0.0873657928605367
271 0.0866810902550248
272 0.0865259007020644
273 0.0817060066726883
274 0.0795197182242152
275 0.0864832298578448
276 0.0771409615359176
277 0.076842311171372
278 0.0765845701569562
279 0.0773803228718017
280 0.0750477997280156
281 0.0744340682408226
282 0.0700314589063069
283 0.0959141308663809
284 0.0700471142598136
285 0.0697514172703942
286 0.0641117101160638
287 0.0619113722744579
288 0.0598577439479426
289 0.0584490232706053
290 0.057741313455208
291 0.0575796513371543
292 0.0596547201270199
293 0.0675254062539811
294 0.0573153842873992
295 0.0563415385504227
296 0.0563011264301721
297 0.0542285769529217
298 0.062061921453212
299 0.0585358611270999
300 0.0540652312931598
301 0.0539789356719181
302 0.0541750532166416
303 0.0539666320901309
304 0.0533442539185546
305 0.0530481619531654
306 0.0537195680832937
307 0.0519006446661028
308 0.0514484392385077
309 0.0501528571330868
310 0.049890221375014
311 0.0561807910896676
312 0.0496784544144629
313 0.0486078354402086
314 0.0479509489598722
315 0.0478800091933032
316 0.0470225772254919
317 0.764298866498659
318 0.435129014206941
319 0.374337123177908
320 0.295671884595617
321 0.300318998469614
322 0.295490783534887
323 0.265896829960292
324 0.265772571326105
325 0.268951463939491
326 0.263459489362311
327 0.266045185911007
328 0.258969702252243
329 0.256725284075303
330 0.256858573245653
331 0.254745010910484
332 0.250530935497263
333 0.246774846942138
334 0.246991787938052
335 0.245067433670897
336 0.256637829444225
337 0.258301618205708
338 0.238492595155645
339 0.237152106712139
340 0.251157667741011
341 0.247258138533261
342 0.23650585751587
343 0.281598104434867
344 0.275277280697029
345 0.233558766266716
346 0.226457865549385
347 0.217104972358346
348 0.21219239564353
349 0.212347802113395
350 0.202380407787017
351 0.201150222780433
352 0.199655334427853
353 0.199231518601077
354 0.214907026371407
355 0.182817079344765
356 0.188185287848949
357 0.182112663861476
358 0.178507552654365
359 0.194176143978798
360 0.199124749553641
361 0.179161605694908
362 0.178409745475856
363 0.177113080076439
364 0.172623428751183
365 0.171194398017794
366 0.171827301828514
367 0.169405534862644
368 0.168573046183153
369 0.168684638491445
370 0.181619742684894
371 0.167935287835149
372 0.163253350779263
373 0.161131521652389
374 0.17797065667689
375 0.15933204627728
376 0.159079664383002
377 0.159043713465522
378 0.160577968927735
379 0.159827472882973
380 0.181600439233949
381 0.149000162486022
382 0.149213683030671
383 0.147447687901466
384 0.146561058939574
385 0.145942975420005
386 0.145538690250053
387 0.145047047767792
388 0.144717979771278
389 0.144241043374852
390 0.144227812643672
391 0.142814299082477
392 0.14278429156044
393 0.142250677310044
394 0.133652404819371
395 0.133014148761317
396 0.139374318820371
397 0.131855781854729
398 0.13139904378739
399 0.129443247667808
400 0.128537771373151
401 0.128139765514701
402 0.127769727544034
403 0.126902932468274
404 0.125983324734052
405 0.135343824010872
406 0.126320885351123
407 0.124051698923582
408 0.123131538942797
409 0.121159907568514
410 0.121024710515394
411 0.12324155872803
412 0.120913942306697
413 0.120181241351362
414 0.117658348789359
415 0.114018651970041
416 0.113899921234768
417 0.112320595742834
418 0.111049375023362
419 0.111650146787478
420 0.110621434607798
421 0.109984799687298
422 0.11505816585276
423 0.107236435243702
424 0.106899218590771
425 0.103698948379322
426 0.102859072500923
427 0.101712286272783
428 0.102238816427787
429 0.102316082032876
430 0.104619754024722
431 0.102197571681312
432 0.104034877278938
433 0.102366472943405
434 0.101808222177118
435 0.0988491547744472
436 0.0971834557060557
437 0.0958947595709325
438 0.0954137322171227
439 0.0945015179185375
440 0.093915371318544
441 0.0936365779670108
442 0.0919027941633457
443 0.0915355957592575
444 0.0910758700278499
445 0.0926733312381945
446 0.0903676013877431
447 0.0899106404150138
448 0.093333885037983
449 0.0882904532392931
450 0.0873394403619879
451 0.086386027348974
452 0.082741086361194
453 0.082155445604403
454 0.080740499477593
455 0.0783643092380495
456 0.077536988794844
457 0.0773625710085759
458 0.0838346699460188
459 0.0814153637028569
460 0.0776010003061534
461 0.0764530409803806
462 0.0857015072960477
463 0.0762436312489953
464 0.0745541325457831
465 0.0744206610808394
466 0.0728170286120407
467 0.0725579987645561
468 0.0721213916928533
469 0.0712879922290446
470 0.0706232527333839
471 0.0718383478402693
472 0.0704844872623379
473 0.070050760703096
474 0.0687265850486371
475 0.0684717296662175
476 0.069989950791272
477 0.067978166398297
478 0.066605786420014
479 0.0664571477245464
480 0.066440985852214
481 0.0661397534139705
482 0.0655879070356288
483 0.0727297460569873
484 0.0654089957978703
485 0.0648503626944171
486 0.0656767302120456
487 0.0649497993663233
488 0.0648030448549765
489 0.0645150096229733
490 0.0724376464672988
491 0.0645169331335947
492 0.0644158264910105
493 0.0646929523039511
494 0.0648744549767713
495 0.0641894476117591
496 0.0635223759466453
497 0.0632199332822222
498 0.0623920478115474
499 0.0623747521664874
500 0.0618090311316526
501 0.0613720107097206
502 0.0658683283958599
503 0.0669730682060947
504 0.0613336193832408
505 0.0615164312394115
506 0.0612970187848866
507 0.0600450689429584
508 0.0599192880685217
509 0.0637498433393164
510 0.0604141233451436
511 0.0598104580587376
512 0.0585610830757746
513 0.125360558042849
514 0.0656496384743161
515 0.063354040917408
516 0.0585693455303663
517 0.0655982089194129
518 0.0613448142226605
519 0.0573391320015206
520 0.0566632950260843
521 0.0560776646812264
522 0.0555475424116429
523 0.0552094104300177
524 0.0547038205666887
525 0.0548010190522657
526 0.0547284379930237
527 0.0545260201367069
528 0.053951201257356
529 0.0528300315317843
530 0.0544799193828946
531 0.0527977816998833
532 0.052726509270395
533 0.0524108600175394
534 0.0519468170956388
535 0.0549271454492686
536 0.0517458272418865
537 0.0511006932447043
538 0.0509032857287596
539 0.0508633820113272
540 0.0510696189435592
541 0.050811705394516
542 0.0503846750928171
543 0.0502251213843867
544 0.050132582793055
545 0.0492559261067664
546 0.0490122190191383
547 0.048885354282347
548 0.0486262609001341
549 0.0523350738131474
550 0.0480152252590143
551 0.0477193200395819
552 0.0471507638872449
553 0.0470319247936377
554 0.0550434432703557
555 0.0470306994976831
556 0.0468468844883585
557 0.0610891006051215
558 0.0470805823237771
559 0.046779059513767
560 0.0467136437664385
561 0.0588798293992516
562 0.0473470271151141
563 0.0468449315580893
564 0.0466621813445928
565 0.0466475475897417
566 0.0459914358875746
567 0.0459144513633296
568 0.0457951767884099
569 0.0463720472783002
570 0.0460888292933638
571 0.0550536346488689
572 0.0457911355813092
573 0.0491663864391761
574 0.0457643082292209
575 0.0454932510193019
576 0.047553301318906
577 0.0447242654974391
578 0.122609416796232
579 0.101356193990773
580 0.0841376492674892
581 0.0670156748375701
582 0.0613393543440645
583 0.0581882984376799
584 0.0571481811735946
585 0.0532601418321221
586 0.0506618574595959
587 0.0500624033642472
588 0.0483051476751989
589 0.0474232537087672
590 0.0469195036231024
591 0.0466640068137949
592 0.0478757556644617
593 0.0496140559337661
594 0.0514460959518509
595 0.045982078191887
596 0.0457446813765235
597 0.045675893772508
598 0.0466669701940678
599 0.0456128016978985
600 0.0453739936787685
601 0.0489986796837217
602 0.0521777325977054
603 0.04480523550461
604 0.0448220004025339
605 0.0445142122341742
606 0.044177185960414
607 0.0440498997391744
608 0.043836732098398
609 0.183701537085502
610 0.171429715964179
611 0.0651815651960438
612 0.0705008074242847
613 0.0631068883858724
614 0.0508133400768399
615 0.0490743926962493
616 0.0487736549959594
617 0.0467682799283933
618 0.048939044351242
619 0.0467293433770828
620 0.046490583038194
621 0.0458895781603302
622 0.0457220778848941
623 0.0472715907300474
624 0.0454739962197148
625 0.0445030502479707
626 0.0453424182382744
627 0.0506201264509004
628 0.0452001165982458
629 0.0429590112791291
630 0.0426051981752041
631 0.0423490051855946
632 0.0493472274166724
633 0.046224371643742
634 0.0437301300216673
635 0.0424577757879314
636 0.0423863539656878
637 0.0441004383385607
638 0.0416128355553371
639 0.0421043572319802
640 0.043028877925625
641 0.0413978406870937
642 0.0411280236797303
643 0.0405227372064766
644 0.0403498137582475
645 0.0402865959146867
646 0.0431547912506134
647 0.0400123753617179
648 0.0434548127026633
649 0.039870237185219
650 0.0389797887261644
651 0.04175329536179
652 0.040310168717013
653 0.0389003115626962
654 0.0377167613042887
655 0.0421955898124421
656 0.0397850542105877
657 0.0373951074535926
658 0.0371479814490302
659 0.0371485110527573
660 0.036865610850864
661 0.0367740866811292
662 0.0378989825740686
663 0.0366897151651888
664 0.0365098287420997
665 0.0373384739305097
666 0.0372609715761316
667 0.036269247592423
668 0.0361885725334643
669 0.0361276734242557
670 0.036120060198951
671 0.0360139358171172
672 0.0359583152769251
673 0.0359667929957855
674 0.0359486167399964
675 0.0358980217160061
676 0.0354976200988203
677 0.0351096897136042
678 0.0350907746183594
679 0.0344383800712784
680 0.0344113867434612
681 0.0452263861785372
682 0.0344892859836571
683 0.0338434505431136
684 0.0338269779408819
685 0.033654394160959
686 0.033254763499551
687 0.0330945948031467
688 0.0329124965561123
689 0.0329713882714265
690 0.0329366619267568
691 0.0329034938774599
692 0.0325879344234168
693 0.0323647399263798
694 0.0321711041648804
695 0.0321656882612205
696 0.03687029395762
697 0.0321511891098649
698 0.0320439641159273
699 0.031884844713307
700 0.0324676420252791
701 0.031430717333078
702 0.0313888410176163
703 0.0308662965380158
704 0.0308949601074498
705 0.0307537422991084
706 0.0306753378734123
707 0.0329795720905175
708 0.032800077618069
709 0.0304148789080166
710 0.0307850577099539
711 0.0298853833956173
712 0.032442958195444
713 0.0301756712355804
714 0.0293252329927578
715 0.0655728026511844
716 0.0534240213158947
717 0.04932486967496
718 0.0418648365159884
719 0.0399816452578203
720 0.0366802399018255
721 0.0346408883998811
722 0.0346011715723055
723 0.0402041304773098
724 0.0411401531778618
725 0.0387290009486773
726 0.0341307387031329
727 0.0338225326673205
728 0.0334221315167159
729 0.0362798771912617
730 0.0332395591689225
731 0.0348885209616878
732 0.0330151444549667
733 0.0328763107919137
734 0.0380591227756331
735 0.0326829628734884
736 0.0323694014356048
737 0.0322102323667775
738 0.0319898680242222
739 0.0317803097450607
740 0.0317471563710096
741 0.0317592565262079
742 0.0316110912171386
743 0.0310750735522169
744 0.0308689065664889
745 0.0305900665563247
746 0.0305473245804297
747 0.0303056037650885
748 0.0299760611856775
749 0.0299586020732308
750 0.0612960711650133
751 0.0666853524644383
752 0.0492871777378316
753 0.0430520924238017
754 0.0405736545105389
755 0.0403973449399095
756 0.0397783772960857
757 0.0366774505894956
758 0.03581810502065
759 0.0372676059872143
760 0.0356929569623907
761 0.0356921904995898
762 0.0356885557286148
763 0.035468660946797
764 0.0357865010659176
765 0.0351275319974171
766 0.0349810650003206
767 0.0348899530222513
768 0.0348790484313539
769 0.0370574738219896
770 0.0348522329208447
771 0.0347396191866618
772 0.0346938231896712
773 0.0349371588956368
774 0.0345165594932881
775 0.0344124330186836
776 0.0342145792464186
777 0.0341516433630383
778 0.0341236024467276
779 0.0340865664154266
780 0.0338054160008235
781 0.0338023362035419
782 0.0333180126423422
783 0.0331027632022662
784 0.0330421995511439
785 0.0329502293579216
786 0.0327488095956702
787 0.0327372204521156
788 0.0327054717094819
789 0.037335674318887
790 0.0350544399978444
791 0.0337196601685806
792 0.0326907819389349
793 0.0329233293925952
794 0.0325477599165349
795 0.0324812822656923
796 0.032439140378433
797 0.0323384144075798
798 0.0322639142934248
799 0.0315803516459319
800 0.0315467120957554
801 0.0313816595329742
802 0.0317335083414066
803 0.0310161492074563
804 0.030993600767741
805 0.0313573975656629
806 0.0306672572378545
807 0.0307188245751314
808 0.0305229745574281
809 0.0304464713008074
810 0.0301215473298128
811 0.0298584046478408
812 0.0298253265337694
813 0.0296957225068921
814 0.0296797443986127
815 0.0296706559173854
816 0.0345820333690822
817 0.031357904651428
818 0.0296260742140937
819 0.0295489896159173
820 0.0293988642372013
821 0.0293987914453128
822 0.0291245513519117
823 0.0285508078502993
824 0.0292604829901567
825 0.0284049328150108
826 0.0285732117321491
827 0.028315744542732
828 0.028244557479426
829 0.028106259468843
830 0.0299903775618308
831 0.0279535726959807
832 0.02735548103961
833 0.0272330941288311
834 0.0272080863677797
835 0.0269103123733807
836 0.0266778657812155
837 0.0266363400729318
838 0.0265533801485992
839 0.0293553820540935
840 0.0266762462243834
841 0.0260817744823411
842 0.0271996903710338
843 0.0260670369652626
844 0.026015846227494
845 0.0260710167122376
846 0.0261953195115411
847 0.0260896300317191
848 0.026004606242704
849 0.025962269850032
850 0.0259187461259292
851 0.0258790124078812
852 0.0255207717771451
853 0.0254471451407347
854 0.0254729246113245
855 0.0253330630159815
856 0.0252600757274118
857 0.0256013332125716
858 0.0267003592757478
859 0.0255405822164077
860 0.0253674450071783
861 0.0253526601265091
862 0.0253470673910535
863 0.0255831509501329
864 0.0265730624203457
865 0.0252733785007119
866 0.0251635635138279
867 0.0251629536803916
868 0.0249702139897536
869 0.0248784830962001
870 0.0257888349326713
871 0.0247420428361083
872 0.0245284820035593
873 0.0266636064220251
874 0.0243463489002286
875 0.0243081321902826
876 0.0252620855144485
877 0.0245705790152277
878 0.0242803494309737
879 0.0242447369235222
880 0.0240537263427862
881 0.0243083374448561
882 0.0237891134677915
883 0.0238262287609669
884 0.0236676162592951
885 0.0235773026503577
886 0.0235563027248631
887 0.0235151368676856
888 0.0232116760467458
889 0.0258210507061435
890 0.0232278113037706
891 0.0230606970017432
892 0.0230447418731098
893 0.0228171905945842
894 0.0228506212293458
895 0.0227581045757544
896 0.0227228565697013
897 0.0264327292020129
898 0.0226620359693305
899 0.0226342339730882
900 0.0226191181023674
901 0.02243393427524
902 0.0221769067771858
903 0.0221538754889501
904 0.0220687971801865
905 0.0220595376873941
906 0.0219360098331271
907 0.0219106089251399
908 0.0218771934803221
909 0.0216353548760527
910 0.0215828214954925
911 0.021515762371036
912 0.021446078588335
913 0.0214276406909777
914 0.0237915173444483
915 0.0211975076673275
916 0.0248958087148745
917 0.0210937338011792
918 0.0214288287731866
919 0.0229664965993554
920 0.0210739442937132
921 0.0209586418891183
922 0.022095494582381
923 0.0209257085308093
924 0.0209013598812187
925 0.0208080315508767
926 0.0208030785582822
927 0.0206763255484101
928 0.0205318742528867
929 0.0205117204797359
930 0.0203786940932631
931 0.0203577539768228
932 0.0203063146227707
933 0.0202926241427545
934 0.0224429009720181
935 0.0201967129328334
936 0.0201001734417321
937 0.020081860827762
938 0.0205434215506324
939 0.0202287758790497
940 0.0200718171706743
941 0.0200198445271554
942 0.0198228657183582
943 0.0233562466737241
944 0.0199462830605127
945 0.0197545821596043
946 0.0196832576802034
947 0.0196564268381883
948 0.0222967229660847
949 0.0195714427679284
950 0.0195154396297078
951 0.0195493817137235
952 0.0194406583867624
953 0.019726836443506
954 0.0193873638881172
955 0.0193118033685717
956 0.0192587458110392
957 0.0192139913616882
958 0.0192085013805173
959 0.0191397719113463
960 0.0190413000095579
961 0.0203220674131612
962 0.019036083740547
963 0.0189514941109718
964 0.0189032345506954
965 0.0187804576679909
966 0.0187182428637354
967 0.0187181024164547
968 0.0208841236389629
969 0.0191180439197588
970 0.0186738454766007
971 0.0187777765940254
972 0.0184307117140395
973 0.0182654958008135
974 0.0190124547574744
975 0.0183282516486863
976 0.0182396374094523
977 0.0186898815094253
978 0.0181797019019436
979 0.0181437152771088
980 0.0180606818612055
981 0.0180344485844496
982 0.0179336675248221
983 0.0178160979321492
984 0.0177218827196889
985 0.0177151194327749
986 0.0176129438931402
987 0.0174780645964194
988 0.0174355957681281
989 0.0172345374046055
990 0.0170836064542021
991 0.0170602081712946
992 0.0170006742026517
993 0.0169859703223846
994 0.0171662134371458
995 0.0168858757841285
996 0.0187218534785551
997 0.0168537914566265
998 0.0166976483959244
999 0.0191299035939417
1000 0.016617964526441
1001 0.0165912668763251
1002 0.0165669522939625
1003 0.0165142089986998
1004 0.0164369993003765
1005 0.0173178783998586
1006 0.0181110159140855
1007 0.0168656536495681
1008 0.0172631680594867
1009 0.0163970003257781
1010 0.0187660928647956
1011 0.016346444483942
1012 0.0159817026896033
1013 0.0159224631844895
1014 0.0158214449880395
1015 0.0167170195214438
1016 0.0158201365848421
1017 0.0157982916041481
1018 0.0157868864748057
1019 0.015737903473498
1020 0.0156950874394881
1021 0.0156538380645701
1022 0.0157663997673756
1023 0.015630397066406
1024 0.0156060456545377
1025 0.0155802804976092
1026 0.0158736863543058
1027 0.0154383403393758
1028 0.0153682900598376
1029 0.0153117466974121
1030 0.0153615947332038
1031 0.0153104488346704
1032 0.0152800138948247
1033 0.0152513688827758
1034 0.0149605594808993
1035 0.0148533862887067
1036 0.0148287308564812
1037 0.0146464711073104
1038 0.0146298791493236
1039 0.0146186974080708
1040 0.0145882648152338
1041 0.0144677534705859
1042 0.0143829933355202
1043 0.0143718166350066
1044 0.0143565887748488
1045 0.0143399113623641
1046 0.0143025224393461
1047 0.0142391216603045
1048 0.0153613355360322
1049 0.0157394072794132
1050 0.0141812599951486
1051 0.0141997510940378
1052 0.0141899379380731
1053 0.0141510360737709
1054 0.0141208543127717
1055 0.0139954461112766
1056 0.0139941921645543
1057 0.0139670328362194
1058 0.0139640512036209
1059 0.0139520852923464
1060 0.01490373875201
1061 0.0139302464715739
1062 0.0139301748318642
1063 0.0157175270155869
1064 0.0139144679266892
1065 0.013905704630391
1066 0.0138438537882963
1067 0.0138345161873296
1068 0.0137839624479978
1069 0.0137678769997174
1070 0.0158371837462953
1071 0.0145316996210452
1072 0.0136712242233563
1073 0.0136442647381311
1074 0.0135833509917471
1075 0.0135707795305637
1076 0.0135342070166389
1077 0.0136024144784427
1078 0.0139194479422331
1079 0.0141317689335065
1080 0.0134967586994139
1081 0.0134961091115669
1082 0.0134857385017602
1083 0.0134438657938747
1084 0.01343783972365
1085 0.013434039326618
1086 0.0134077998741928
1087 0.0133953004798614
1088 0.0153729514157866
1089 0.0133866736604175
1090 0.01337951476091
1091 0.0133646546168323
1092 0.0142351409919923
1093 0.0132246816091971
1094 0.0132085785865302
1095 0.0131438555041641
1096 0.0130883855536955
1097 0.016898766434134
1098 0.0137626662398208
1099 0.0131186627808769
1100 0.013081910378555
1101 0.0130617208524601
1102 0.0130499157898954
1103 0.0129157003419245
1104 0.0131166591032233
1105 0.0131807211831473
1106 0.0128300131687163
1107 0.0130943088524764
1108 0.0128243016653411
1109 0.0127255507961679
1110 0.0127062198502582
1111 0.0139445650146595
1112 0.0129927476038285
1113 0.0125086432585218
1114 0.0124845993776749
1115 0.0124261792964565
1116 0.0123762904176824
1117 0.0123644916245729
1118 0.0122242677350537
1119 0.012219564795982
1120 0.0122147104534164
1121 0.0122069846289534
1122 0.0121451965573128
1123 0.0121346945201056
1124 0.012056266016396
1125 0.0124362021498463
1126 0.0120219640021286
1127 0.0128150037122919
1128 0.0132593445216357
1129 0.0120105937482634
1130 0.0119853965555673
1131 0.0122169464184408
1132 0.01196912826315
1133 0.0119529203233856
1134 0.0119496177061931
1135 0.0121285111270701
1136 0.0118866367742075
1137 0.0117163211212986
1138 0.0119980219161066
1139 0.0116940293564058
1140 0.011641021410986
1141 0.0123377611486398
1142 0.0115934540075019
1143 0.0115586662864893
1144 0.0115107984834012
1145 0.0114550562973624
1146 0.0114296757883378
1147 0.0114119165659002
1148 0.0114279265768308
1149 0.0119920341063175
1150 0.0114121468618975
1151 0.0114008474004532
1152 0.0113737697323027
1153 0.0124975629092691
1154 0.0113508191626372
1155 0.0136569949188314
1156 0.0112673293661176
1157 0.0112251743254209
1158 0.0113661432811919
1159 0.0112100997074474
1160 0.0111907016704101
1161 0.0121149593728851
1162 0.0120737430282293
1163 0.0111812978340025
1164 0.01116194442763
1165 0.0111368253468632
1166 0.0110887357450058
1167 0.0130060988626213
1168 0.0111301447535305
1169 0.0135290380548166
1170 0.011122487390115
1171 0.0119614514808777
1172 0.0119712816391439
1173 0.0119441443350715
1174 0.011859609684834
1175 0.010857915811048
1176 0.0108479097413036
1177 0.0110427334242829
1178 0.0108034133568707
1179 0.0107883403113995
1180 0.0108013281987035
1181 0.0107695371867094
1182 0.0107656979805179
1183 0.0107520488442591
1184 0.0108301100378455
1185 0.0107952866268215
1186 0.0107218562776938
1187 0.0116316462203796
1188 0.0106082524542841
1189 0.0104789970035212
1190 0.0104784804792521
1191 0.0104461823265644
1192 0.0108044970390442
1193 0.0104400334280083
1194 0.0104391147832963
1195 0.0104272275156927
1196 0.0115145654438038
1197 0.0103548918712046
1198 0.0116363764959332
1199 0.0102516484552546
1200 0.0317461657884046
1201 0.0147827168538387
1202 0.0135118354182036
1203 0.011135586966933
1204 0.0140301647899776
1205 0.0115250034651012
1206 0.0102518262983705
1207 0.0122185154066906
1208 0.0102455085646176
1209 0.010402628891182
1210 0.0102331016162192
1211 0.0102183899560857
1212 0.0102168099179847
1213 0.0101900916594598
1214 0.0101583485689971
1215 0.0101263533864465
1216 0.0101141896821169
1217 0.0103737224365818
1218 0.0101075153374462
1219 0.011647262974481
1220 0.0100773010408095
1221 0.010022383302624
1222 0.00999776941244247
1223 0.00999273634268026
1224 0.0112382199286955
1225 0.00996690405967999
1226 0.0104075334521093
1227 0.0101650146004851
1228 0.0102367167372063
1229 0.00991275006134731
1230 0.0097698051991772
1231 0.00976960418099612
1232 0.00976836118698077
1233 0.00974707972087109
1234 0.00969443000647014
1235 0.00968774444810223
1236 0.0105661399716931
1237 0.00968073588636813
1238 0.00964999571147091
1239 0.010177030792176
1240 0.0100271769410234
1241 0.0096028259625587
1242 0.00966319993278517
1243 0.0110305055830115
1244 0.00956299884513286
1245 0.00951822236082469
1246 0.00949598755859276
1247 0.0104272215018731
1248 0.00946375769033097
1249 0.00945106113714314
1250 0.00939807368107115
1251 0.00935759968948337
1252 0.0093343487261952
1253 0.00926559275751911
1254 0.00920473982787363
1255 0.0167057347156938
1256 0.00934795208735229
1257 0.00919702945206327
1258 0.00916927530363943
1259 0.0091556549457068
1260 0.00915424017489356
1261 0.00915210711314558
1262 0.00945263276614683
1263 0.00913237965381263
1264 0.00908343925149368
1265 0.00932352737611679
1266 0.00984183828893863
1267 0.00905048134456875
1268 0.0090314875402367
1269 0.0090044960986002
1270 0.00896854744209749
1271 0.00928109562786473
1272 0.00896639376287234
1273 0.00894188223902518
1274 0.00893801393146102
1275 0.00892345214925971
1276 0.00890440114361919
1277 0.00888215385524045
1278 0.0102582351556679
1279 0.00903560382986882
1280 0.00891214529514627
1281 0.00884318638517665
1282 0.00937868366754015
1283 0.00884287539749093
1284 0.00883164438656103
1285 0.0088312800509103
1286 0.00879930816081821
1287 0.00878035524263011
1288 0.00917370304062302
1289 0.00880711423807268
1290 0.00927064303071081
1291 0.00875971699369416
1292 0.00875674146273395
1293 0.00927362345733015
1294 0.00996465293890181
1295 0.00872990031410943
1296 0.00872150565632087
1297 0.00871245816340855
1298 0.00870868605800941
1299 0.00907675526992331
1300 0.00891641712177671
1301 0.00870109604396738
1302 0.00867793817271636
1303 0.00867637711556487
1304 0.00866231791825867
};
\end{axis}

\end{tikzpicture}

%% file: Figures/vis_decay_pacman_singular_f_sol.tex
\begin{tikzpicture}

\definecolor{darkgray176}{RGB}{176,176,176}
\definecolor{steelblue31119180}{RGB}{31,119,180}

\begin{axis}[
width=0.951\fwidth,
height=0.75\fwidth,
log basis x={10},
log basis y={10},
tick align=outside,
tick pos=left,
x grid style={darkgray176},
xmin=0.698612127441536, xmax=1866.55792073854,
xmode=log,
xtick style={color=black},
y grid style={darkgray176},
ymin=1e-4, ymax=3,
xlabel={\# collocation points},
ymode=log,
ytick style={color=black}
]
\addplot [semithick, steelblue31119180]
table {%
0 1.99926184372634
1 0.858385478857202
2 0.794696716326469
3 0.380570540501563
4 0.146443575602331
5 0.176268873203302
6 0.0790621769791575
7 0.0619290843430922
8 0.0289894503492396
9 0.0352871249640225
10 0.047420176757317
11 0.0157626840225671
12 0.0130017871568155
13 0.0113962092700459
14 0.013677134398776
15 0.0148233777687117
16 0.0321013603526346
17 0.0206171246786169
18 0.0211415661433674
19 0.0223089671565351
20 0.0439063717155288
21 0.0513202344352035
22 0.0186849407345036
23 0.0100357762915926
24 0.0165065802294511
25 0.0120504114940365
26 0.0121483636102533
27 0.0146013792160873
28 0.0127890136696331
29 0.0116975551743108
30 0.0106971944526857
31 0.0118672948295568
32 0.0166366749146283
33 0.0108729069134614
34 0.0108607903834954
35 0.0134183777708929
36 0.0118375577090826
37 0.00410297449478558
38 0.00485386696496182
39 0.00496163493772572
40 0.00599637735401348
41 0.00698519391537511
42 0.00607118734645917
43 0.00832447203295672
44 0.00843032768657359
45 0.00369488339949897
46 0.0037077498353435
47 0.00361029997402795
48 0.0037330478669706
49 0.00372935227703208
50 0.00372729217661161
51 0.00374140349899998
52 0.00372046228611755
53 0.00370314713074693
54 0.00364247320051736
55 0.00365020606417432
56 0.00365318331419684
57 0.00365299281810016
58 0.00365294940069649
59 0.00479149477631458
60 0.00601475445903299
61 0.00639272297142601
62 0.00602062828745997
63 0.006049946287221
64 0.00604962370285866
65 0.00632984794841684
66 0.00634723129255343
67 0.00597573004761731
68 0.00596513592347137
69 0.00595468934012588
70 0.00596403832600778
71 0.00997269974188941
72 0.0126967954743451
73 0.0137985213640723
74 0.0134938442101862
75 0.0112150905780499
76 0.0104895708195123
77 0.010308881622531
78 0.0100132716297576
79 0.00950211099248333
80 0.00934128390490896
81 0.00938515483454627
82 0.00826080826534947
83 0.00794285501645131
84 0.00749632186548443
85 0.00585251736149295
86 0.00581320098116622
87 0.00585099245204024
88 0.00602386540966382
89 0.00591774180356763
90 0.00567883572669148
91 0.00488786403269836
92 0.005071723672321
93 0.00508971844967676
94 0.00507331014086576
95 0.00508923719429699
96 0.00509926201638589
97 0.00510966179236982
98 0.00515267179600842
99 0.00512639049763175
100 0.00528214827579254
101 0.00533385345434012
102 0.00534411951884839
103 0.00541419412110389
104 0.00535920940981316
105 0.00536606439901255
106 0.00535634352288938
107 0.00540393915954795
108 0.00550911354248984
109 0.00409559626568612
110 0.00408295486333943
111 0.00406621687158215
112 0.00406831754675707
113 0.00407031473267994
114 0.00407265547038982
115 0.00407784977821368
116 0.00408058630567121
117 0.00408167359766853
118 0.00405091040247685
119 0.00393967994999311
120 0.00393334681028468
121 0.00393299484944332
122 0.00397948359692935
123 0.00398052502143265
124 0.00400344976928935
125 0.00402729328750162
126 0.00404222238181684
127 0.00410262248533311
128 0.00411128524224802
129 0.00411214612964317
130 0.00411882828814303
131 0.00414001300601496
132 0.00414337860753466
133 0.00414667808184932
134 0.0041482497474794
135 0.00414785061710998
136 0.0023032097441722
137 0.00231723474829248
138 0.00236115744463428
139 0.00240061551488679
140 0.00230156998675124
141 0.00219793602572804
142 0.00183047271794146
143 0.00184977396840158
144 0.0017721178724901
145 0.00167027249797624
146 0.00166074754491707
147 0.00156278369832008
148 0.00158731914632471
149 0.00159194962915032
150 0.00159539008587051
151 0.00159649575675247
152 0.00159708680852133
153 0.00159459496834757
154 0.00159874810741178
155 0.00160012212133398
156 0.00159047921678224
157 0.00160064155123751
158 0.00160116205210947
159 0.00160359538887445
160 0.00159934016998919
161 0.00159942723621875
162 0.00160787854778266
163 0.00160886314162068
164 0.00161115759637309
165 0.00161166406203139
166 0.00161187747827984
167 0.00161451064201157
168 0.00156265949694823
169 0.00160594806225167
170 0.00160011658669745
171 0.0016001211265253
172 0.00159767288630808
173 0.0015979370646122
174 0.0015986646254742
175 0.00159770960923411
176 0.0016174132754323
177 0.00161959059643313
178 0.00159935312964121
179 0.00160021212010331
180 0.00158590061380304
181 0.00158626001155304
182 0.00158855949581493
183 0.00159128473397252
184 0.00159181305035183
185 0.00159460741001594
186 0.00159601111021712
187 0.00160583509438994
188 0.00160705135392258
189 0.00161924239227473
190 0.0016258039920074
191 0.00162663623358328
192 0.00164167705879881
193 0.00166832046817444
194 0.00154143101313098
195 0.00154204502521749
196 0.00154266168847972
197 0.00154279314533001
198 0.00148754189724931
199 0.00148655132082576
200 0.00148678598679353
201 0.00148694253438042
202 0.00149373334106517
203 0.0010355965859794
204 0.00100071748459607
205 0.000999663419441132
206 0.00100235342871935
207 0.00100288954890981
208 0.00100131898526046
209 0.0010022278562194
210 0.000996401828155191
211 0.000989196645956758
212 0.000973762311373694
213 0.000944684387244266
214 0.000920702802321482
215 0.000904440820559405
216 0.000885166537193172
217 0.000885234536974977
218 0.000885296324343976
219 0.000885409488574407
220 0.000885469071100431
221 0.000885608549699324
222 0.000880708046409495
223 0.000878688113820525
224 0.000876939531265331
225 0.000878146157806015
226 0.00087825380458928
227 0.000878052381144467
228 0.00087807764362946
229 0.000878169697500653
230 0.000878160105670212
231 0.000878217470499099
232 0.000878100560447681
233 0.0008779154448324
234 0.000877959332504563
235 0.000878005485367428
236 0.000878075172195736
237 0.000878100544963178
238 0.000878472744745729
239 0.000878702994476566
240 0.000877819711967875
241 0.00087787607635148
242 0.000879489451571125
243 0.000880643410021698
244 0.000880897241418932
245 0.000881977761079122
246 0.000882051819784069
247 0.000882105030949321
248 0.000882214921527957
249 0.000882298121394065
250 0.00087957686996365
251 0.000879108326258482
252 0.000879210560263521
253 0.000879285723273782
254 0.000879373839111075
255 0.000879342870908406
256 0.000879893565227885
257 0.000879705786531737
258 0.000880136757551719
259 0.00088019066166134
260 0.000880253789632413
261 0.00151972997304162
262 0.00156275260992755
263 0.00151921760739326
264 0.00154258353922265
265 0.00153380030797612
266 0.00153496025942079
267 0.00150652247965599
268 0.00149884034307379
269 0.00149891156217308
270 0.00149903841788745
271 0.00149919836431556
272 0.00149944524844292
273 0.00149949627775103
274 0.00149978046339139
275 0.00150000917826665
276 0.00150037333937547
277 0.00150043421550516
278 0.00150070644411548
279 0.00150090567406913
280 0.0015010930240642
281 0.00149609932181693
282 0.00150063598954286
283 0.00150331643257662
284 0.0015070090623599
285 0.00150714555135978
286 0.00150754057451774
287 0.00150779158885106
288 0.00150813092122193
289 0.00150820365659499
290 0.00150851885249192
291 0.00150866052035648
292 0.00150868465035803
293 0.00150780188437594
294 0.00150816509177987
295 0.00150528077518364
296 0.0015053298638108
297 0.00150535614728398
298 0.00150533132061481
299 0.0015051827823096
300 0.00150513661677509
301 0.00150522036755429
302 0.00150410148348756
303 0.00150401077792894
304 0.00150243705520636
305 0.00150248162673083
306 0.00150329602691457
307 0.00150478756118178
308 0.00150483564505488
309 0.00150494062810869
310 0.00150496522231536
311 0.00150499291831085
312 0.00150500659989916
313 0.00150495895840042
314 0.00150449584595158
315 0.00150460316690326
316 0.00150462583156585
317 0.00570257643027183
318 0.00559420615658501
319 0.0055501199441339
320 0.00547633425954719
321 0.0055734454242371
322 0.0054696817562947
323 0.0054285862857808
324 0.00534225731208204
325 0.0051156611136538
326 0.00494072414557922
327 0.00494309245509439
328 0.00494472640383825
329 0.0049398753276273
330 0.00503847804872781
331 0.00504028574806781
332 0.00504076592662384
333 0.00505145145853336
334 0.00506655440490555
335 0.00506455597661737
336 0.00506568620175307
337 0.00506703358754135
338 0.00506824213079238
339 0.00511960054143401
340 0.00512131789601855
341 0.00512390829445608
342 0.00512692453315933
343 0.00512186691291094
344 0.00511180796492261
345 0.00506490483860556
346 0.0050674673371669
347 0.00505583211585114
348 0.00505608232780186
349 0.00507288345764834
350 0.00507583070572282
351 0.00507845564086584
352 0.00508046473815771
353 0.00508341208761065
354 0.00508448186364441
355 0.00508289168185061
356 0.00508354047499382
357 0.0050955443582803
358 0.00509834420836475
359 0.00509926098047764
360 0.00511744729895325
361 0.00508109989813033
362 0.00507948464395391
363 0.00507597288079942
364 0.00508019772628043
365 0.00513615551457858
366 0.00511791758682345
367 0.00511771090845436
368 0.00511243962845875
369 0.00511332554889887
370 0.00511409643994054
371 0.00511538290461444
372 0.00511842999346013
373 0.00511964061545078
374 0.00277564450998291
375 0.00274578461882458
376 0.0027137770707959
377 0.00270601876213949
378 0.00273484395785939
379 0.00264970242519658
380 0.00262895741745339
381 0.00257937685888221
382 0.00258953361424297
383 0.00258744365241936
384 0.00255207444002892
385 0.00254504726540783
386 0.0025107664503019
387 0.0024916022702699
388 0.0024734807208624
389 0.00246338126954604
390 0.00244813028797064
391 0.00241000993292206
392 0.00227742576868462
393 0.00226542170787525
394 0.00220463338877419
395 0.00218794904710662
396 0.00218327929119821
397 0.00216757734148398
398 0.00208996032873854
399 0.00208188920529873
400 0.00205979282422786
401 0.00205963242361085
402 0.00200687793203747
403 0.00201190243282579
404 0.00201489333917393
405 0.00201853541447972
406 0.00202145561181033
407 0.00201023973466885
408 0.00200465587514098
409 0.00200176667137497
410 0.0019669543087848
411 0.00195362953107514
412 0.00193178635127689
413 0.00192830276526568
414 0.00190142118864345
415 0.00190533987325359
416 0.00190493969333394
417 0.00190508801900102
418 0.00190408605005299
419 0.00190427744471022
420 0.0019044911734718
421 0.00189088353667954
422 0.0018909584333342
423 0.00189107889953366
424 0.00188159879076144
425 0.00189723504575223
426 0.00189776574817446
427 0.00189692458892643
428 0.00190129594354382
429 0.00190742593165494
430 0.00190936266680741
431 0.00191061980321949
432 0.00191237410692113
433 0.00191895356472371
434 0.00191910034493858
435 0.00191939791242923
436 0.0019194991408944
437 0.00191973899133635
438 0.00191763894721841
439 0.00200891069775344
440 0.00201580195145112
441 0.0020165819890281
442 0.00201668802328547
443 0.00201663291171905
444 0.00202017575263391
445 0.00202020647757051
446 0.00202036180792953
447 0.00202065372293903
448 0.00202065739345358
449 0.00202077416145485
450 0.00198068547000219
451 0.00198102015866475
452 0.00198102265026989
453 0.00198110886453851
454 0.0017553446179468
455 0.0017554119071026
456 0.00175582397799112
457 0.00175595280743157
458 0.00175538978020318
459 0.00175894536823051
460 0.00176018953420742
461 0.00175158819004051
462 0.00175763821005126
463 0.00176562846850303
464 0.00176843140519289
465 0.00177008297353454
466 0.00177045478226323
467 0.00177225800361414
468 0.00177235260968689
469 0.00177232144914119
470 0.00177416390070473
471 0.00177412375518049
472 0.00177434539241972
473 0.00177430785674293
474 0.00177440975632837
475 0.00177438557816489
476 0.00178211896639002
477 0.00178221582534466
478 0.00178227882000526
479 0.00177704449239857
480 0.00177714478366586
481 0.00177723431946641
482 0.00179647759852508
483 0.0017951225065389
484 0.00179478006443889
485 0.0017948874271938
486 0.00180195745978651
487 0.00180275885384829
488 0.00180280411725797
489 0.00180298024060477
490 0.00180306563954913
491 0.00180315988479229
492 0.00180930628225595
493 0.00181315702794049
494 0.00182418451257593
495 0.00183085032808505
496 0.00183033680054301
497 0.00183044524512077
498 0.00183053154159296
499 0.00183103945788998
500 0.00183112407443975
501 0.0018313435536228
502 0.00183160332259269
503 0.00183182895317224
504 0.00183204802187942
505 0.00183223600917737
506 0.00183242383606252
507 0.00183090057435087
508 0.00183003272911564
509 0.00183545862345147
510 0.00184539041882781
511 0.00184695655081857
512 0.00184801105204691
513 0.00141949965090216
514 0.00141943346040463
515 0.00141944017421192
516 0.00141943966760727
517 0.00141947831931599
518 0.00141952284446878
519 0.00141957170083018
520 0.00141958156178956
521 0.00141961495621545
522 0.00141961735256801
523 0.00142004759838588
524 0.00142006916122828
525 0.00142017375444259
526 0.00142029052697934
527 0.00142038083176943
528 0.00142061103085378
529 0.00142073362739914
530 0.00142077263067686
531 0.00142080085252394
532 0.00142082284357059
533 0.00142087778652322
534 0.0014209293387184
535 0.00142136471959153
536 0.00142165635649771
537 0.00142055870584845
538 0.00142057687910557
539 0.00141443912162131
540 0.00141445617649893
541 0.00141444928760637
542 0.00141449504255198
543 0.00141451978959195
544 0.00141463648091722
545 0.00141478170520248
546 0.00141550097169119
547 0.0014163674176364
548 0.00141640884402838
549 0.00141639706674479
550 0.00141649650605624
551 0.00141651127224518
552 0.00141659708556308
553 0.00141661760992373
554 0.00141673458197356
555 0.00141681191906606
556 0.0014168282329915
557 0.00140557569896205
558 0.00139946434111715
559 0.00139947398585527
560 0.00139951844102326
561 0.00140804562450481
562 0.00141057997041305
563 0.00140739884878049
564 0.0014074276029622
565 0.00140746675349213
566 0.00140750897510122
567 0.00140753162857721
568 0.00140753846204289
569 0.00141694928719094
570 0.00141695176189427
571 0.0014169640199635
572 0.00141696730253515
573 0.00141699292178366
574 0.00141702177199399
575 0.00141706265544705
576 0.00141706810341535
577 0.0014170797226265
578 0.00135977773158302
579 0.00136027686078211
580 0.00136044144768843
581 0.00136064479738884
582 0.00136072305115253
583 0.00136133453520659
584 0.00136137667597414
585 0.00136141080921037
586 0.00136163592619476
587 0.00136166493590006
588 0.00136167641705709
589 0.00136219217095912
590 0.00136227802312328
591 0.00136233176879186
592 0.00136234131307322
593 0.00136235395436568
594 0.00136236285807745
595 0.00136237793793725
596 0.00136244895058257
597 0.00136261773334501
598 0.00136263846584295
599 0.00136265748381836
600 0.00136268558746222
601 0.00136270273501871
602 0.00136272553026329
603 0.00136274808581116
604 0.00136277489118819
605 0.00136280872073069
606 0.00136284277570664
607 0.00136291932726307
608 0.00136300489531516
609 0.00176504746391437
610 0.00176760190164726
611 0.00177239922224448
612 0.00176738590043057
613 0.00176730056893915
614 0.00176919420267896
615 0.000923927954235015
616 0.000923998365752965
617 0.000924006943762379
618 0.000924027836073815
619 0.000924053213214204
620 0.000924061377877372
621 0.000924102667315463
622 0.000924114780368024
623 0.000924127022641663
624 0.000924134265591681
625 0.000924160500164195
626 0.000924305468223308
627 0.000924608641472036
628 0.000925042607538007
629 0.000925121662235462
630 0.000925123354299728
631 0.000925164160471503
632 0.000920968527399202
633 0.000923154728100339
634 0.000924338132661351
635 0.000921854665630661
636 0.000921881085309106
637 0.000921898075595395
638 0.000921917848364151
639 0.000921923159572735
640 0.000921936469286511
641 0.000921946468729562
642 0.000921972385009573
643 0.000921978995990669
644 0.000922119069558303
645 0.000922191693079055
646 0.000922205212056326
647 0.000922219621962928
648 0.000922253188886391
649 0.000922292897828747
650 0.000922305780305832
651 0.000922330611543121
652 0.000922358731016537
653 0.000922391580233217
654 0.000922406835361711
655 0.000922427227735723
656 0.000922449115357438
657 0.000922473008250879
658 0.000922522512639468
659 0.000922556316324208
660 0.000922549086916913
661 0.000922854949302332
662 0.000922888076375195
663 0.000922939177893767
664 0.0009229728213902
665 0.000923039383277313
666 0.000923066133994288
667 0.000923195725976989
668 0.000923224311273785
669 0.000922016034501372
670 0.000922046858455428
671 0.000922053794659261
672 0.000922089423487815
673 0.000922940407070305
674 0.00092297563968935
675 0.000922991931943784
676 0.000923062606052127
677 0.000923156741569287
678 0.000923192244609794
679 0.000923199162352395
680 0.000923206095208906
681 0.000923131087760654
682 0.00092354854258403
683 0.000923585850150399
684 0.000923639985289659
685 0.000923847466165029
686 0.000923857878638135
687 0.000923880387055931
688 0.00092390821335897
689 0.000923945241568802
690 0.00092388698948298
691 0.000923907457400563
692 0.000923913123005882
693 0.000923936304510153
694 0.000925352933077495
695 0.000925368303002383
696 0.000925379051421427
697 0.000925396319771732
698 0.000925412230983413
699 0.000925452521360004
700 0.000925483763353885
701 0.00092552133953383
702 0.000925814550694071
703 0.000925820082874118
704 0.000925815579247979
705 0.000925809951433765
706 0.000925860907296627
707 0.000925872953777329
708 0.000925886013614541
709 0.000925903371258752
710 0.000921305958181184
711 0.000921373573863571
712 0.000921383709058032
713 0.000921394654342267
714 0.000921403231285645
715 0.00086480376643272
716 0.000867033622921287
717 0.000867127394618139
718 0.000867173776078589
719 0.000867281815301491
720 0.000867856963460367
721 0.000868284127722774
722 0.000868443225358506
723 0.000868463534763553
724 0.000868485408921149
725 0.000868504226981148
726 0.000868521202528783
727 0.000868541267384337
728 0.000868440018497063
729 0.000868451951583582
730 0.000868463449531287
731 0.000868519255486255
732 0.000868573484430479
733 0.000868711725113824
734 0.000868731587476024
735 0.000868767947258764
736 0.000868784132218581
737 0.000867599865618685
738 0.000867622509956201
739 0.000867731920608739
740 0.000867748074342645
741 0.000869159927048946
742 0.000869198002095839
743 0.000869307010575415
744 0.000869311322262645
745 0.000869339870776642
746 0.000869348285059246
747 0.000869366392755699
748 0.000869362554495989
749 0.000871793418062872
750 0.000668584698466201
751 0.000669655290689963
752 0.000670435320097473
753 0.000669944440744752
754 0.000669814342247665
755 0.000670208395743721
756 0.00066924779380928
757 0.000669906311528745
758 0.000669037885185908
759 0.000667739517570265
760 0.000665586483244773
761 0.000660635343695981
762 0.000658966546782658
763 0.000656095902971554
764 0.000653994621752751
765 0.000650175226197813
766 0.00065033565103878
767 0.000642835811940579
768 0.000641721842524756
769 0.000635488155009689
770 0.000629524917910906
771 0.000627209056639266
772 0.000622898943269501
773 0.000621221580104248
774 0.000619491289794327
775 0.000617969160238152
776 0.000616876204572847
777 0.0006140369978751
778 0.000611138771800857
779 0.000611235549465405
780 0.000610857927857644
781 0.000607473580003237
782 0.000605747156225078
783 0.000603751782984219
784 0.000602259812901185
785 0.000598988648542464
786 0.000597682959047496
787 0.000596336301837974
788 0.000593590160727464
789 0.000588410128701922
790 0.000576840522877209
791 0.000569542759877084
792 0.000562008020071447
793 0.000562118018873958
794 0.000562256030434138
795 0.00056155852418982
796 0.000558542289372133
797 0.000555257487434124
798 0.000552123437157048
799 0.000550982703950265
800 0.000548392089612992
801 0.000547863185987207
802 0.000543848597718721
803 0.0005414996127584
804 0.000538204527122721
805 0.000534795052131942
806 0.00053215074339974
807 0.000530238437616459
808 0.000529008228978833
809 0.000527431425030089
810 0.000525195835822112
811 0.000523983555644092
812 0.000524028074115446
813 0.000522702940619624
814 0.000523146691129517
815 0.000519971901769889
816 0.000513459254879667
817 0.000506380672403273
818 0.000501118493036357
819 0.000496137953247322
820 0.000493454971063123
821 0.000493095080297712
822 0.000491514294621131
823 0.000490590214214714
824 0.000486205558499542
825 0.000482130088937716
826 0.000482322247828293
827 0.000482392386299546
828 0.000481171896478561
829 0.000477825423331524
830 0.000477616144832549
831 0.000477124716415034
832 0.000476903188481037
833 0.000476207616891378
834 0.00047474487805288
835 0.000470420903118329
836 0.000471068269252584
837 0.00046960535206364
838 0.000468480488258871
839 0.000469278148900854
840 0.000469775038488551
841 0.000469366876222788
842 0.000468622928456508
843 0.000467955954449595
844 0.00046751030917247
845 0.000467447665673193
846 0.000467723755417371
847 0.000467899340619304
848 0.000468143271855848
849 0.000466863432800002
850 0.000465684456268489
851 0.00046535767886291
852 0.000461439756044735
853 0.000461439317617662
854 0.000458341906336868
855 0.000454664436014163
856 0.00045335155880788
857 0.000457651341026155
858 0.000458272113273317
859 0.000458669797650746
860 0.000458391140800263
861 0.000458061297163415
862 0.00045521427254358
863 0.000452974581058418
864 0.000450798478139669
865 0.000448603192617414
866 0.000444714084111508
867 0.000444253996791266
868 0.000444327389799382
869 0.00044411444628123
870 0.000443174742002173
871 0.000441677160324394
872 0.000438971993487858
873 0.000434177448887407
874 0.000430773690344388
875 0.000430113930822484
876 0.000429308964066344
877 0.000428927201484353
878 0.000427986536310598
879 0.000426736713744491
880 0.000424762544705803
881 0.000423427655112718
882 0.000421445241492791
883 0.000420995294026838
884 0.000420430886197254
885 0.00041982288293041
886 0.000418155404443787
887 0.000418318032972431
888 0.000418382501381487
889 0.000413513396400722
890 0.000408160504591359
891 0.000402446570494819
892 0.000401915946016906
893 0.000401611774238253
894 0.000403703742564598
895 0.000403716711813384
896 0.000403771716292245
897 0.000403771667674024
898 0.000403854498884804
899 0.000403857329841162
900 0.000403901065742618
901 0.000403903148863405
902 0.00040391953092267
903 0.000403977032603864
904 0.000403979773373031
905 0.000403985884527058
906 0.000404499826066385
907 0.000404513052024935
908 0.000404519861351038
909 0.000404520243955431
910 0.000404521712316197
911 0.000404375907515298
912 0.000404380011511218
913 0.000401407576238855
914 0.000401396213783878
915 0.000401385654268838
916 0.000401381766973463
917 0.000401461755294541
918 0.000401463404540836
919 0.000401462374478134
920 0.000401461162698569
921 0.000401466791172922
922 0.000401848567322061
923 0.000402082314038932
924 0.000402119019379032
925 0.000402123749912331
926 0.000402187767047035
927 0.000402189755558613
928 0.000402194388948507
929 0.000402196720328485
930 0.000402181750430941
931 0.000405603033114321
932 0.000405583167814028
933 0.000405605671058629
934 0.000405609367473136
935 0.000405613865469556
936 0.000405599907091547
937 0.000405605487885152
938 0.000405607192793367
939 0.000405609379470206
940 0.000405611065309008
941 0.000405611744263457
942 0.000405765023583804
943 0.000405752675999738
944 0.000405745054753437
945 0.000405747105384657
946 0.000405751989148495
947 0.000405752677002491
948 0.000405760209528427
949 0.000405344880530345
950 0.000405346973074039
951 0.000405352582308893
952 0.000405356572759974
953 0.000405624003559701
954 0.000405440940226764
955 0.000405198017013886
956 0.000405345528685652
957 0.000405627465359437
958 0.000405608723072826
959 0.000405609872252244
960 0.000402418484029332
961 0.000402419380932328
962 0.000402420993186103
963 0.000402425198434697
964 0.000401040815955045
965 0.000400986736022513
966 0.000400997734245179
967 0.000400998649177087
968 0.000401247557593964
969 0.000401397536581638
970 0.00040140423370949
971 0.0004014094119833
972 0.000401418238738183
973 0.000401424355315205
974 0.000401426336236188
975 0.000401431920649564
976 0.000401444953105345
977 0.000401444931366068
978 0.000401446007801232
979 0.000401447996458248
980 0.000401449288078837
981 0.000401447391491283
982 0.00040144866430003
983 0.000401453691682541
984 0.00040145641091871
985 0.000401462761463023
986 0.000401467798164412
987 0.000401471563249323
988 0.000402687611664954
989 0.000402694910284662
990 0.000402697210517888
991 0.000402700361949893
992 0.00040270210920379
993 0.000402704074189741
994 0.000402706805220254
995 0.000402709080676078
996 0.000402710043049925
997 0.000402711028352654
998 0.000401959937112117
999 0.000401967303813366
1000 0.000401974146173267
1001 0.000402274864526597
1002 0.000402292307717933
1003 0.000402293536838294
1004 0.000402326257398777
1005 0.000402336155349015
1006 0.000402337164209232
1007 0.000402396861140408
1008 0.00040239614264248
1009 0.000402395819849133
1010 0.00040247210493205
1011 0.000402429978618946
1012 0.000402434034555199
1013 0.000402506494867838
1014 0.000402508678547631
1015 0.000402512633816188
1016 0.000402526818107374
1017 0.000402530478190943
1018 0.000402527332908909
1019 0.00040252913680372
1020 0.000402546876181575
1021 0.00040255041829218
1022 0.000402552155819635
1023 0.000402554069401595
1024 0.000402618249073505
1025 0.000401078060922044
1026 0.000401040539532826
1027 0.000401026002478821
1028 0.000401027422693989
1029 0.000401044044790799
1030 0.000401540471870909
1031 0.000404226017828524
1032 0.000404253056127901
1033 0.000404237268449847
1034 0.000407266249318083
1035 0.000407268343142309
1036 0.000407269996571702
1037 0.000407110658865784
1038 0.000407111878991562
1039 0.000407212872579943
1040 0.000407213659893713
1041 0.000407217173889696
1042 0.000407239162442896
1043 0.000407236584179071
1044 0.000407240099435824
1045 0.000407263792139401
1046 0.000407265764939546
1047 0.000407265810203783
1048 0.000407267335825301
1049 0.00040726900088317
1050 0.000407270607692856
1051 0.000407677206717527
1052 0.000406836986656556
1053 0.000406837511728764
1054 0.000406883349768461
1055 0.000406882386879803
1056 0.000406883853759754
1057 0.000406886408474194
1058 0.000406911135498844
1059 0.000406863987016615
1060 0.000406872261677682
1061 0.000406870749612764
1062 0.000406872994728746
1063 0.000406877209592094
1064 0.000406889057605553
1065 0.000406890968270734
1066 0.000406902164080347
1067 0.00040748470588714
1068 0.000408172557449293
1069 0.000408175561123048
1070 0.000408176208507749
1071 0.0004080954068586
1072 0.000408150190714163
1073 0.000408152363252978
1074 0.000408154804739547
1075 0.000408157182285263
1076 0.00040816178020997
1077 0.000408162914162902
1078 0.000408163561851582
1079 0.000408167563866524
1080 0.000408149724151485
1081 0.000408604143389057
1082 0.000408606113765475
1083 0.000408610495811756
1084 0.000408616360674907
1085 0.000408617438210301
1086 0.000408661684320188
1087 0.000408662349428823
1088 0.000408666633193833
1089 0.000408672067986227
1090 0.000408672951244027
1091 0.000408676708726574
1092 0.000409763105425398
1093 0.000409919838537443
1094 0.000409925389014854
1095 0.000410186187353556
1096 0.000410196897589366
1097 0.000308678537593643
1098 0.000308673891642686
1099 0.000308669629631941
1100 0.000306045266122101
1101 0.000306046672560178
1102 0.000306047798900089
1103 0.000307836545244022
1104 0.000307842474603248
1105 0.000307837685390222
1106 0.000307839356258555
1107 0.000307843569169908
1108 0.000307847388484328
1109 0.000307849588915721
1110 0.00030785251544696
1111 0.000307853376945388
1112 0.000307854336589752
1113 0.000307855289856551
1114 0.000307855041601579
1115 0.000308841722013176
1116 0.000308970935045449
1117 0.000308975619936591
1118 0.000309139062568464
1119 0.000309149343367787
1120 0.000309184702110388
1121 0.000309186459262589
1122 0.000309193881082859
1123 0.000309177119925952
1124 0.000309178086036033
1125 0.000309177369358649
1126 0.000309170925074875
1127 0.000309172068959418
1128 0.000309173322477818
1129 0.000309174582370009
1130 0.000309171672073338
1131 0.00030900614675633
1132 0.000309008716154358
1133 0.000309323983118892
1134 0.000309274445871699
1135 0.000310451649012444
1136 0.000310131564448701
1137 0.000310498511917556
1138 0.000310499102214479
1139 0.000310499875840531
1140 0.000310502005214097
1141 0.000310502256151146
1142 0.000310500432604277
1143 0.000310499870275427
1144 0.000310500641689693
1145 0.000310502183494155
1146 0.000310507425575812
1147 0.000312469946705418
1148 0.000312469454298636
1149 0.000312469689124129
1150 0.000312467633601043
1151 0.000312502763718525
1152 0.000312504983555728
1153 0.000312492523820707
1154 0.000312485038782784
1155 0.000315752353813759
1156 0.000316963347880117
1157 0.000316979167424902
1158 0.000316995877018611
1159 0.000317003083503353
1160 0.00031700410449842
1161 0.000317012676177342
1162 0.000317018245759471
1163 0.00031702150076951
1164 0.000317022409151768
1165 0.000317025220294642
1166 0.000317027191045982
1167 0.000319505107944318
1168 0.000319799380437802
1169 0.000319799469339133
1170 0.000319804493842168
1171 0.00031980491111594
1172 0.000319805450822663
1173 0.000319806005231627
1174 0.000319806545652446
1175 0.00031980842683299
1176 0.000319720447144567
1177 0.000319723417047113
1178 0.000319722826390478
1179 0.000319723748892997
1180 0.0003197292709749
1181 0.000319760941443814
1182 0.000319761585723333
1183 0.000319761039447863
1184 0.000319698862809137
1185 0.000319811504812417
1186 0.00031979965861062
1187 0.000319779437810475
1188 0.000319768729804437
1189 0.000319767944458871
1190 0.000319768208295157
1191 0.00031985238502874
1192 0.000319852872325832
1193 0.000319853322682029
1194 0.000319855027021365
1195 0.000321922978038902
1196 0.000322005083076782
1197 0.000322091303098571
1198 0.000322092403685748
1199 0.00032209350556256
1200 0.000343264527677078
1201 0.000375718786328694
1202 0.000360597518452854
1203 0.000352947350851229
1204 0.000218378722311563
1205 0.000218378821381204
1206 0.000218378872165026
1207 0.000218379036069027
1208 0.000218379209911523
1209 0.0002183472661621
1210 0.000218317357188802
1211 0.000218317959733483
1212 0.000218318659632066
1213 0.000218299384379561
1214 0.000218300262811555
1215 0.000218305810871344
1216 0.000218306003995528
1217 0.00021830609902973
1218 0.000218306170320703
1219 0.000218306929052892
1220 0.000218307844204624
1221 0.000218308918396248
1222 0.000218313323916641
1223 0.000218314027712774
1224 0.000218314599951475
1225 0.000218315321672824
1226 0.000218315652785739
1227 0.000218315970976546
1228 0.000218316383453487
1229 0.000218316712484734
1230 0.000218317137529178
1231 0.000218318231766768
1232 0.00021831888633228
1233 0.000218320267764138
1234 0.000218320829908247
1235 0.000218321512116981
1236 0.000218324092562838
1237 0.000218323147616495
1238 0.000218323595570835
1239 0.000218324465943498
1240 0.000218325557406418
1241 0.000218326782036371
1242 0.000218328240905619
1243 0.000218329417777108
1244 0.000218331284378426
1245 0.000218274417408626
1246 0.00021827257419349
1247 0.000218273432200045
1248 0.000218274336149404
1249 0.000218264231244047
1250 0.000218266143839685
1251 0.000218271946155868
1252 0.000218272696689503
1253 0.000218273165155214
1254 0.000218273572119676
1255 0.000218278840145469
1256 0.000218366499009193
1257 0.000218380399862772
1258 0.000218358141290453
1259 0.00021835838253792
1260 0.000218358578712552
1261 0.000218361324931804
1262 0.000218361681000312
1263 0.000218361980649062
1264 0.000218362779232484
1265 0.000218363491072848
1266 0.000218364418996142
1267 0.00021836557014776
1268 0.000218365617711047
1269 0.000218366591706154
1270 0.000218366737320785
1271 0.000218367154620314
1272 0.000218367600681946
1273 0.000218368400610069
1274 0.000218368570023886
1275 0.000218368548780878
1276 0.000218369252431794
1277 0.000218369738154811
1278 0.000218370134555057
1279 0.000218370721073224
1280 0.000218371324224975
1281 0.000218371937766415
1282 0.000218372471311179
1283 0.000218372928051824
1284 0.00021839462100659
1285 0.000218394842500524
1286 0.00021839502880705
1287 0.000218395110455072
1288 0.000218395301471164
1289 0.00021839555849712
1290 0.000218395817318084
1291 0.00021839619195263
1292 0.000218426162144514
1293 0.000218425471125716
1294 0.000218420133799135
1295 0.000218416031434465
1296 0.000218348374259492
1297 0.000218349058707323
1298 0.000218349536143636
1299 0.000218372855025128
1300 0.00021837399521818
1301 0.000218363197458826
1302 0.000218363505523733
1303 0.000218363828678569
1304 0.000218363914983533
};
\end{axis}

\end{tikzpicture}

%% file: Figures/vis_highdim_decay_Lerror.tex
\begin{tikzpicture}

\definecolor{darkgray176}{RGB}{176,176,176}
\definecolor{darkorange25512714}{RGB}{255,127,14}
\definecolor{forestgreen4416044}{RGB}{44,160,44}
\definecolor{lightgray204}{RGB}{204,204,204}
\definecolor{steelblue31119180}{RGB}{31,119,180}

\begin{axis}[
width=0.951\fwidth,
height=0.75\fwidth,
legend cell align={left},
legend style={
  fill opacity=0.8,
  draw opacity=1,
  text opacity=1,
  at={(0.97,0.97)},
  anchor=north east,
  draw=lightgray204
},
log basis x={10},
log basis y={10},
tick align=outside,
tick pos=left,
title={$\sup_{\lambda \in \Lambda} |\lambda(u-s_n)|$},
x grid style={darkgray176},
xmin=7.94288539149424, xmax=1260.24731651288,
xmode=log,
xtick style={color=black},
y grid style={darkgray176},
ymin=0.0641278387873326, ymax=33.0448600094758,
ymode=log,
ytick style={color=black}
]
\addplot [semithick, steelblue31119180, mark=x, mark size=3, mark options={solid}]
table {%
10 2.42303936650923
12 1.71095886782282
16 0.937356953402855
20 0.90107483949668
26 0.619600211055925
33 0.505434404607982
42 0.630590638012904
54 0.360096222921896
69 0.333749110315956
88 0.262244201795006
112 0.238468201982236
143 0.202136727579301
183 0.159310716072739
233 0.143563365686319
297 0.167451991851379
379 0.125422657040154
483 0.119768373071707
616 0.107669640326563
785 0.10395860782539
1001 0.0909097050322742
};
\addlegendentry{$w = 10^0$}
\addplot [semithick, darkorange25512714, mark=x, mark size=3, mark options={solid}]
table {%
10 3.34743775498087
12 3.20214041762726
16 1.99454086878041
20 1.94465928958183
26 1.00088520451607
33 0.787760986449186
42 0.438861517345828
54 0.392323316051545
69 0.328960223165673
88 0.252568510694587
112 0.257943886284764
143 0.220946098734231
183 0.183946662905484
233 0.180756500625115
297 0.141605137971638
379 0.12672232392481
483 0.126026391882277
616 0.112363681629986
785 0.0986243631153627
1001 0.0851769208246935
};
\addlegendentry{$w = 10^3$}
\addplot [semithick, forestgreen4416044, mark=x, mark size=3, mark options={solid}]
table {%
10 24.878751602198
12 23.2034878818923
16 19.7207321356746
20 14.2494222397268
26 8.20337578458276
33 5.48516925698243
42 3.91560068637682
54 2.63974825263067
69 1.40476684401148
88 0.814749940639558
112 0.611597067345663
143 0.380328272165173
183 0.379712173652564
233 0.26168905224263
297 0.212245184836277
379 0.198752076846063
483 0.19311614575048
616 0.140131269321117
785 0.135336968190643
1001 0.109619415552217
};
\addlegendentry{$w = 10^5$}
\end{axis}

\end{tikzpicture}

%% file: Figures/vis_highdim_decay_error.tex
\begin{tikzpicture}

\definecolor{darkgray176}{RGB}{176,176,176}
\definecolor{darkorange25512714}{RGB}{255,127,14}
\definecolor{forestgreen4416044}{RGB}{44,160,44}
\definecolor{lightgray204}{RGB}{204,204,204}
\definecolor{steelblue31119180}{RGB}{31,119,180}

\begin{axis}[
width=0.951\fwidth,
height=0.75\fwidth,
legend cell align={left},
legend style={
  fill opacity=0.8,
  draw opacity=1,
  text opacity=1,
  at={(0.03,0.03)},
  anchor=south west,
  draw=lightgray204
},
log basis x={10},
log basis y={10},
tick align=outside,
tick pos=left,
title={$\sup_{x \in \Omega} |(u-s_n)(x)|$},
x grid style={darkgray176},
xmin=7.94288539149424, xmax=1260.24731651288,
xmode=log,
xtick style={color=black},
y grid style={darkgray176},
ymin=0.000120549730687209, ymax=0.454476560967287,
ymode=log,
ytick style={color=black}
]
\addplot [semithick, steelblue31119180, mark=x, mark size=3, mark options={solid}]
table {%
10 0.31257219493124
12 0.186998020328551
16 0.187605060390021
20 0.215556255759441
26 0.166594244688612
33 0.177383664025113
42 0.197789192776478
54 0.182960041220278
69 0.165483873981191
88 0.157432758635618
112 0.170257189309543
143 0.072510900959698
183 0.0764046912454492
233 0.0734470416347708
297 0.0363408077356482
379 0.0362359641231542
483 0.0365709189469285
616 0.0142527346969188
785 0.0137323781935741
1001 0.0133617672113902
};
\addplot [semithick, darkorange25512714, mark=x, mark size=3, mark options={solid}]
table {%
10 0.112947472378206
12 0.16821010631283
16 0.0607586301234031
20 0.0316887605773144
26 0.0425253382297561
33 0.0169680151346983
42 0.0097770355218274
54 0.00862347227024185
69 0.00596936218969679
88 0.00628388960505011
112 0.00654112951568542
143 0.00451717325742496
183 0.00338032840148306
233 0.00170409164529972
297 0.00202722104127617
379 0.00146312304847807
483 0.00114880835778131
616 0.00100711794362374
785 0.000920708492610167
1001 0.000698652193953375
};
\addplot [semithick, forestgreen4416044, mark=x, mark size=3, mark options={solid}]
table {%
10 0.181306806927847
12 0.286748407554399
16 0.141777682456115
20 0.103691409017957
26 0.0620787342861844
33 0.0309443547797053
42 0.0172943849386269
54 0.0100021131747268
69 0.00511671399251967
88 0.00345138557606939
112 0.00243442703618868
143 0.00160664557027834
183 0.00107859700639201
233 0.000774248299726832
297 0.000623218939358017
379 0.000450920688435108
483 0.000308101875326372
616 0.000279785797406173
785 0.000186131414251944
1001 0.000175277993105905
};
\end{axis}

\end{tikzpicture}

%% file: chapters/08_outlook.tex
\section{Conclusion and outlook} \label{sec:outlook_ideas}

This paper considered the approximation of solutions of linear PDEs via symmetric kernel collocation with help of
PDE-$\beta$-greedy strategies.
After deriving estimates on the Kolmogorov $n$-widths of the sets of
linear functionals of those BVP, an abstract analysis of
greedy kernel methods was applied to derive worst-case optimal approximation
rates.
Extending an analysis of target-data dependent greedy kernel algorithms for standard interpolation to symmetric kernel collocation, it was possible to derive convergence rates for the full scale of values of $\beta$, with increasingly fast rates
going from target-data-independent ($\beta=0$) to for target-data dependent ($\beta>0$) algorithms.
These algorithms yield point sets which are adapted to both the domain $\Omega$ and the right hand side
of the considered BVP. 

The experiments demonstrated that for smooth target functions the PDE-greedy algorithms behave very favourably.
This even holds in comparison to FEM, for low dimensional problem where the FEM can be applied.

An important insight for the PDE-greedy schemes compared to pure function approximation is revealed by our study:
In the PDE context there is a crucial difference 
between the functional error $\sup_{\lambda \in \Lambda} | \lambda(e_n)|$ and the
residual error $\| e_n \|_{L^\infty(\Omega}$ (which coincide in the case of plain function approximation, i.e. $L=Id$).
Especially the schemes for $\beta>0$ directly aim at driving the functional error to zero. But even if the
theoretical convergence rates guarantee an identical decay for both quantities, this decay, however,
may differ by a large constant.
In order to drive the residual error down, 
additional preference to choosing boundary functionals via weighting of the
boundary selection indicator turned out to be essential. 

Overall we conclude that kernel PDE-greedy collocation algorithms are
useful approaches for approximation of linear PDE problems which admit sufficiently smooth solutions. 
In this scope the method's benefits are the use of meshless points and the ease of implementation
in a dimension-independent fashion. 
This renders those methods particularly useful in high-dimensional cases, 
where traditional mesh-based schemes such as finite volumes/differences/elements can
not easily be applied due to the increasingly complex grid management and curse of dimensionality. 
This curse of dimensionality is provably overcome with the kernel PDE-greedy approaches.
Clear limitations can, however, be seen in approximation of target functions that have singularities.

We want to comment on some directions for future research: 
An extension to nonlinear PDE problems seems possible by using the presented scheme for the linearized problems within corresponding fix-point schemes of nonlinear solvers.
Further work may as well consider ``escaping the native space'', i.e. problems
where the solution of the PDE problem is not in the RKHS of the chosen
kernel.  
We expect to be able to derive convergence orders, which are however reduced in accordance with the limited smoothness of the solution \cite{narcowich2006sobolev,wenzel2024sharp}.
Future work will also address the analysis of greedy algorithms for other methods of solving PDEs via kernel methods, 
such as unsymmetric kernel collocation or RBF-FD methods.
Further open questions include the optimality of the proven decay rates as well as approximation in different norms and work on analytic kernels such as the Gaussian kernel.

%% file: chapters/app_01_proofs.tex
\section{Proofs for Section \ref{sec:analysis_kolm_width}} \label{sec:proofs0}

\begin{proof}[Proof of \Cref{prop:upper_bound}]
We calculate:
\begin{align*}
d_n(\Lambda) &\equiv \inf_{\substack{G_n \subset \calh' \\ \dim(G_n) = n}} \sup_{\mu \in \Lambda} \dist(\mu, G_n)_{\calh'} \\
&= \inf_{\substack{G_n \subset \calh' \\ \dim(G_n) = n}} \max_{j=1,...,M} \left( \sup_{\mu \in \Lambda_j} \dist(\mu, G_n)_{\calh'} \right) \\
&\leq \inf_{\substack{G_{n_j} \subset \calh' \\ \dim(G_{n_j}) =n_j \\ \sum_{j=1}^M n_j \leq n}} \max_{j=1,...,M} \left( \sup_{\mu \in \Lambda_j} \dist(\mu, G_{n_j})_{\calh'} \right) \\
&\leq \min_{\sum_{j=1}^M n_j \leq n} \inf_{\substack{ G_{n_1} \subset \calh' \\  \dim(G_{n_1}) = n_1 }} ... \inf_{\substack{ G_{n_M} \subset \calh' \\  \dim(G_{n_M}) = n_M }} \max_{j=1,...,M} \left( \sup_{\mu \in \Lambda_j} \dist(\mu, G_{n_j})_{\calh'} \right) \\
&= \min_{\sum_{j=1}^M n_j \leq n} \max_{j=1,...,M} \left( \inf_{\substack{G_{n_j}  \subset \calh' \\ \dim(G_{n_j})=n_j}} \sup_{\mu \in \Lambda_j} \dist(\mu, G_{n_j})_{\calh'} \right) \\
&= \min_{\sum_{j=1}^M n_j \leq n} \max_{j=1,...,M} \left( d_{n_j}(\Lambda_j) \right).
\end{align*}
\end{proof}

\begin{proof}[Proof of \Cref{th:kolmogorov_estimate_Lambda1}]
Based on Eq.\ \eqref{eq:kolmogorov_dual} and Eq.\ \eqref{eq:power_func_generalized} we have the following representation of the Kolmogorov $n$-width $d_n(\Lambda_L)$ (using the Riesz representer $v_\mu$ of a functional $\mu$):
\begin{align}
\begin{aligned}\label{eq:kolmogorov_proof}
d_n(\Lambda_L, \ns') &= \inf_{\substack{H_n \subset \ns \\ \dim(H_n) = n}} \sup_{\mu \in \Lambda_L} \Vert v_\mu - \Pi_{H_n}(v_\mu) \Vert_{\ns} \\
&= \inf_{\substack{H_n \subset \ns \\ \dim(H_n) = n}} \sup_{\mu \in \Lambda_L} \sup_{0 \neq u \in \ns} \frac{|\mu(u - \Pi_{H_n}(u))|}{\Vert u \Vert_{\ns}} \\
&= \inf_{\substack{H_n \subset \ns \\ \dim(H_n) = n}} \sup_{x \in \Omega} \sup_{0 \neq u \in \ns} \frac{|L(u - \Pi_{H_n}(u))(x)|}{\Vert u \Vert_{\ns}} \\
&= \inf_{\substack{H_n \subset \ns \\ \dim(H_n) = n}} \sup_{0 \neq u \in \ns} \frac{\Vert L(u - \Pi_{H_n}(u)) \Vert_{L^\infty(\Omega)}}{\Vert u \Vert_{\ns}} \\
&\leq \inf_{\substack{H_n \subset \ns \\ H_n = \Sp \{ v_{\lambda_1}, .., v_{\lambda_n} \} \\ \lambda_1, .., \lambda_n \in \Lambda_L}} \sup_{0 \neq u \in \ns} \frac{\Vert L(u - \Pi_{H_n}(u)) \Vert_{L^\infty(\Omega)}}{\Vert u \Vert_{\ns}}.
\end{aligned}
\end{align}
For the third equality $\Lambda_L = \{ \delta_x \circ L ~ | ~ x \in \Omega \}$ was used and in the final step the infimum was upper bounded by the infimum computed over spaces spanned by any $n$ elements of $\Lambda_L$.

Since $u, \Pi_{H_n}(u) \in \ns \asymp H^\tau(\Omega)$ and the differential operator $L$ is of degree $2$, it follows that $Lu, L\Pi_{H_n}(u) \in H^{\tau - 2}(\Omega)$. 
Recall that $x_\lambda \in \Omega \cup \partial\Omega$ is the point related to the functional $\lambda \in \Lambda_L \cup \Lambda_B \subset \ns'$, 
i.e.\ where the evaluation takes place within $\Omega$ (for evaluations of $\lambda \in \Lambda_L$) or on its boundary $\partial \Omega$ (for evaluations of the boundary functionals $\lambda \in \Lambda_B$).
Then (because of the construction of $L\Pi_{H_n}(u)$) it holds 
\begin{align*}
(L\Pi_{H_n}u)(x_{\lambda_i}) &= (Lu)(x_{\lambda_i}) &&\forall i = 1, \dots, n \\
\Leftrightarrow (Lu - L\Pi_{H_n}u)(x_{\lambda_i}) &= 0 &&\forall i = 1, \dots, n,
\end{align*}
i.e.\ $Lu - L\Pi_{H_n}(u) \in H^{\tau - 2}(\Omega)$ is a function which is zero at least on the set $X_n := \{x_{\lambda_i} \}_{i=1}^n$. 
Therefore the sampling inequalities from Theorem \ref{th:estimate_derivatives_wendland} can be applied, provided $h\equiv h_{X_n,\Omega} \leq h_0$ (using $m = 0, q = \infty, p = 2, \tau - 2$ instead of $\tau$ and $Lu - L\Pi_{H_n}(u)$ instead of $u$)
\begin{align*}
\Vert Lu - L\Pi_{H_n}(u) \Vert_{L^\infty(\Omega)} &\leq C h^{\tau-2-d/2} \cdot |Lu - L\Pi_{H_n}(u)|_{H^{\tau-2}(\Omega)}.
\end{align*}
To bound the infimum over all $H_n \subset \ns$ that are induced by $\Lambda_L$-functionals, one can now consider functionals related to asymptotically uniformly distributed points. 
For such a choice of points the fill distance $h_{X_n,\Omega}$ decays according to $h_{X_n,\Omega} \leq c n^{-1/d}$, such that we obtain
\begin{align*}
\Vert Lu - L\Pi_{H_n}(u) \Vert_{L^\infty(\Omega)} &\leq C n^{\frac{1}{2} - \frac{\tau-2}{d}} \cdot |Lu - L\Pi_{H_n}(u)|_{H^{\tau-2}(\Omega)} \\
&\leq \tilde{C} n^{\frac{1}{2} - \frac{\tau-2}{d}} \cdot \Vert Lu - L\Pi_{H_n}(u) \Vert_{H^{\tau-2}(\Omega)}.
\end{align*}
This result can be plugged into the (in)equality chain above for the Kolmogorov $n$-width $d_n(\Lambda_L)$, such that we obtain

\begin{align*}
d_n(\Lambda_L, \ns') &\leq C n^{\frac{1}{2} - \frac{\tau-2}{d}} \cdot \inf_{\substack{H_n \subset \ns \\ \dim(H_n) = n}} \sup_{0 \neq u \in \ns} \frac{\Vert Lu - L\Pi_{H_n}(u)\Vert_{H^{\tau-2}(\Omega)}}{\Vert u \Vert_{\ns}} \\
&\leq C' n^{\frac{1}{2} - \frac{\tau-2}{d}} \cdot \inf_{\substack{H_n \subset \ns \\ \dim(H_n) = n}} \sup_{0 \neq u \in \ns} \frac{\Vert u - \Pi_{H_n}(u)\Vert_{H^{\tau}(\Omega)}}{\Vert u \Vert_{\ns}} \\
&\leq C'' n^{\frac{1}{2} - \frac{\tau-2}{d}} \cdot \inf_{\substack{H_n \subset \ns \\ \dim(H_n) = n}} \sup_{0 \neq u \in \ns} \frac{\Vert u - \Pi_{H_n}(u)\Vert_{\ns}}{\Vert u \Vert_{\ns}} \\
&\leq C'' n^{\frac{1}{2} - \frac{\tau-2}{d}} \cdot \inf_{\substack{H_n \subset \ns \\ \dim(H_n) = n}} \sup_{0 \neq u \in \ns} \frac{\Vert u \Vert_{\ns}}{\Vert u \Vert_{\ns}} \\
&\leq C'' n^{\frac{1}{2} - \frac{\tau-2}{d}}.
\end{align*}
The second inequality is due to \cite[Chapter 1: Lemma 7.2+Remark 8.1]{lions2012non} and the triangle inequality.
\end{proof}

\begin{proof}[Proof of \Cref{th:kolmogorov_estimate_Lambda2}]

The proof follows the same strategy as the one of Theorem \ref{th:kolmogorov_estimate_Lambda1}.
We start by decomposing the boundary $\partial \Omega$ into smooth manifolds $\mathbb{M}_i$, i.e.\ $\partial \Omega = \cup_{i=1}^N \mathbb{M}_i$.

Based on Eq.\ \eqref{eq:kolmogorov_dual} and Eq.\ \eqref{eq:power_func_generalized} we have analogously to the calculation in Eq.\ \eqref{eq:kolmogorov_proof} the following representation of the Kolmogorov $n$-width $d_n(\Lambda_B)$ (using the Riesz representer $v_\mu$ of a functional $\mu$):
\begin{align}
\label{eq:kolm_width_boundary}
d_n(\Lambda_B, \ns') &= \inf_{\substack{H_n \subset \ns \\ \dim(H_n) = n}} \sup_{\mu \in \Lambda_B} \Vert v_\mu - \Pi_{H_n}(v_\mu) \Vert_{\ns} \notag \\
&\leq \inf_{\substack{H_n \subset \ns \\ H_n = \Sp \{ v_{\lambda_1}, .., v_{\lambda_n} \} \\ \lambda_1, .., \lambda_n \in \Lambda_B}} \sup_{0 \neq u \in \ns} \frac{\Vert u - \Pi_{H_n}(u) \Vert_{L^\infty(\partial \Omega)}}{\Vert u \Vert_{\ns}}.
\end{align}
As all involved function are continuous, we waived to include the trace operator.
In order to bound $\Vert u - \Pi_{H_n}(u) \Vert_{L^\infty(\partial \Omega)} = \max_{i=1, ..., N} \Vert u - \Pi_{H_n}(u) \Vert_{L^\infty(\mathbb{M}_i)}$, 
we leverage standard restriction theorems of RKHS to conclude that $u - \Pi_{H_n}(u)|_{\mathbb{M}_i}$ $\in$ $\mathcal{H}_k(\mathbb{M}_i) \asymp H^{\tau - 1/2}(\mathbb{M}_i)$.
We recall that the functionals $\lambda_1, \dots, \lambda_n$ are associated to the interpolation points $x_{\lambda_1}, ...$, $x_{\lambda_n}$, i.e.\ $u - \Pi_{H_n}(u)$ is a function with zeros in these interpolation points $x_{\lambda_1}, ...$, $x_{\lambda_n}$. 
Therefore the sampling inequalities from Theorem \ref{th:sampling_inequ_manifolds} can be applied to any smooth manifold $\mathbb{M}_i$, $i=1, ..., N$,
provided $h_{X_n, \mathbb{M}_i} \leq C_{\mathbb{M}_i}$ (using $k = d-1, \mu = 0, q = \infty, p = 2, t = \tau - 1/2$ and $u - \Pi_{H_n}(u)$ instead of $u$, which suffice the assumptions):
\begin{align}
\label{eq:appl_sampling_manifold}
\Vert u - \Pi_{H_n}(u) \Vert_{L^\infty(\mathbb{M}_i)} &\leq C_i h_{X_n \cap \mathbb{M}_i, \mathbb{M}_i}^{\tau-1/2-(d-1)/2} \cdot |u - \Pi_{H_n}(u)|_{W_2^{\tau-1/2}(\mathbb{M}_i)},
\end{align}
where $h_{X_n \cap \mathbb{M}_i, \mathbb{M}_i}$ is the fill distance related to $\{x_{\lambda_1}, ..., x_{\lambda_n} \} \cap \mathbb{M}_i \subset \mathbb{M}_i \subset \partial \Omega$. 
Observing that $\mathcal{H}_k(\mathbb{M}_i)$ is norm-equivalent to $H^{\tau - 1/2}(\mathbb{M}_i)$ for all $i = 1, ..., N$ (see e.g.\ \cite{fuselier2012scattered}),
we can estimate the semi-norm $| u - \Pi_{H_n}(u) |_{W_2^{\tau-1/2}(\mathbb{M}_i)}$ as follows:
\begin{align*}
|u - \Pi_{H_n}(u)|_{W_2^{\tau-1/2}(\mathbb{M}_i)} &\leq \Vert u - \Pi_{H_n}(u) \Vert_{W_2^{\tau-1/2}(\mathbb{M}_i)} \\
&\leq C_i' \Vert u - \Pi_{H_n}(u) \Vert_{\mathcal{H}_k(\mathbb{M}_i)} \\
&\leq C_i' \Vert u - \Pi_{H_n}(u) \Vert_{\mathcal{H}_k(\Omega)} \\
&\leq C_i' \Vert u \Vert_{\mathcal{H}_k(\Omega)}
\end{align*}
The infimum within Eq.~\eqref{eq:kolm_width_boundary} over all $H_n \subset \calh$ that are induced by $\Lambda_B$-functionals can be upper bounded by considering asymptotically uniformly distributed (wrt.\ the intrinsic distance on the manifold $\partial \Omega$) points $\{x_{\lambda_1}, ..., x_{\lambda_n} \} \subset \partial \Omega$. 
For such a choice of points the fill distance $h_{X_n \cap \mathbb{M}_i, \mathbb{M}}$ decays according to $h_{X_n, \mathbb{M}} \leq c n^{-1/(d-1)}$ (where the constant $c$ also depends on the number $N$ of manifolds $\mathbb{M}_i$), 
such that Eq.~\eqref{eq:appl_sampling_manifold} turns into (due to $-\frac{1}{d-1}(\tau - 1/2 - (d-1)/2) = \frac{1}{2} - \frac{\tau - 1/2}{d-1}$):
\begin{align*}
\Vert u - \Pi_{H_n}(u) \Vert_{L^\infty(\mathbb{M}_i)} 
&\leq C_i'' n^{\frac{1}{2} - \frac{\tau-1/2}{d-1}} \cdot |u - \Pi_{H_n}(u)|_{W_2^{\tau-1/2}(\partial \Omega)} \\
&\leq \tilde{C}_i n^{\frac{1}{2} - \frac{\tau-1/2}{d-1}} \cdot \Vert u \Vert_{\ns} \\
\Rightarrow \max_{i=1, ..., N} \Vert u - \Pi_{H_n}(u) \Vert_{L^\infty(\mathbb{M}_i)} &\leq \left( \max_{i=1, ..., N} \tilde{C}_i \right) \cdot n^{\frac{1}{2} - \frac{\tau-1/2}{d-1}} \cdot \Vert u \Vert_{\ns}.
\end{align*}
This result can be plugged into the (in-)equality chain for the Kolmogorov $n$-width $d_n(\Lambda_B)$ of Eq.~\eqref{eq:kolm_width_boundary},
such that we obtain directly
\begin{align*}
d_n(\Lambda_B, \ns) 
&\leq \left( \max_{i=1, ..., N} \tilde{C}_i \right) \cdot n^{\frac{1}{2} - \frac{\tau-1/2}{d-1}}.
\end{align*}
\end{proof}

\begin{proof}[Proof of \Cref{th:kolmogorov_estimate_Lambda}]
For this proof we introduce $\alpha := -(\frac{1}{2} - \frac{\tau - 2}{d}) > 0$ and $\beta := -(\frac{1}{2} - \frac{\tau - 1/2}{d-1}) > 0$, i.e.\ it holds $\beta > \alpha > 0$.
We choose
\begin{align*}
n_2(n) &= \lceil n^{\alpha / \beta} \rceil \\
n_1(n) &= \lfloor n - n^{\alpha / \beta} \rfloor = \lfloor n \cdot (1 - n^{\alpha/\beta - 1}) \rfloor. 
\end{align*}
For $n \geq 2$ we have $0 < 1 - 2^{\alpha/\beta - 1} \leq 1 - n^{\alpha / \beta - 1} \leq 1$, therefore we have $n_1(n) \asymp n$. 
In more details, there are constants $c_1, c_2>0$ such that $c_1 n \leq n_1(n) \leq c_2 n$ for  $n\geq 2$, $c_1, c_2$ are independent of 
$n$, and $c_1<1$.

Thus, using \Cref{prop:upper_bound}, \Cref{th:kolmogorov_estimate_Lambda1} and \Cref{th:kolmogorov_estimate_Lambda2} gives
\begin{align*} 
d_n(\Lambda) &\leq \inf_{n_1 + n_2 \leq n} \max \left( d_{n_1}(\Lambda_L), d_{n_2}(\Lambda_B) \right) \notag \\
&\leq \max(C_L, C_B) \cdot \inf_{n_1 + n_2 \leq n} \max \left( n_1^{-\alpha}, n_2^{-\beta} \right) \notag \\
&\leq \max(C_L, C_B) \cdot \max \left( n_1(n)^{-\alpha}, n_2(n)^{-\beta} \right). \notag
\end{align*}
A tiny extra calculation helps to resolve the expressions $n_1(n)^{-\alpha}$  and $n_2(n)^{-\beta}$:
\begin{align*}
n_1(n)^{-\alpha} &\leq  (c_1 n)^{-\alpha} = c_1^{-\alpha} n^{-\alpha} \\
n_2(n)^{-\beta} &\leq  \lceil n^{\alpha / \beta} \rceil^{-\beta} \leq (n^{\alpha / \beta})^{-\beta} = n^{-\alpha}.
\end{align*}
Therefore we finally obtain
\begin{align*}
d_n(\Lambda) 
&\leq \max(C_{L}, C_{B}) \cdot \max(c_1^{-\alpha}, 1) \cdot n^{-\alpha}\\
&\leq \max(C_{L}, C_{B}) \cdot c_1^{-\alpha} \cdot n^{-\alpha}.
\end{align*}
Recalling the abbreviation $-\alpha = \frac{1}{2} - \frac{\tau - 2}{d}$ we obtain the statement.
\end{proof}

\section{Proofs for Section \ref{sec:analysis_fgreedy}} \label{sec:proofs}

\begin{proof}[Proof of Lemma \ref{lem:estimate_product}]
Let 
\begin{equation*}
R_n^2:= \left[ \prod_{i=n+1}^{2n} \left( \frac{\lambda_{i+1}(e_i)}{P_{\Lambda_i}(\lambda_{i+1})} \right)^2 \right]^{1/n}.
\end{equation*}
The geometric arithmetic mean inequality gives
\begin{align*} 
R_n^2 
&\leq \frac{1}{n} \sum_{i=n+1}^{2n} \left(\frac{\lambda_{i+1}(e_i)}{P_{\Lambda_i}(\lambda_{i+1})} \right)^2
= \frac{1}{n} \left( \sum_{i=0}^{2n} \left( \frac{\lambda_{i+1}(e_i)}{P_{\Lambda_{i}}(\lambda_{i+1})} \right)^2 
- \sum_{i=0}^{n} \left(\frac{\lambda_{i+1}(e_i)}{P_{\Lambda_i}(\lambda_{i+1})} \right)^2 \right).
\end{align*}
We now use Eq.\ \eqref{eq:newtown_coefficient} applied to $s_{2n+1}$ and $s_{n+1}$, and the properties of orthogonal projections to obtain
\begin{align*} 
R_n^2
&\leq \frac{1}{n} \left( \Vert s_{2n+1} \Vert_{\ns}^2 - \Vert s_{n+1}\Vert_{\ns}^2 \right)
\leq \frac{1}{n}  \left( \Vert f \Vert_{\ns}^2 - \Vert s_{n+1} \Vert_{\ns}^2 \right)\\
&= \frac{1}{n}  \Vert f - s_{n+1} \Vert_{\ns}^2
= \frac{1}{n}  \Vert e_{n+1} \Vert_{\ns}^2.
\end{align*}
It follows that $R_n \leq n^{-1/2} \cdot \Vert e_{n+1} \Vert_{\ns}$, and thus
\begin{equation*}
\left[ \prod_{i=n+1}^{2n} |\lambda_{i+1}(e_i)| \right]^{1/n} \leq n^{-1/2} \cdot \Vert e_{n+1} \Vert_{\ns} \cdot \left[ \prod_{i=n+1}^{2n} 
P_{\Lambda_i}(\lambda_{i+1}) \right]^{1/n}.
\end{equation*}
\end{proof}

\begin{proof}[Proof of Lemma \ref{lem:beta_greedy}]
We prove the two cases separately:
\begin{enumerate}[label={\alph*)}]
\item For $\beta = 0$, i.e.\ the PDE-$P$-greedy algorithm, this is the standard power function estimate in conjunction with the PDE-$P$-greedy selection criterion 
$P_{\Lambda_n}(\lambda_{n+1}) = \sup_{\lambda \in \Lambda} P_{\Lambda_n}(\lambda)$. For $\beta = 1$ this holds with equality as it is simply the selection criterion 
of PDE-$f$-greedy since we have here $\lambda_{n+1}(e_n) = \sup_{\lambda \in \Lambda} | r_n |$. We thus consider $\beta \in (0, 1)$ and let $\tilde{\lambda}_{i+1} \in 
\Lambda$ be such that $|\tilde{\lambda}_{i+1}(e_i)| = \sup_{\lambda \in \Lambda} | \lambda(e_i) |$. Then the selection criterion from 
Eq.~\eqref{eq:beta_greedy_selection_criterion} gives
\begin{equation*}
|\lambda(e_i)|^\beta \cdot P_{\Lambda_i}(\lambda)^{1-\beta} \leq |\lambda_{i+1}(e_i)|^\beta \cdot P_{\Lambda_i}(\lambda_{i+1})^{1-\beta}\;\; \forall \lambda \in \Lambda,
\end{equation*}
and in particular
\begin{align*}
P_{\Lambda_i}(\tilde{\lambda}_{i+1}) \leq \frac{|\lambda_{i+1}(e_i)|^{\frac{\beta}{1-\beta}}}{|\tilde{\lambda}_{i+1}(e_i)|^{\frac{\beta}{1-\beta}}} \cdot P_{\Lambda_i}(\lambda_{i+1}).
\end{align*}
Using this bound with the standard power function estimate gives
\begin{align*}
\sup_{\lambda \in \Lambda} | \lambda(e_i) | &= |\tilde{\lambda}_{i+1}(e_i)| \leq P_{\Lambda_i}(\tilde{\lambda}_{i+1}) \cdot \Vert f - s_i \Vert_{\ns} \\
&\leq \frac{|\lambda_{i+1}(e_i)|^{\frac{\beta}{1-\beta}}}{|\tilde{\lambda}_{i+1}(e_i)|^{\frac{\beta}{1-\beta}}} \cdot P_{\Lambda_i}(\lambda_{i+1}) \cdot \Vert f - s_i \Vert_{\ns} \\
&= \frac{|\lambda_{i+1}(e_i)|^{\frac{\beta}{1-\beta}}}{\sup_{\lambda \in \Lambda} | \lambda(e_i) |^{\frac{\beta}{1-\beta}}} \cdot P_{\Lambda_i}(\lambda_{i+1}) \cdot \Vert f - s_i \Vert_{\ns}.
\end{align*}
This can be rearranged for $\sup_{\lambda \in \Lambda} | \lambda(e_i) |$ to yield the final result.
\item For $\beta \in (1, \infty)$, the selection criterion from Eq.~\eqref{eq:beta_greedy_selection_criterion} can be rearranged to
\begin{align*}
|\lambda(e_i)|^\beta &\leq \frac{|\lambda_{i+1}(e_i)|^\beta}{P_{\Lambda_i}(\lambda_{i+1})^{\beta - 1}} \cdot P_{\Lambda_i}(\lambda)^{\beta - 1} \qquad \forall {\lambda \in \Lambda \setminus \Lambda_i},
\end{align*}
and taking the supremum $\sup_{\lambda \in \Lambda \setminus \Lambda_i}$ gives
\begin{align*}
\sup_{\lambda \in \Lambda} | \lambda(e_i) | &\leq \frac{|\lambda_{i+1}(e_i)|}{P_{\Lambda_i}(\lambda_{i+1})^{\frac{\beta - 1}{\beta}}} \cdot \sup_{\lambda \in \Lambda} P_{\Lambda_i}(\lambda)^{\frac{\beta - 1}{\beta}} \qquad \forall {\lambda \in \Lambda \setminus \Lambda_i}.
\end{align*}
For $\beta = \infty$, the selection criterion of the PDE-$f/P$-greedy algorithm can be directly rearranged to yield the statement (when using the notation 
$1/\infty = 0$).
\end{enumerate}
\end{proof}

\begin{proof}[Proof of Theorem \ref{th:final_result}]
We prove the two cases separately:
\begin{enumerate}[label={\alph*)}]
\item For $\beta = 0$, i.e. PDE-$P$-greedy, Eq.~\eqref{eq:bound_ri_first} gives $\sup_{\lambda \in \Lambda} |\lambda(e_i)| \leq P_{\Lambda_i}(\lambda_{i+1}) \cdot \Vert e_i \Vert_{\ns}$. 
Taking the product $\prod_{i=n+1}^{2n}$ and the $n$-th root in conjunction with the estimate $\Vert e_i \Vert_{\ns} \leq \Vert e_{n+1} \Vert_{\ns}$ for $i = 
n+1, \dots, 2n$ gives the result. 

For $\beta \in (0, 1]$, we start by reorganizing the estimate \eqref{eq:bound_ri_first} of Lemma \ref{lem:beta_greedy} to get
\begin{equation*}
|\lambda_{i+1}(e_i)| \geq \left( \sup_{\lambda \in \Lambda} \lambda(e_i)^{1/\beta} \right) / \left(P_{\Lambda_i}(\lambda_{i+1})^{\frac{1-\beta}{\beta}} \cdot 
\Vert e_i \Vert_{\ns}^{\frac{1-\beta}{\beta}}\right),
\end{equation*}
and we use this to bound the left hand side of Eq.~\eqref{eq:estimate_product} as
\begin{align*} 
n^{-1/2} \cdot &\Vert e_{n+1} \Vert_{\ns} \cdot \left[ \prod_{i=n+1}^{2n} P_{\Lambda_i}(\lambda_{i+1}) \right]^{1/n} \geq \left[ \prod_{i=n+1}^{2n} |\lambda_{i+1}(e_i)| 
\right]^{1/n} \\
&\geq \left[ \prod_{i=n+1}^{2n} \left( \sup_{\lambda \in \Lambda} \lambda(e_i)^{1/\beta} \right) / \left(P_{\Lambda_i}(\lambda_{i+1})^{\frac{1-\beta}{\beta}} \cdot 
\Vert e_i \Vert_{\ns}^{\frac{1-\beta}{\beta}}\right) \right]^{1/n} \\
&= \left[ \prod_{i=n+1}^{2n} \left( \sup_{\lambda \in \Lambda} \lambda(e_i)^{1/\beta} \right) \right]^{1/n} \left[ \prod_{i=n+1}^{2n} P_{\Lambda_i}(\lambda_{i+1})^{\frac{1-\beta}{\beta}} \cdot 
\Vert e_i \Vert_{\ns}^{\frac{1-\beta}{\beta}} \right]^{-1/n}.
\end{align*}
Rearranging the factors, and using again the fact that $\Vert e_i \Vert_{\ns} \leq \Vert e_{n+1} \Vert_{\ns}$ for $i = 
n+1, \dots, 2n$, gives
\begin{align*}
&\left[ \prod_{i=n+1}^{2n} \left( \sup_{\lambda \in \Lambda} \lambda(e_i)^{1/\beta} \right) \right]^{1/n} \\
&\leq n^{-1/2} \cdot \Vert e_{n+1} \Vert_{\ns} \cdot \left[\prod_{i=n+1}^{2n} P_{\Lambda_i}(\lambda_{i+1})^{1/\beta} \right]^{1/n} \cdot \left[ \prod_{i=n+1}^{2n} \Vert e_i 
\Vert_{\ns}^{\frac{1-\beta}{\beta}}  \right]^{1/n} \\
&\leq n^{-1/2} \cdot \Vert e_{n+1} \Vert_{\ns} \cdot \left[ \prod_{i=n+1}^{2n} P_{\Lambda_i}(\lambda_{i+1})^{1/\beta} \right]^{1/n} \cdot \Vert e_{n+1} 
\Vert_{\ns}^{\frac{1-\beta}{\beta}} \\
&\leq n^{-1/2} \cdot \Vert e_{n+1} \Vert_{\ns}^{1/\beta} \cdot \left[ \prod_{i=n+1}^{2n} P_{\Lambda_i}(\lambda_{i+1})^{1/\beta} \right]^{1/n}.
\end{align*}
Now, the inequality can be raised to the exponent $\beta$ to give the final statement.

\item For $\beta \in (1, \infty]$ we proceed similarly by first rewriting Eq.~\eqref{eq:bound_ri_second} of Lemma \ref{lem:beta_greedy} as
\begin{equation*}
|e_i(x_{i+1})| \geq \left(\Vert e_i \Vert_{L^\infty(\Omega)} \cdot P_i(x_{i+1})^{1-1/\beta}\right)/\left(\Vert P_i \Vert_{L^\infty(\Omega)}^{1-1/\beta}\right),
\end{equation*}
and we lower bound the left hand side of Equation \eqref{eq:estimate_product} as
\begin{align*} 
n^{-1/2} \cdot &\sup_{\lambda \in \Lambda} | e_{n+1} | \cdot \left[ \prod_{i=n+1}^{2n} P_{\Lambda_i}(\lambda_{i+1}) \right]^{1/n} \geq \left[ \prod_{i=n+1}^{2n} |\lambda_{i+1}(e_i)| 
\right]^{1/n} \\
&\geq \left[ \prod_{i=n+1}^{2n} \left( \sup_{\lambda \in \Lambda} | e_i | \cdot P_{\Lambda_i}(\lambda_{i+1})^{1-1/\beta}\right)/\left( \sup_{\lambda \in \Lambda} P_{\Lambda_i}(\lambda)^{1-1/\beta}\right)\right]^{1/n}.
\end{align*}
Rearranging for $\left[ \prod_{i=n+1}^{2n} \sup_{\lambda \in \Lambda} | \lambda(e_i) | \right]^{1/n}$ yields
\begin{align*}
&\left[ \prod_{i=n+1}^{2n} \sup_{\lambda \in \Lambda} |\lambda(e_i) | \right]^{1/n} \\
&\leq n^{-1/2} \cdot \Vert e_{n+1} \Vert_{\ns} \cdot \left[\prod_{i=n+1}^{2n} \sup_{\lambda \in \Lambda} P_{\Lambda_i}(\lambda)^{1-1/\beta} \right]^{1/n} \cdot \left[ 
\prod_{i=n+1}^{2n} P_{\Lambda_i}(\lambda_{i+1})^{1/\beta} \right]^{1/n},
\end{align*}
which gives the final result due to $\sup_{\lambda \in \Lambda} P_i(\lambda) \leq 1$ for all $i = 0, 1, ..$.
\end{enumerate}
\end{proof}